\documentclass[oneside,11pt,reqno]{amsart}
 
\usepackage{amsmath}

\usepackage{amsfonts}
\usepackage{amssymb}
\usepackage{amsxtra}
\usepackage{tikz-cd}
\usepackage{hyperref}
\usepackage{enumitem}
\usepackage{amsthm}
\usepackage[colorinlistoftodos,textwidth=1in]{todonotes}
\usepackage{tabularray}

\DefTblrTemplate{caption}{default}{}    
\DefTblrTemplate{capcont}{default}{}    
\DefTblrTemplate{contfoot}{default}{}   

\usepackage{amsmath}
\usepackage{amssymb}
\usepackage{mathtools}
\usepackage{graphicx, overpic}
\usepackage{caption}
\usepackage{subcaption}
\usepackage{pinlabel}
\usepackage{floatrow}
\usepackage{soul}
\usepackage{tikz}
\usetikzlibrary{decorations.pathreplacing}
\usetikzlibrary{decorations.markings}
\usetikzlibrary{shapes.geometric}
\usetikzlibrary{calc}
\usepackage{color}
\usetikzlibrary{shapes.gates.logic.US,trees,positioning,arrows}
\usetikzlibrary{knots}

\newcommand*\circled[1]{\tikz[baseline=(char.base)]{
            \node[shape=circle,draw,inner sep=2pt] (char) {#1};}}

\usepackage{float}
 
\usepackage[all]{xy}
\usepackage{multicol}

\usepackage{mathtools}
 
\usepackage{xcolor}

\definecolor{flaggreen}{RGB}{4,106,56}
\definecolor{flagnavy}{RGB}{6,3,141}
\definecolor{flagsaffron}{RGB}{255,103,31}
\usepackage[dvipsnames]{xcolor}

\makeatletter
\DeclareFontFamily{OMX}{MnSymbolE}{}
\DeclareSymbolFont{MnLargeSymbols}{OMX}{MnSymbolE}{m}{n}
\SetSymbolFont{MnLargeSymbols}{bold}{OMX}{MnSymbolE}{b}{n}
\DeclareFontShape{OMX}{MnSymbolE}{m}{n}{
    <-6>  MnSymbolE5
   <6-7>  MnSymbolE6
   <7-8>  MnSymbolE7
   <8-9>  MnSymbolE8
   <9-10> MnSymbolE9
  <10-12> MnSymbolE10
  <12->   MnSymbolE12
}{}
\DeclareFontShape{OMX}{MnSymbolE}{b}{n}{
    <-6>  MnSymbolE-Bold5
   <6-7>  MnSymbolE-Bold6
   <7-8>  MnSymbolE-Bold7
   <8-9>  MnSymbolE-Bold8
   <9-10> MnSymbolE-Bold9
  <10-12> MnSymbolE-Bold10
  <12->   MnSymbolE-Bold12
}{}

\let\llangle\@undefined
\let\rrangle\@undefined
\DeclareMathDelimiter{\llangle}{\mathopen}%
                     {MnLargeSymbols}{'164}{MnLargeSymbols}{'164}
\DeclareMathDelimiter{\rrangle}{\mathclose}%
                     {MnLargeSymbols}{'171}{MnLargeSymbols}{'171}
\makeatother

\newtheorem{thm}{Theorem}[section]
\newtheorem{prop}[thm]{Proposition}
\newtheorem{lem}[thm]{Lemma}
\newtheorem{cor}[thm]{Corollary}

\theoremstyle{definition}
\newtheorem{defn}[thm]{Definition}
\newtheorem{rem}[thm]{Remark}

\newtheorem{exmp}[thm]{Example}

\newtheorem*{prop-MC}{Proposition 4.28}

\renewcommand{\bar}[1]{\overline{#1}}

\renewcommand{\emptyset}{\varnothing}

\newcommand{\field}[1]{\mathbb{#1}}
\newcommand{\Z}{\field{Z}}

\newcommand{\R}{\field{R}}

\newcommand{\N}{\field{N}}

\renewcommand*\circled[1]{\tikz[baseline=(char.base)]{
            \node[shape=circle,draw,inner sep=2pt] (char) {#1};}}

\DeclareMathOperator{\Area}{Area}
\DeclareMathOperator{\Ker}{Ker}

\DeclareMathOperator{\dehnleq}{\preccurlyeq}

\newcommand{\laproj}{{\rm pr}^{\rm L}}
\newcommand{\grpproj}{{\rm pr}^{\rm G}}

\allowdisplaybreaks

\newcommand{\showcomments}{yes}

\usepackage{ifthen}
\newsavebox{\commentbox}
{\ifthenelse{\equal{\showcomments}{yes}}%
{\footnotemark
    \begin{lrbox}{\commentbox}
    \begin{minipage}[t]{1.25in}\raggedright\sffamily\tiny
    \footnotemark[\arabic{footnote}]}
{\begin{lrbox}{\commentbox}}}%
{\ifthenelse{\equal{\showcomments}{yes}}%
{\end{minipage}\end{lrbox}\marginpar{\usebox{\commentbox}}}
{\end{lrbox}}}

\makeatletter
\newtheorem*{rep@theorem}{\rep@title}
\newcommand{\newreptheorem}[2]{%
\newenvironment{rep#1}[1]{%
 \def\rep@title{#2 \ref{##1}}%
 \begin{rep@theorem}}%
 {\end{rep@theorem}}}
\makeatother

\newreptheorem{thm}{Theorem}

\begin{document}

\title[Isoperimetric behavior of generalized Stallings-Bieri groups]{Isoperimetric behavior of generalized Stallings-Bieri groups}

\begin{abstract}
We introduce the notion of $n$-split for an epimorphism from a group to a finite rank free abelian group. This is used to provide bounds for the Dehn functions of certain coabelian subgroups of direct products of finitely presented groups. Such subgroups include and significantly generalize the Stallings-Bieri groups. 
\end{abstract} 

\author{Noel Brady}
\address{University of Oklahoma, Norman, OK 73019-3103, USA}
\email{nbrady@ou.edu}

\author{Pratit Goswami}
\address{University of Oklahoma, Norman, OK 73019-3103, USA}
\email{pratit.goswami-1@ou.edu}

\author{Rob Merrell}
\address{University of Oklahoma, Norman, OK 73019-3103, USA}
\email{rob.merrell@ou.edu}

\maketitle


\section{Introduction}

The Dehn function is a fundamental quasi-isometry invariant of finitely presented groups. Although the geometric intuition behind the Dehn function can be found in the work of Dehn \cite{Dehn}, it was not until the early 1990s (\cite{Gromov_AI} and \cite{Gersten}) that it was defined as a quasi-isometry invariant. Since then, there has been an active research enterprise estimating bounds for Dehn functions of finitely presented groups.

Estimates for Dehn functions of finitely presented subgroups of various groups have been computed as a way of demonstrating the extent to which the large scale geometry of the subgroup differs from that of the parent group. Examples include finitely presented subgroups of the following classes of groups: automatic groups (\cite{MR1489138});  CAT(0) groups (\cite{MR3705143}, \cite{BT});  virtually special cubical groups (\cite{MR3956192}, \cite{MR1967746}); hyperbolic groups (\cite{MR1724853}, \cite{MR1934010}, \cite{MR3822289}, \cite{MR4278346}, \cite{MR4526820}, \cite{MR4773175}, \cite{MR4688705}); and direct product of free groups (\cite{dehnI}, \cite{dehnII}).

In other situations, it has been shown that the Dehn function of a finitely presented subgroup necessarily agrees with that of its parent group; examples of such situations are studied in \cite{MR2545243}, \cite{MR3651586}, and \cite{KLI}. This last work is the most relevant to our efforts, as it is the techniques of that article that we investigate.

In many of the above examples, the finitely presented subgroups appear as kernels of maps from the ambient group to $\mathbb{Z}$, giving a short exact sequence
    \[\begin{tikzcd}
    1 \arrow[r] & K \arrow[r] & G \arrow[r] & \Z \arrow[r]  & 1 \, .
    \end{tikzcd}\]
Dison showed in \cite{MR2418794} that in the case where $G$ is a right-angled Artin group, the Dehn function of $K$ is bounded above by a quartic function. Dison, Elder, Riley and Young showed in \cite{MR2545243} that Stallings' group has a quadratic Dehn function (in this case $G$ is a direct product of three copies of $F_2$ and the homomorphism takes all 6 generators of $G$ to a generator of $\mathbb{Z}$).  Carter and Forester showed in \cite{MR3651586} that the  Stallings-Bieri groups (where $G = (F_2)^n$ for $n \geq 3$) have quadratic Dehn functions. In \cite{KLI}, Kropholler and Llosa Isenrich abstracted and generalized the techniques of \cite{MR3651586} to include the case of sequences with $\mathbb{Z}^m$ as quotient. Our main result fits into the framework of this last article, and uses a condition called being $n$-split, which is an analogue of their $P$-split condition. 

\begin{repthm}{thm:main}
Let $n > 2$ be an integer. For each $1 \leq i \leq 2n$, let $G_i$ be a finitely presented group such that
    \[\begin{tikzcd}
    1 \arrow[r] & N_i \arrow[r] & G_i \arrow[r,"\varphi_i"] & \Z^m \arrow[r]  & 1
    \end{tikzcd}\]
is an $n$-split short exact sequence. Let $G = \prod_{i=1}^{2n} G_i$, and let $\Phi : G \to \Z^m$ be defined by $\Phi(g) = \sum_{i=1}^{2n} \varphi_i(\grpproj_i(g))$.

Then the kernel $\Ker(\Phi)$ in the short exact sequence
    \[\begin{tikzcd}
    1 \arrow[r] & \Ker(\Phi) \arrow[r] & G \arrow[r,"\Phi"] & \Z^m \arrow[r]  & 1
    \end{tikzcd}\]
is finitely presented and its Dehn function satisfies $\delta_G(n) \dehnleq \delta_{\Ker(\Phi)}(n) \dehnleq \overline{\delta_G}(n) \cdot \log(n)$. If, moreover, $\frac{\delta_G(n)}{n}$ is superadditive, then $\delta_G(n) \simeq \delta_{\Ker(\Phi)}(n)$.
\end{repthm}

Here $\overline{\delta_G}(n)$ denotes the superadditive closure of the Dehn function $\delta_G(n)$ of the group $G$, defined in Definition~\ref{def:superadd}.

The following table summarizes the differences and similarities between the analogous versions of Theorem~\ref{thm:main} in this paper and in \cite{KLI}. Note that although Theorem 4.2 in \cite{KLI} is only stated for a direct product of four groups, grouping some of the subproducts into pairs (the structure of which is forced by their $P$-splitting) permits the result to hold for arbitrary products of an even number of groups, and so we include this more general setup in the table.


\begin{longtblr}{
    colspec = {|c|c|c|c|},
    rowhead = 1, rows={abovesep=2pt,belowsep=2pt},
}
\hline
$n$-split     & $\Ker(\Phi)$                               & Geometry of filling  & Location \\
\hline \hline
$n = 1$               & $\Phi : \displaystyle\prod_{j=1}^3 G_j \to \Z^m$      & triangular                & \cite[Section 3]{KLI} \\
\hline
$n = 2$               & $\Phi : \displaystyle\prod_{j=1}^{2n} G_j \to \Z^m$   & square                    & \cite[Section 4]{KLI} \\
\hline
$n \geq 3$               & $\Phi : \displaystyle\prod_{j=1}^{2n} G_j \to \Z^m$   & triangular                & Section~\ref{sec:dehnfunction} \\
\hline
\end{longtblr}

We highlight several features of our construction.

\begin{itemize}
    \item The $n$-split condition is a generalization of the {\em $P$-split} condition of \cite{KLI} (which corresponds to being $2$-split in our sense). For $n \geq 3$, the $n$-split condition exploits different features of the lattice of subgroups of $\Ker(\Phi)$ than \cite{KLI}. This enables us to construct algebraic triangles which are used to build van Kampen diagrams. 
    
    These diagrams have  different fillings than the algebraic square fillings of \cite{KLI} in the cases where both techniques apply (for example, when $n \geq 4$ is even). 

    \item In our construction of the algebraic triangle, the regions of the triangle closely model the types of subgroups in the lattice of subgroups of $\Ker(\Phi)$. In particular, quadrilateral regions correspond to subgroups which have a direct product decomposition. Furthermore,  triangular and hexagonal regions correspond to subgroups which admit the following  $3$-fold ``symmetry";  they appear as kernels $\Ker(A^3 \to A)$ where $A$ is free-abelian and the homomorphism is $(a_1,a_2,a_3) \mapsto a_1+a_2+ a_3$.  

    \item Finally, the level of detail provided by the 25 region decomposition  of the  algebraic triangle in our construction, allows us to explicitly describe the $\Ker(\Phi)$ group elements associated to every vertex  (the description in 
    \cite{KLI} is more implicit than our description) and the edge paths associated to every edge of the algebraic triangle. In particular, we can distinguish between purely  {\em linear algebra operations} and other group operations inside $\Ker(\Phi)$. See proofs of Lemma~\ref{lem:triside-labdist} and Lemma~\ref{lem:triinner-labdist} for details. 
\end{itemize}
The proofs of the main results of \cite{KLI} and the current paper suggest investigating the combinatorics of the lattice of (particular families of) subgroups of a given finitely presented group group as an approach to understanding the Dehn function of the ambient group.  
In particular, it would be interesting to discover other situations where the combinatorics of a lattice of subgroups of a given group can be used to provide upper bounds for the Dehn function of the group. It is also tempting to think about lattices of subgroups with higher dimensional geometric realizations and possible applications to higher dimensional filling functions of groups. 

\section{Acknowledgements}
The first two authors acknowledge travel support from Simons Foundation collaboration grant \#430097. 
The second author acknowledges support from the 2025 Summer Research Fellowship at the University of Oklahoma. The third author also extends thanks to the University of Oklahoma for support via the 2024 Summer Research Fellowship.

\section{Background}
The main theorem of this paper concerns Dehn functions of certain finitely presented subgroups of direct products of groups. This section provides background on equivalence of functions, Dehn functions of groups, and the notion of superadditivity of functions.

\begin{defn}[Asymptotic equivalence]\label{def:equiv}
Let $\mathbb{N}$ denote the set of natural numbers. Given non-decreasing functions $f,g : \mathbb{N} \rightarrow \mathbb{N}$, we say that $f$ is \emph{dominated} by $g$, denoted by $f \preccurlyeq g$, if there is a constant $C > 0$ such that for all $n \in \mathbb{N}$, we have $f(n) \leq Cg(Cn+C) + Cn + C$. 

We say $f$ is \emph{asymptotically equivalent} to $g$, denoted $f \simeq g$, if both $f \preccurlyeq g$ and $g \preccurlyeq f$ (not necessarily for the same constant $C$). This defines an equivalence relation on the set of all non-decreasing functions from $\mathbb{N}$ to $\mathbb{N}$.
\end{defn}

\begin{defn}[Dehn function]
Let $G$ be a group given by a finite presentation $\mathcal{P} = \langle X \,|\, R \rangle$. For each word $w$ lying in the normal closure of $R$ in the free group $F(X)$, define
    \[
    \mathrm{Area}_{\mathcal{P}}(w) \coloneqq \min \left\{ N \;\Bigg|\; w \,=\, \prod_{i=1}^N u_i^{-1}r_iu_i, \; u_i \in F(X), \; r_i \in R^{\pm} \right\} \, .
    \]
The \emph{Dehn function} of the presentation $\mathcal{P}$, denoted as $\delta_{\mathcal{P}} : \mathbb{N} \rightarrow \mathbb{N}$, is defined by
    \[
    \delta_{\mathcal{P}}(n) \coloneqq \max \left\{ \mathrm{Area}_{\mathcal{P}}(w) \;|\; w \in F(X), \, w \in \llangle R \rrangle, |w|_X \leq n \right\} \, ,
    \]
where $|w|_X$ denotes the length of the word $w$ in the generators $X^{\pm}$.
\end{defn}

Viewed up to $\simeq$ equivalence, the Dehn function of a group is independent of the choice of the presentation (see Proposition 1.3.3 of \cite{MR1967746}); thus we typically write $\delta_G(n)$ as opposed to $\delta_{\mathcal{P}}(n)$.

Our main theorem and certain arguments in this paper use a notion called the \emph{superadditivity of functions}. We record the definition and a lemma about superadditivity of functions. For more details about superadditivity, we refer the reader to \cite[Section~$2$]{KLI}.

\begin{defn}[Superadditivity]\label{def:superadd}
A function $f : \mathbb{N} \to \mathbb{R}$ is called \emph{superadditive} (sometimes referred to in the literature as \emph{subnegative}) if $f(n+m) \geq f(m) + f(n)$. The \emph{superadditive closure} $\bar{f}$ of $f$ is the (pointwise) smallest superadditive function with $\bar{f}(n) \geq f(n)$, which may be given explicitly as
    \[
    \overline{f}(n) \coloneqq \max\left\{ \sum_{i=1}^r f(n_i) \;\Bigg|\; r \geq 1, \, \sum_{i=1}^r n_i = n , \, n_i \in \N \right\} \, .
    \]
\end{defn}

We record two lemmas which we will use while proving our main theorem, Theorem~\ref{thm:main}. The first is a folklore result courtesy of \cite{MR1102884}, for which we first require the following short definition.

\begin{defn}[Group-theoretic retraction]
Suppose that $G$ is a group and $H \leq G$. We say that a homomorphism $\sigma : G \to H$ is a \emph{(group-theoretic) retraction} if $\sigma(h)=h$ for all $h \in H$. 
\end{defn}

\begin{lem}\label{lemret}
Let $G$, $H$ be finitely presented groups such that $H \leq G$ and there exists a retraction $G \rightarrow H$. Then $\delta_H \preccurlyeq \delta_G$.
\end{lem}

The second of these lemmas also appears with proof in \cite{KLI} as Lemma 2.3.

\begin{lem}\label{maxdehn}
For $k \geq 2$, let $G_1, \ldots, G_k$ be infinite finitely presented groups and denote by $\delta_i$ the Dehn function of $G_i$. Then the Dehn function of the product $G = \prod_i G_i$ is equivalent to $\max \{ n^2 , \delta_i \;|\; 1 \leq i \leq k \}$ under the ordering $\preccurlyeq$ on such functions.
\end{lem}

Finally, we record the definition of the particular structures we investigate, namely, coabelian subgroups.

\begin{defn}[Coabelian subgroups]
Let $G$ be a group and $H \leq G$. The subgroup $H$ is said to be a \emph{coabelian subgroup} of $G$ if $[G,G] \leq H$. Equivalently, $H \unlhd G$ and $G/H$ is abelian.
\end{defn}






\section{The Algebraic Setup}\label{sec:alg-setup}

The main theorem of this paper concerns the Dehn function of $\Ker(\Phi)$ where $\Phi : \prod_{i=1}^{2n} G_i \to \Z^m$ is a particular homomorphism. The ultimate goal is to define a collection of special subgroups of $\Ker(\Phi) \leq G = \prod_{i=1}^{2n} G_i$, called edge groups and face groups, and to use the combinatorial structure of the lattice of these  subgroups in order to provide upper bounds on the Dehn function of the kernel.

This section starts with  the definition of an $n$-split sequence of groups, and a description of the basic setup underlying the main result. Each factor group $G_i$ is endowed with a particular generating set called a {\em kernel-section} generating set, and these generating sets are promoted to {\em vector generating sets} for subgroups of $\Ker(\Phi)$. Finally, we define {\em standard arrangement} of $\mathbb{Z}^m$ in $\Ker(\Phi)$.

\begin{defn}[$n$-split sequence]\label{def:nsplit}
Let $G$ be a group and $0<n\leq m$ be positive integers. A short exact sequence
    \[\begin{tikzcd}
    1 \arrow[r] & \Ker(f) \arrow[r] & G \arrow[r,"f"] & \Z^m \arrow[r]  & 1
    \end{tikzcd}\]
is said to {\em $n$-split} if there exists a decomposition 
    \[
    \Z^m = A_1 \oplus \cdots \oplus A_n
    \]
together with projection maps $\laproj_j : \Z^m \to A_j$ for $1 \leq j \leq n$ and section homomorphisms $s^j : A_j \to G$ such that $\laproj_j \circ f \circ s^j = {\rm id}_{A_j}$ for $1 \leq j \leq n$. We  say that $f$ is $n$-split {\em with respect to} the given decomposition $A_1 \oplus \cdots \oplus A_n$ of $\Z^m$.

Note that a $1$-split short exact sequence is just a split short exact sequence in the usual sense, and that a $2$-split short exact sequence is $P$-split in the sense of \cite{KLI} for $P = \{A_1, A_2\}$. Finally, note that, since $\Z$ is free, every  short exact sequence above  is automatically $m$-split.
\end{defn}

\begin{defn}[Basic setup for Theorem~\ref{thm:main}]\label{def:setup}
The {\em basic setup} consists of the following data. 
    \begin{enumerate}
    \item Integers $ m\geq n \geq 3$, and a decomposition of $\Z^m$ into non-trivial free abelian factors 
        \[
        \Z^m = A_1 \oplus \cdots \oplus A_n ,
        \]
    together with projections $\laproj_j : \Z^m \to A_j$ for $1 \leq j \leq n$. Additionally, there is a choice $B_j$ of a basis for $A_j$ for each $1 \leq j \leq n$. 
    
    \item A family of groups $\{G_i \,|\, 1 \leq i \leq 2n\}$ and $n$-split epimorphisms
        \[
        \varphi_i : G_i \to \Z^m
        \]
    with respect to the given decomposition $\Z^m = A_1 \oplus \cdots \oplus A_n$ in item (1) above for each $1 \leq i \leq 2n$. 
    
    \item From Definition~\ref{def:nsplit}, the $n$-split conditions in item (2) mean that there exist section homomorphisms 
        \[
        s^j_i : A_j \to G_i
        \]
    satisfying $\laproj_j \circ \varphi_i \circ s_i^j = {\rm id}_{A_j}$ for each $1 \leq i \leq 2n$ and each $1 \leq j \leq n$. 
    
    \item Let $G = \prod_{i=1}^{2n}G_i$. The $\varphi_i$ sum to give an epimorphism $\Phi : G \to \Z^m$ defined by
        \[
        \Phi(g) = \sum_{i=1}^{2n} \varphi_i(\grpproj_i(g)) ,
        \]
    where $\grpproj_i : G \to G_i$ denotes the standard one-factor projection.
    \end{enumerate}
\end{defn}

\begin{rem}[Cases $n=1,2$]
The basic setup for Theorem~\ref{thm:main} includes the hypothesis that $n \geq3$. The case  $n=1$ (resp.\ $n=2$) is handled in Section~3 (resp.\ Section~4) of \cite{KLI}. 
\end{rem}

Our goal is to determine upper bounds for the isoperimetric inequality of the kernel of the homomorphism $\Phi$ in the case that this is finitely presented. To do this, we first choose special finite generating sets for the groups $G_i$ called {\em kernel-section generating sets} which partition into two subsets; one a subset of $\Ker(\varphi_i)$ and the other consisting of section lifts of the elements in a basis $B_i$ for each factor $A_i$ in the decomposition of $\Z^m$. 

\begin{defn}[Kernel-section generating sets]\label{def:ks}
Let $G$ be a finitely generated group and 
    \[
    f : G \to \Z^m = A_1 \oplus \cdots \oplus A_n
    \]
be an $n$-split epimorphism as in Definition~\ref{def:nsplit} with sections $s^j: A_j \to G$ satisfying $\laproj_j \circ f \circ s^j = {\rm id}_{A_j}$, for $1 \leq j \leq n$. 

A {\em kernel-section generating set} for $G$ with respect to the $n$-split epimorphism $f$ is a generating set for $G$ which is a disjoint union $Y \sqcup Z^1 \sqcup \cdots \sqcup Z^n$ of  finite sets $Y \subseteq \Ker(f)$ and $Z^j = s^j(B_j)$, where $B_j$ is a basis for the free abelian group $A_j$, for each $1 \leq j \leq n$.

The elements of $Y$ are called {\em kernel-generators} and the elements of $Z^1 \sqcup \cdots \sqcup Z^n$ are called {\em section-generators}. 
\end{defn}

\begin{rem}[Kernel generators]
Note that the set of kernel-generators $Y$ defined above do not necessarily generate $\Ker(f)$. Indeed, this kernel need not be finitely generated. Even in the case when $\Ker(f)$ is finitely generated, the set $Y$ may not be a generating set. The phrase {\em kernel-generators} is used simply to indicate that $Y \subseteq \Ker(f)$, and to distinguish them from the {\em section-generators} in $Z^j$. 
\end{rem}

\begin{rem}[Notation for inverses]
We will use exponential notation to denote the inverse of elements (so the inverse of the element $a$ is denoted by $a^{-1}$). We reserve the overline notation $\overline{a}$ for certain ``vectorized'' elements of $G = \prod_{j=1}^{2n} G_j$. 
\end{rem}

Our first result is that every finitely generated group $G$ as in Definition~\ref{def:ks} admits a finite kernel-section generating set. 

\begin{lem}\label{lem:ker-sec}
Let $G$ be a finitely generated group and $f : G \to \Z^m = A_1 \oplus \cdots \oplus A_n$ be an $n$-split epimorphism with sections $s_j : A_j \to G$ for $1 \leq j \leq n$. Then there exists a finite kernel-section generating set $Y \sqcup Z^1 \sqcup \cdots \sqcup Z^n$ for $G$ with respect to $f$. 
\end{lem}

\begin{proof}
For each $1 \leq j \leq n$ let $\laproj_j : \Z^m \to A_j$ denote the projection map onto the $A_j$ factor group. Since $G$ is finitely generated, we can choose a finite generating set $\{a_1, \ldots, a_p\}$. For each $1 \leq i \leq p$ define 
    \begin{equation}\label{def:ker-sec-gens}
    y_i = a_i s^1(\laproj_1(f(a_i)))^{-1} \ldots s^n(\laproj_n(f(a_i)))^{-1}
    \end{equation}
and set 
    \[
    Y = \{y_i \;|\; 1 \leq i \leq p\} \, .
    \]
Applying  $f$ to each $y_i$ gives  
    \begin{align*}
    f(y_i) &= f(a_i) + f(s^1(\laproj_1(f(a_i)))^{-1}) + \ldots + f(s^n(\laproj_n(f(a_i)))^{-1}) \\
    &= f(a_i) + \laproj_1(f(s^1(\laproj_1(f(a_i)))^{-1})) + \ldots + \laproj_n(f(s^n(\laproj_n(f(a_i)))^{-1})) \\
    &= f(a_i) - \laproj_1(f(a_i)) - \ldots - \laproj_n(f(a_i)) \\   
    & = 0 \, ,
    \end{align*}
where each $\laproj_j(f(s^1(\laproj_1(f(a_i)))^{-1})) = f(s^1((\laproj_1(f(a_i)))^{-1})$ since these elements belong to $A_j \leq \Z^m$, and where the $\laproj_j \circ f \circ s^j$ composites cancel since the $s^j$ are section maps. Therefore, $Y \subseteq \Ker(f)$. 

Let $B_j$ be a basis for the module $A_j$ and set
    \[
    Z^j \;=\;  \{s^j(z) \;|\; z \in B_j\} \, .
    \]
By definition, the elements of the sets $Z^j$ are section-generators. 
    
Finally, for each $1 \leq i \leq p$, 
    \[
    a_i = y_i s^n((\laproj_n(f(a_i)))\cdots s^1((\laproj_1(f(a_i)))
    \]
is a word in $Y \sqcup Z^1 \sqcup \cdots \sqcup Z^n$. Hence, $Y \sqcup Z^1 \sqcup \cdots \sqcup Z^n$ is a finite kernel-section generating set for $G$ with respect to $f$. 
\end{proof}

Next, we establish notation for kernel-section generating sets for the basic setup of Definition~\ref{def:setup}. 

\begin{defn}[Kernel-section generating sets for the basic setup]\label{def:setupgen}
Let $n \geq 3$ be an integer and 
    \[
    \varphi_i: G_i \to \Z^m = A_1 \oplus \cdots \oplus A_n
    \]
be an $n$-split epimorphism for each $1 \leq i \leq 2n$, as in the basic set up of Definition~\ref{def:setup}. Then for each $1 \leq i \leq 2n$, Lemma~\ref{lem:ker-sec} takes as input a finite generating set for $G_i$ and outputs a finite {\em kernel-section generating set for $G_i$} which we denote by 
    \[
    Y_i \sqcup Z^1_{i} \sqcup \cdots \sqcup Z^n_{i} \, .
    \]
\end{defn}

\begin{rem}[Word metric on $G = \prod_{i=1}^{2n}G_i$]\label{rem:l1metric}
From now on, we assume that each $G_i$ is given a finite kernel-section generating set as in Definition~\ref{def:setupgen}. The word metric on each $G_i$ is determined by the kernel-section generating set $Y_i \sqcup Z^1_i \sqcup \cdots \sqcup Z^n_i$, and the word metric on the product $G = \prod_{i=1}^{2n} G_i$ is the corresponding $\ell^1$-metric induced by the word metrics on the $G_i$.
\end{rem}

\begin{defn}[Projections]
We make use of projection maps from $G = \prod_{i=1}^{2n} G_i$ to various target groups. 
    \begin{itemize}
    \item The most basic projection map is the one-factor projection 
        \[
        \grpproj_i : \prod_{j=1}^{2n} G_j \;\to\; G_i : (g_1, \ldots, g_{2n}) \;\mapsto\; g_i 
        \]
    defined for each $1 \leq i \leq 2n$. 
        
    \item We also consider multi-factor projection maps where the factors are indexed by some index set $\alpha \subseteq \{1, \ldots , 2n\}$:
        \[
        \grpproj_{\alpha} : \prod_{j=1}^{2n} G_j \;\to\; \prod_{j \in \alpha} G_j \;\reflectbox{$\coloneqq$}\; G_{\alpha}
        \]
    for $\alpha \subseteq \{1, \ldots , 2n\}$. This map is defined by $\mathrm{pr}^{G_\alpha}_j\circ \grpproj_{\alpha} =\grpproj_j$ for all $j \in \alpha$. 
    \end{itemize}
\end{defn}

Recall from Definition~\ref{def:setup} that the epimorphisms $\varphi_i : G_i \to \Z^m$ for $1\leq i \leq 2n$ sum together to determine an epimorphism $\Phi : G = \prod_{i=1}^{2n} G_i \to \Z^m$. We define special subgroups of $\Ker(\Phi)$ via generating sets consisting of vector ($2n$-tuple) analogues of the kernel-section generating sets of the $G_i$. We define these vector generators presently. 

\begin{defn}[Vector generating sets for subgroups of $\Ker(\Phi)$]\label{def:vector_gen}
Let $0 < n \leq m$ and $G_i$ and $\Phi : G \to \Z^m$ be as in the basic setup Definition~\ref{def:setup}. For each $1 \leq i \leq 2n$, we define vector analogues of the set
    \[
    Y_i \sqcup Z_i^1 \sqcup \cdots \sqcup Z_i^n 
    \]
of kernel-section generators for the group $G_i$.  
    \begin{itemize}
    \item For each $1 \leq i \leq 2n$, define 
        \[
        \overline{Y}_i \;=\; \{\overline{y} \,|\, y \in Y_i\} \, ,
        \]
    where $\overline{y}$ is a $2n$-tuple in $G$ defined by
        \[
        \grpproj_j(\overline{y}) \;=\;
            \begin{cases}
            e & j \not= i \\
            y & j = i
            \end{cases} \, .
        \]
        
    \item For each $1 \leq i \not=k \leq 2n$ and each $1 \leq j \leq n$ define 
        \[
        \overline{Z}^j_{i,k} \; = \; \{ \overline{z} \,|\, z \in Z^j_i \} \, ,
        \]
    where (using the fact that $z=s^j_i(b)$ for some $b \in B_j$) we define the $2n$-tuple $\overline{z}$ by  
        \[
        \grpproj_\ell(\overline{z}) \;=\;
            \begin{cases}
            e & \ell \not= i,k \\
            s_i^j(b) & \ell = i \\
            s_k^j(-b) & \ell = k
            \end{cases} \, .
        \]
    \end{itemize}
        
By definition, these sets are all contained in $\Ker(\Phi)$. 
\end{defn}

\begin{rem}[Ranges of the indices of $\overline{Z}^j_{i.k}$]\label{rem:superscripts}
The superscript index $j$ refers to the the summand $A_j$ in the decomposition $A_1 \oplus \cdots \oplus A_n$ of $\Z^m$, and so ranges between $1$ and $n$.

When used in defining subgroups of $\Ker(\Phi)$, the subscript index $i$ ranges over some subset $A \subseteq \{1, \ldots , 2n\}$ and the second subscript index $k$ lies in the complement of $A$ in $\{1, \ldots, 2n\}$. 
\end{rem}

\begin{rem}[Representatives modulo $k$]
In this paper, the set of representatives of integers modulo $k$ is taken to be $\{1, 2, \dots, k-1, k\}$. That is, $k$ is used instead of $0$ for the value of $k \pmod{n}$. 
\end{rem}

In Lemma~\ref{lem:std-Zm} below, we describe a subgroup $\mathcal{L}$ of $\Ker(\Phi)$ via its generating set and show that this subgroup is isomorphic to $\Z^m$. The isomorphism type of $\mathcal{L}$ is established via the {\em $A$-determinate} property, which we define and explore below. This property is used repeatedly in the next section in establishing isomorphism types of other special subgroups of $\Ker(\Phi)$. 

\begin{defn}[$A$-determinate elements and subgroups]\label{def:determinate} 
Let $A \subseteq \{1, \dots, 2n\}$ have cardinality $n$ and let $B = \{1, \ldots, 2n\} - A$.   An element $\overline{g} \in \prod_{i=1}^{2n} G_i$ is said to be {\em $A$-determinate} if, for each $j \in B$, there exists a homomorphism
    \[
    f_j : \prod_{i \in A} G_i \;\to\; G_j
    \]
such that the $j$-th coordinate entry of $\overline{g}$ satisfies
    \[
    \grpproj_j(\overline{g}) \;=\; f_j(\grpproj_A(\overline{g})).
    \]
That is, $g_j$ is the $f_j$-homomorphic image of the $A$-entries of $\overline{g}$. 

A subgroup of $\prod_{i=1}^{2n} G_i$ is said to be {\em $A$-determinate} if each of its elements are $A$-determinate with respect to a given, fixed, collection of homomorphisms $f_j, \; j \in B$.
\end{defn}

A key property of an $A$-determinate subgroup $H$ is that the projection map from $H$ to the product of its $A$-factors is injective. This property is used in determining the isomorphism type of various subgroups of $\Ker(\Phi)$ defined in this and subsequent sections. 

\begin{lem}[$A$-determinate implies injectivity of $\grpproj_A$]\label{lem:detinj}
Let $A \subseteq \{1, \dots , 2n\}$ have cardinality $n$ and $B$ be its complement, and let $H \leq \prod_{i=1}^{2n} G_i$ be an $A$-determinate subgroup. Then the projection map
    \[
    \grpproj_A : H \;\to\; \prod_{i \in A} G_i
    \]
is injective.
\end{lem}

\begin{proof}
Given $\overline{g} \in \Ker(\grpproj_A)$, then $g_i = e_i$ for all $i \in A$, and each $g_j = f_j(\grpproj_A(\overline{g})) = e$ for $j \in B$ by $A$-determinate hypothesis. Therefore, $\overline{g} = \overline{e} \in \prod_{i=1}^{2n} G_i$, and so $\Ker(\grpproj_A)$ is the trivial subgroup.
\end{proof}

The next result shows that the $A$-determinacy of a subgroup $H$ with respect to a collection of homomorphisms follows from the $A$-determinacy of a generating set for $H$ with respect to the same fixed collection of homomorphisms. It is used in establishing the $A$-determinacy of various subgroups of $\Ker(\Phi)$. 

\begin{lem}[$A$-determinate generators]\label{lem:detgen}
Let $G = \prod_{i=1}^{2n} G_i$ and $A, \, B \subseteq \{1, \ldots, 2n\}$ be as in Definition~\ref{def:determinate}. Suppose that $H \leq G$ is generated by a set $S$ and that there exists a collection of homomorphisms 
    \[
    f_j : \prod_{i \in A} G_i \to G_j \, ,
    \]
indexed by $j \in B$, such that 
    \[
    \grpproj_j(\overline{s}) \;=\; f_j(\grpproj_A(\overline{s}) )
    \]
for all $\overline{s} \in S$. Then $H$ is $A$-determinate. 
\end{lem}

\begin{proof}
Each element $\overline{h} \in H$ can be written as a word 
    \[
    \overline{h} \;=\; w(S^\pm)
    \]
in the generators $\overline{s} \in S$ and their inverses. Since the same homomorphisms $f_j$ are used for each $\overline{s} \in S$, we have, for $\overline{h} \in H$ and $j \in B$, 
    \begin{align*}
    \grpproj_j(\overline{h}) \;&=\; \grpproj_j(w(S^\pm)) \\
    &=\; w(\grpproj_j(S^\pm)) \\
    &=\; w(f_j(\grpproj_A(S^\pm))) \\
    &=\; f_j(\grpproj_A(w(S^\pm))) \\
    &=\; f_j(\grpproj_A(\overline{h})) ,
    \end{align*}
and so $H$ is $A$-determinate. The main point here is that for each fixed $j \in B$, the same homomorphism $f_j \circ \grpproj_A$  is used (uniformly) for all of the generators $\overline{s} \in S$.  
\end{proof}

The next lemma uses the previous two results on $A$-determinacy to show that the so-called {\em standard arrangement $\mathcal{L}$ of $\Z^m$} in $G$ is indeed free abelian of rank $m$. 

\begin{lem}[Standard arrangement of $\Z^m$ in $\Ker(\Phi)$]\label{lem:std-Zm}
Let $n \leq m$ be positive integers and $\Phi : G \to \Z^m$ be as in Definition~\ref{def:setup}. The subgroup $\mathcal{L}$ of $\Ker(\Phi)$ generated by 
    \[
    \bigcup_{i=1}^n \overline{Z}_{i,i+n}^i
    \]
is isomorphic to $A_1 \times \cdots \times A_n \cong \Z^m$. 
\end{lem}

\begin{proof}
Let $\mathcal{L}$ denote the subgroup generated by the above set. Let $\alpha = \{1, \ldots, n\} \subseteq \{1, \ldots, 2n\}$. The projection map $\grpproj_\alpha : G \to \prod_{j \in \alpha} G_j$ restricts to $H$ to give a homomorphism
    \[
    \grpproj_{\alpha}|_{\mathcal{L}} : \mathcal{L} \;\to\; \prod_{j \in \alpha} A_j \;=\; A_1 \times \cdots \times A_n .
    \]
Since the $i$ entries of the elements of $\overline{Z}^i_{i,i+n}$ are section lifts of basis elements of $A_i$, this map is a surjection. 

Note that each element $\overline{z}$ of the generating set $\bigcup_{i=1}^n \overline{Z}^i_{i,i+n}$ has the property that its last $n$ coordinates are particular homomorphic images of its first $n$ coordinates. In fact,
    \[
    z_{j+n} \;=\; s^j_{j+n}\left(-\laproj_j\left(\sum_{i=1}^n\varphi_i(z_i)\right)\right)
    \]
for $1 \leq j \leq n$. By Lemma~\ref{lem:detgen} the subgroup $\mathcal{L}$ is $\alpha$-determinate, and so $\grpproj_{\alpha}|_{\mathcal{L}}$ is injective by Lemma~\ref{lem:detinj}. Therefore, $\mathcal{L}$ is isomorphic to $A_1 \times \cdots \times A_n \cong \Z^m$. 
\end{proof}

\begin{defn}[Standard arrangement of $\Z^m$]\label{def:standard1}
Let $n$ be a positive integer and $\Phi : G \to \Z^m$ be as in Definition~\ref{def:setup}. The subgroup $\mathcal{L}$ of $\Ker(\Phi)$ generated by
    \[
    \bigcup_{i=1}^n \overline{Z}^i_{i,i+n}
    \]
is isomorphic to $\Z^m$ by Lemma~\ref{lem:std-Zm}, and is called the {\em standard arrangement} of $\Z^m$ in $\Ker(\Phi)$.
\end{defn}

\section[defining subgroups]{Standard and nonstandard subgroups of \texorpdfstring{$\Ker(\Phi)$}{the kernel}}\label{sec:subgroups}

In this section, we define the notions of {\em standard} and {\em non-standard} subgroups of $\Ker(\Phi)$. We define a collection of twenty-five face subgroups and twenty-one edge subgroups, which are used in Section~\ref{sec:con-algtri} in the  definition of  the {\em algebraic triangle} -- a key combinatorial object used in establishing the isoperimetric upper bounds for $\Ker(\Phi)$ in Theorem~\ref{thm:main}. A face subgroup (resp. edge subgroup) is associated to each of the $2$-cells (resp. $1$-cells) of the algebraic triangle.

The idea of standard and non-standard subgroups of $\Ker(\Phi)$ is motivated in Subsection~\ref{sub:subgroups.0}. In Subsection~\ref{sub:subgroups.1}, we provide background definitions and preliminary results which are useful for defining and describing the face and edge subgroups of $\Ker(\Phi)$. These include the notions of {\em label sequences}, {\em subgroup patterns}, and {\em linear algebra encoding}. 

In Subsection~\ref{sub:subgroups.2}, we define the thirteen standard face subgroups of $\Ker(\Phi)$ and study their basic combinatorial and algebraic properties. In Subsection~\ref{sub:subgroups.3}, we do the same for the twelve non-standard face subgroups of $\Ker(\Phi)$, which requires us to adapt the notions of pattern (to consider {\em $B$-patterns}) and linear algebra encoding for this setting. Finally, in Subsection~\ref{sub:subgroups.4}, we define the twenty-one (standard and non-standard) edge subgroups of $\Ker(\Phi)$, and  record their algebraic and combinatorial properties.  

\subsection{Motivation for the standard and non-standard constructions}\label{sub:subgroups.0}

Many of the face and edge subgroups (to be defined below) are abstractly isomorphic to products of the groups $G_j$ as $j$ ranges over some index set. One needs to be careful when trying to build products inside of $\Ker(\Phi)$. 

For example, suppose one wishes to build a copy of $G_1 \times G_2$ as a subgroup $H \leq \Ker(\Phi)$ with the property that $\grpproj_{12}(H)= G_1 \times G_2$. One needs a copy of the kernel-section generators for $G_1$ in coordinate 1 and a copy of the kernel-section generators for $G_2$ in coordinate 2. This is achieved by using appropriate vector forms of the kernel-section generators as in Definition~\ref{def:vector_gen}. 

However, one has to be careful with the choices of vector forms of the section generators. For example, suppose a particular section generator for $G_1$ comes from the lift of an element $b \in A_i \leq \Z^m$. The vector form of this element consists of a lift $s^i_1(b)$ in coordinate $1$, and  a lift $s^i_k(-b)$ for some coordinate $k \not= 1$. Suppose further that a given section generator for $G_2$ comes from a lift of some element $b' \in A_j$ where $j \not= i$. The vector form of this element consists of $s^j_2(b')$ in coordinate 2, and $s^j_\ell(-b')$ in some coordinate $\ell \not= 2$. If $\ell = k$, then the two lifts $s^i_k(-b)$ and $s^j_k(-b')$ may not commute in the group $G_k$, and so the two vector forms of the section generators would not commmute, and so the subgroup $H$ may not be isomorphic to $G_1 \times G_2$. 

In order to  avoid this situation, we set aside $n+2$ coordinate slots. Slots 1 and 2 are for the subgroups $G_1$ and $G_2$, and the remaining $n$ slots are used to park the various $s^i_k(-b)$ terms in such a way that lifts of elements from $A_i$ always commute with lifts of elements from $A_j$ whenever $i \not= j$. We call this the problem of handling the {\em linear algebra}. 

One strategy is to lift generators $b \in A_i$ to either $s^i_i(-b)$ or to $s^i_{i+n}(-b)$, with the subscript index taken modulo $2n$ as usual. This  is referred to as a  {\em standard linear algebra} choice. 

For certain subgroups that we build, the standard  way of handling the linear algebra lifts is not possible (because a pair of coordinates $j$ and $j+n$ may be used for other purposes), and we will use {\em non-standard linear algebra} choices. The non-standard linear choices that we employ are motivated by the definition of the four corner subgroups in the algebraic square construction of \cite{KLI}.

\subsection{Preliminary definitions and results}\label{sub:subgroups.1}

The various standard and non-standard subgroups of $\Ker(\Phi)$ are identified by {\em label sequences}. We first define six subsets of $\{1, \ldots, 2n\}$ called {\em $\alpha$-label sequences} and describe how to build new label sequences from these. 

\begin{defn}[$\alpha$-label sequences]\label{def:lab_seq}
Associated to each integer $n \geq 3$, define six basic {\em label sequences}, denoted by $\alpha_1, \ldots, \alpha_6$, which partition the set $\{1, \dots, 2n\}$ approximately into sixths. The label sequences are defined in two groups of three as follows. 

For $1 \leq i \leq 3$, define 
    \[
    \alpha_i \coloneqq \left\{ \left\lceil \frac{1}{3}(i-1)n \right\rceil + 1, \dots, \left\lceil \frac{1}{3}in \right\rceil \right\} \, .
    \]
Explicitly, the values of these $\alpha_i$ are summarized in the following table, depending on the value of $n \pmod{3}$; here, $k = \lfloor n/3 \rfloor$.

    \vspace{0.1in}
    \begin{center}
        \begin{tabular}{|l||c|c|c|}
        \hline
        \mbox{}             & $i = 1$               & $i = 2$                   & $i = 3$ \\
        \hline \hline
        $n = 0 \pmod{3}$    & $\{1, \ldots, k\}$    & $\{k+1, \ldots, 2k\}$     & $\{2k+1, \ldots, 3k\}$ \\
        \hline
        $n = 1 \pmod{3}$    & $\{1, \ldots, k+1\}$  & $\{k+2, \ldots, 2k+1\}$   & $\{2k+2, \ldots, 3k+1\}$ \\
        \hline
        $n = 2 \pmod{3}$    & $\{1, \ldots, k+1\}$  &  $\{k+2, \ldots, 2k+2\}$  & $\{2k+3, \ldots, 3k+2\}$ \\
        \hline
        \end{tabular}
    \end{center}
    \vspace{0.1in}
Finally,  for $4 \leq i \leq 6$, define 
    \[
    \alpha_i = \{ j+n \,|\, j \in \alpha_{i-3} \} .
    \]

Let $|\alpha|$ denote the length (cardinality) of the label sequence $\alpha$. Note that $|\alpha_i| = |\alpha_{i+3}|$ for each $1 \leq i \leq 3$ and $|\alpha_i| + |\alpha_{(i+1) \!\pmod{6}}| + |\alpha_{(i+2) \!\pmod{6}}| = n$ for all $1 \leq i \leq 6$.

The term `sequence' refers to the fact that we think of these $\alpha_i$ as being linearly ordered sets, with ordering inherited from the usual ordering on $\{1, \ldots, 2n\}$. 
\end{defn}

Next we describe how to define new label sequences, built from the $\alpha_i$ sequences above. These new sequences inherit a linear ordering as segments of the cyclic ordering on $\{1, \ldots, 2n\}$. 

\begin{defn}[New $\alpha$-sequences from old]
We introduce the following notations for unions and complements of $\alpha$-sequences. 
    \begin{itemize}
    \item Unions of $\alpha$-sequences are denoted by extending the subscripts. So, $\alpha_{ij} = \alpha_i \cup \alpha_j$, $\alpha_{ijk} = \alpha_i \cup \alpha_j \cup \alpha_k$, and so on. 

    We think of these new sequences as inheriting  linear orderings as segments in  the standard cyclic ordering on $\{1, \ldots, 2n\}$. So for example, the elements of $\alpha_6$ all precede those of $\alpha_1$ in the sequence $\alpha_{61}$.  
    
    \item Complements in $\{1, \ldots , 2n\}$ are denoted via a $c$-superscript, e.g., $\alpha_{61}^c = \alpha_{2345}$. 
    \end{itemize}
Finally, the interval notation $\alpha [ i,j ]$ is used to denote the subsequence of $\alpha$ beginning at the $i$th entry of $\alpha$ and ending at the $j$th entry (both inclusive).
\end{defn}

For each of the face and edge subgroups of $\Ker(\Phi)$ we determine three features:
    \begin{itemize}
    \item the {\em isomorphism type} of the subgroup,
    \item the {\em pattern} of the subgroup, and
    \item the {\em linear algebra encoding} of the subgroup.
    \end{itemize}
The {\em pattern} of a subgroup (see Definition~\ref{defn:pattern} below) of $\Ker(\Phi)$ is a $5$- or $6$-tuple of homomorphic images of the subgroup which is obtained by projecting the subgroup onto various factors. The {\em linear algebra encoding} (see Definition~\ref{def:lin_algebra} below) is a particular matrix that records the various choices of section generators $\overline{Z}^j_{i,k}$ used in defining the subgroup. 

\begin{defn}[Projection images]\label{defn:proj-special}
The following two types of groups occur as images of the special subgroups of $\Ker(\Phi)$ under the projection homomorphisms $\grpproj_{\alpha_i}$ for $1 \leq i \leq 6$.

    \begin{itemize}
    \item The first type is a product of the factor groups $G_i$. For $1 \leq i \leq 6$, define 
        \[
        G_{\alpha_i} \;=\; \prod_{j\in\alpha_i} G_j \, . 
        \]
    \item The second type is a product of copies of  factors $A_i$ in $\Z^m$. For each $1 \leq i \leq 3$, define 
        \[
        A^{\alpha_i}_i \;=\; \prod_{j \in \alpha_i} s^j_j(A_j)
        \]
    and 
        \[
        A^{\alpha_i}_{i+3} \;=\; \prod_{j \in \alpha_i} s^j_{j+n}(A_j) \, . 
        \]
    \end{itemize}
\end{defn}

\begin{defn}[Subgroup patterns]\label{defn:pattern}
Let $G = \prod_{j=1}^{2n} G_j$ and let $H \leq G$. For each $1 \leq i \leq 6$ there are projection homomorphisms 
    \[
    \grpproj_{\alpha_i} : G \to \prod_{j \in \alpha_i} G_j \, . 
    \]
The {\em pattern} of $H$ is defined to be the $6$-tuple
    \[
    (\grpproj_{\alpha_1}(H), \; \grpproj_{\alpha_2}(H), \; \grpproj_{\alpha_3}(H), \; \grpproj_{\alpha_4}(H), \; \grpproj_{\alpha_5}(H), \;  \grpproj_{\alpha_6}(H)) \, .
    \]
For certain (non-standard) subgroups $H$, we combine two of the $\alpha_i$ sequences together and obtain a $5$-tuple instead of a $6$-tuple pattern. For example, the projection 
    \[
    \grpproj_{\alpha_{12}} : G \;\to\; \prod_{j \in \alpha_{12}} G_j
    \]
is used in the $5$-tuple pattern 
    \[
    (\grpproj_{\alpha_{12}}(H), \; \grpproj_{\alpha_3}(H), \; \grpproj_{\alpha_4}(H), \; \grpproj_{\alpha_5}(H), \; \grpproj_{\alpha_6}(H)) \, .
    \]
There are two other similarly defined $5$-tuple patterns that we use: 
    \[
    (\grpproj_{\alpha_1}(H), \; \grpproj_{\alpha_2}(H), \; \grpproj_{\alpha_{34}}(H), \; \grpproj_{\alpha_5}(H), \; \grpproj_{\alpha_6}(H)) 
    \]
and
    \[
    (\grpproj_{\alpha_1}(H), \; \grpproj_{\alpha_2}(H), \; \grpproj_{\alpha_3}(H), \; \grpproj_{\alpha_4}(H), \; \grpproj_{\alpha_{56}}(H)) \, .
    \]
\end{defn}

\begin{defn}[Pattern entries]\label{defn:pattern_entry}
In the case of $6$-tuple patterns, the possible entries are described in Definition~\ref{defn:proj-special}; namely, the groups $G_{\alpha_i}$ for $1 \leq i \leq 6$ and the groups $A^{\alpha_i}_i$ and $A^{\alpha_i}_{i+3}$ for $1 \leq i \leq 3$.

In the case of $5$-tuple patterns, there are three additional entry formats; we denote these as $A^{\alpha_3}_1 E$, $A^{\alpha_2}_3 E$, and $A^{\alpha_1}_5 E$. 

    \begin{itemize}
    \item The first group is the $\grpproj_{\alpha_{12}}$-image of various special subgroups of $\Ker(\Phi)$ and has the form 
        \[
        A^{\alpha_3}_1 E \;=\; \prod_{j \in \alpha_3} s^j_{j-|\alpha_{12}|}(A_j) \times \prod_{j=|\alpha_3|+1}^{|\alpha_{12}|} \{e\} \, .
        \]
    That is, lifts of $A_j$ for $j \in \alpha_3$ to the first $|\alpha_3|$ coordinates of $G$ and then trivial subgroups in the remaining coordinates up through coordinate $|\alpha_{12}|$. It is abstractly isomorphic to $\bigoplus_{j \in \alpha_3} A_j$. Note that a second notation of the lift maps $s^j_{j-|\alpha_{12}|}$ is $s^j_{j+|\alpha_3|-n}$. 

    \item The second group is the $\grpproj_{\alpha_{34}}$-image of various special subgroups of $\Ker(\Phi)$, and has the form
        \[
        A^{\alpha_2}_3 E \; =\; \prod_{j \in \alpha_2}s^j_{j+|\alpha_2|}(A_j)\times \prod_{j = |\alpha_{12}|+|\alpha_2| + 1}^{|\alpha_{1234}|}\{e\} \, . 
        \]
    This consists of lifts of $A_j$ for $j \in \alpha_2$ to  coordinates $|\alpha_{12}| + 1$ through $|\alpha_{12}| + |\alpha_2|$ of $G$ and trivial groups in all other coordinates up through coordinate $|\alpha_{1234}|$. It is abstractly isomorphic to $\bigoplus_{j \in \alpha_2} A_j$. 

    \item The third and final group is the $\grpproj_{\alpha_{56}}$-image of various special subgroups of $\Ker(\Phi)$, and is of the form 
        \[
        A^{\alpha_1}_5 E \; =\; \prod_{j \in \alpha_1}s^j_{j+|\alpha_{1234}|}(A_j)\times \prod_{j=|\alpha_{1234}|+|\alpha_1|+1}^{2n}\{e\} \, .
        \]
    This consists of lifts of $A_j$ for $j \in \alpha_1$ to the coordinates $|\alpha_{1234}| + 1$ through $|\alpha_{1234}| + |\alpha_1|$ of $G$ and trivial groups in all other coordinates up through coordinate $2n$. It is abstractly isomorphic to $\bigoplus_{j \in \alpha_1} A_j$.  
    \end{itemize}
\end{defn}

\begin{defn}[Linear algebra encoding]\label{def:lin_algebra}
{\em Linear algebra encoding} is a way of recording the various section lifts of the three factors $\bigoplus_{j \in \alpha_i} A_j$ of $\Z^m$ (for $1 \leq i \leq 3$) that occur in face and edge  subgroups of $\Ker(\Phi)$ and which appear in the corresponding patterns. 

The $\bigoplus_{j \in \alpha_i} A_j$ factor can appear as $A^{\alpha_i}_i$, $A^{\alpha_i}_{i+3}$, and $A^{\alpha_i}_j E$ (with $j \not\in \{i, i+3\}$) in these patterns; these are encoded by column vectors $\begin{pmatrix} i \\ i \end{pmatrix}$, $\begin{pmatrix} i \\ i + 3 \end{pmatrix}$, and $\begin{pmatrix} \dot{i} \\ j \end{pmatrix}$ respectively. The dot in the last encoding denotes a non-standard lift of $\bigoplus_{j \in \alpha_i} A_j$. If a particular $A^{\alpha_i}$ occurs more than once in the pattern, the corresponding values are separated by a slash ($/$) in the second row of the matrix.

Finally,  if each $A^{\alpha_i}$ occurs at least once in a given pattern, then the encoding is given as a $(2 \times 3)$-matrix, with first row being $(1 \; 2 \; 3)$ with a dot over entries involving non-standard lifts. In other cases, the linear algebra encoding is given by a $(2\times 1)$-column vector. 
\end{defn}

The following definitions and subsequent lemma are useful in establishing subgroup patterns of special subgroups of $\Ker(\Phi)$ which are defined via $\overline{Y}$ and $\overline{Z}$ generating sets.

\begin{defn}[Cyclically convex subset of $\{1, \ldots, 2n\}$]
A subset $A \subseteq \{1, \dots, 2n\}$ is called {\em cyclically convex} if there exists an interval (i.e., a convex subset of $\R$) $[a,b]$ with $a,b \in \Z$ such that $A = ([a,b] \cap \Z) \pmod{2n}$.

For example, two cyclically convex subsets of $\{1, \dots, 6\}$ would be $\{3, 4, 5\}$ and $\{1, 2, 6\}$. The latter example can be seen as $([6, 8] \cap \Z) \pmod{6}$. The subset $\{2, 4\}$ is not cyclically convex in $\{1, \dots, 6\}$.
\end{defn}

We now introduce a notion of $B$-patterns in order to determine the {\em linear algebra encoding} structure (that is, the lifts of the  $A_{\alpha_i}$ subgroups of $\Z^m$, see Definition~\ref{def:lin_algebra}) of various subgroups $H \leq \Ker(\Phi)$. 

\begin{defn}[$B$-patterns]\label{def:Bpatterns}
Let $A \subseteq \{1, \dots, 2n\}$ be cyclically convex and $B$ be the complement of $A$ in $\{1, \dots, 2n\}$. Let $G$ and $H \leq G = \prod_{i=1}^{2n} G_i$ be an $A$-determinate subgroup. The {\em $B$-pattern} of $H$ is defined to be
    \[
    (\grpproj_{\alpha_1 \cap B}(H), \; \grpproj_{\alpha_2 \cap B}(H), \; \grpproj_{\alpha_3 \cap B}(H), \; \grpproj_{\alpha_4 \cap B}(H), \; \grpproj_{\alpha_5 \cap B}(H), \;  \grpproj_{\alpha_6 \cap B}(H)) \, . 
    \]
\end{defn}

This lemma explains how to compute $B$-patterns. 

\begin{lem}[$B$-patterns]\label{lem:Bpatterns}
Let $A$ be a cyclically convex subset of $\{1, \ldots, 2n\}$ of length $|A| \leq n$ and $B$ be its complement. Let $H \leq G$ have generating set
    \[
    \bigcup_{i \in A} \left( \overline{Y}_i \;\cup\; \bigcup_{j=1}^n \bigcup_{j'} \overline{Z}^j_{i,j'}  \right)
    \]
where $j' \in \{j, j+n\} \cap B$. Then for any $k \in B$, $\grpproj_k(H) = A_j$, where $j$ is the unique element of $\{k, k \pm n\} \cap \{1, \ldots, n\}$. 
\end{lem} 

\begin{proof}
Note that 
    \[
    \bigcup_{j \in \{1, \ldots, n\}} \{j, j+n\} \;=\; \{1, \ldots, 2n\}
    \]
and so, as $j$ ranges over $\{1, \ldots, n\}$, then $j'$ ranges over all of $\{1, \ldots, 2n\} \cap B = B$. This means, given $k \in B$, there exists $j \in \{1, \ldots, n\}$ such that $j' = k$. 

Furthermore, the value $j$ in the previous paragraph is uniquely determined by $k$, since it must equal $k$ or $k+n \!\!\pmod{2n}$, and the difference is $|k -(k+n)| = n$, which is strictly greater than the difference between elements of $\{1, \ldots, n\}$. That is, $j= k'$ is the unique element of the set $\{k, k \pm n\} \cap \{1, \ldots, n\}$. 

For $i \in A$, $\grpproj_k(\overline{Y}_i) = \{e\}$ by definition of the elements of $\overline{Y}_i$, and since $k \not\in A$. Therefore
    \begin{align*}
    \grpproj_k(H) \;&=\; \grpproj_k \left( \left\langle \bigcup_{i \in A} \left( \overline{Y}_i \;\cup\; \bigcup_{j=1}^n \bigcup_{j'} \overline{Z}^j_{i,j'} \right) \right\rangle \right) \\
    &=\; \left\langle \grpproj_k \left( \bigcup_{i \in A} \left( \overline{Y}_i \;\cup\; \bigcup_{j=1}^n \bigcup_{j'} \overline{Z}^j_{i,j'} \right) \right) \right\rangle \\
    &=\; \left\langle \grpproj_k\left( \bigcup_{i \in A} \bigcup_{j=1}^n \bigcup_{j'} \overline{Z}^j_{i,j'} \right) \right\rangle
    {\hbox{ (since $\grpproj_k(\overline{Y}_i) = \{e\}$ for all $i \in A$)}} \\
    &=\; \left\langle \bigcup_{i \in A} \bigcup_{j=1}^n \bigcup_{j'} \grpproj_k \left( \overline{Z}^j_{i,j'} \right) \right\rangle \\
    &=\; \left\langle \bigcup_{i \in A} \grpproj_k\left(\overline{Z}^{k'}_k\right) \right\rangle {\hbox{ (where $k'$ is the unique element of $\{k, k \pm n\} \cap \{1, \ldots, n\})$}} \\
    &=\; \left\langle \grpproj_k\left(\overline{Z}^{k'}_k\right) \right\rangle \\
    &=\; A_{k'} \\
    &=\; A_j {\hbox{ (where $j = k'$ is uniquely defined above)}} .
    \end{align*}
Note that the notation $\overline{Z}^i_k$ appearing in the third and fourth last line above is used to denote the collection of elements of $G$ obtained by lifting a basis for $A_i$ into coordinate $k$ via the section map $s^i_k$ and setting the remaining $2n-1$ coordinates to be the identity. 

With this notation, the 5th equality above holds because $\grpproj_k(\overline{Z}^j_{i,j'}) = \grpproj_k(\overline{Z}^j_{j'})$, since $i \in A$ and $k \not\in A$ (and so $\grpproj_k(\overline{Z}^j_i) = \{e\}$), and also because $\grpproj_k(\overline{Z}^j_{j'}) = \{e\}$ unless $j' = k$, in which case $j= k'$ is uniquely determined by $k$ via the formula above. Therefore,
    \[
    \bigcup_{j=1}^n \bigcup_{j'} \grpproj_k \left( \overline{Z}^j_{i,j'} \right) \;=\; \bigcup_{j=1}^n \bigcup_{j'} \grpproj_k \left( \overline{Z}^{k'}_k \right) \;=\; \grpproj_k \left( \overline{Z}^{k'}_k \right) . \qedhere
    \]
\end{proof}

\subsection{The standard face subgroups of \texorpdfstring{$\Ker(\Phi)$}{the kernel}}\label{sub:subgroups.2}

We define thirteen standard face subgroups of $\Ker(\Phi)$. The standard arrangement $\mathcal{L}$ of $\Z^m$ given in Definition~\ref{def:standard1} is one of them; the remaining twelve standard subgroups are defined using the sets $\overline{Y}_i$ and $\overline{Z}^j_{i,k}$ that are compatible with $\mathcal{L}$. Here are the details.

\begin{rem}[Working with  indices modulo $2n$ and $n$; the $j^\ast$ notation]
Some formulae in this and following sections include lower indices outside of the range $\{1, \dots, 2n\}$; for compactness, such lower indices are always understood to be taken modulo $2n$. 

Furthermore, some expressions  require inputs between $1$ and $n$. For these expressions, we introduce the notation that $j^*$ is defined to be $j \!\pmod{n}$, where, as before, we  adopt the convention of using $n$ for the equivalence class of $0\!\pmod{n}$.
\end{rem}

\begin{defn}[Thirteen standard face subgroups of $\Ker(\Phi)$]\label{defn:stdsub}
The thirteen {\em standard face subgroups of $\Ker(\Phi)$} consist of the standard arrangement $\mathcal{L}$ of $\Z^m$, together with twelve subgroups which are indexed by label sequences $\alpha_\sigma$, where $\sigma \subseteq \{1, \dots, 6\}$ may be a triple or a pair.

Let $\sigma$ be a cyclically convex subset of $\{1, \dots, 6\}$ with $|\sigma| \in \{2,3\}$. The remaining twelve standard face subgroups of $\Ker(\Phi)$ are defined in terms of $\overline{Y}$ and $\overline{Z}$ generators as follows, depending on whether $|\sigma| = 2$ or $|\sigma| = 3$.
    
If $|\sigma| = 2$, then let $t$ be the unique element of $\{1, 2, 3\} \cap (\sigma \pmod{3})^c$, and define
    \[
    A^t G^\sigma \;\coloneqq\; \left\langle \; \bigcup_{i \in \alpha_\sigma} \left( \overline{Y}_i \;\cup\; \bigcup_{j \in \alpha_\sigma^c} \overline{Z}^{j^*}_{i,j} \right) \right\rangle \, .
    \]
    
On the other hand, if $|\sigma| = 3$, then define
    \[
    G^\sigma \;=\; \left\langle \; \bigcup_{i \in \alpha_\sigma} \left( \overline{Y}_i \;\cup\; \bigcup_{j \in \alpha_\sigma^c} \overline{Z}^{j^*}_{i,j} \right) \right\rangle \, .
    \]
\end{defn}

\begin{rem}
As is shown in Proposition~\ref{thm:table_std}, this naming scheme evokes the isomorphism type of the standard face subgroups; specifically, when $|\sigma| = 2$, then
    \[
    A^t G^\sigma \cong A_{\alpha_t} \times \prod_{i \in \sigma} G_i \, ,
    \]
and when $|\sigma| = 3$, then
    \[
    G^\sigma \cong \prod_{i \in \sigma} G_i \, .
    \]
    
Note that within this scheme, one could also write $\mathcal{L} = A^{123}$ as the case where $\sigma$ is empty. We will continue to refer to this subgroup as $\mathcal{L}$ due to its function as ``linear algebra'', but realizing $\mathcal{L}$ as $A^{123}$ makes the following results intuitively clear. 
\end{rem}

\begin{exmp}
Choosing $\sigma = \{1, 2, 3\}$ or $\sigma = \{1, 2\}$, respectively, would give
    \begin{align*}
    G^{123} &= \left\langle \; \bigcup_{i \in \alpha_{123}} \left( \overline{Y}_i \;\cup\; \bigcup_{j \in \alpha_{456}} \overline{Z}^{j^*}_{i,j} \right) \right\rangle \, = \left\langle \; \bigcup_{i \in \alpha_{123}} \left( \overline{Y}_i \;\cup\; \bigcup_{j \in \alpha_{123}} \overline{Z}^{j}_{i,j+n} \right) \right\rangle \, \text{ and } \\
    A^3 G^{12} &= \left\langle \; \bigcup_{i \in \alpha_{12}} \left( \overline{Y}_i \;\cup\; \bigcup_{j \in \alpha_{3456}} \overline{Z}^{j^*}_{i,j} \right) \right\rangle \, = \left\langle \; \bigcup_{i \in \alpha_{12}} \left( \overline{Y}_i \;\cup\; \bigcup_{j \in \alpha_3} \overline{Z}^{j}_{i,j} \;\cup\; \bigcup_{j \in \alpha_{123}} \overline{Z}^j_{i,j+n} \right) \right\rangle \, .
    \end{align*}
\end{exmp}

\begin{lem}\label{lem:detstd}
The standard face subgroups of $G$ are all $A$-determinate for the sets $A$ specified by the following table and homomorphisms $f_j$ defined afterwards.
    \begin{longtblr}{
        colspec = {|c|c|c|c|},
        rowhead = 1,
    }
    \hline
    Subgroup    & Index set $A$                     & Subgroup      & Index set $A$ \\
    \hline \hline
    $G^{123}$   & $\alpha_{123}$                    & $G^{234}$     & $\alpha_{234}$ \\
    \hline
    $G^{126}$   & $\alpha_{126}$                    & $G^{345}$     & $\alpha_{345}$ \\
    \hline
    $G^{156}$   & $\alpha_{156}$                    & $G^{456}$     & $\alpha_{456}$ \\
    \hline
    $A^3G^{12}$    & $\alpha_{123}$ or $\alpha_{126}$  & $A^2G^{34}$      & $\alpha_{234}$ or $\alpha_{345}$ \\
    \hline
    $A^1G^{23}$    & $\alpha_{123}$ or $\alpha_{234}$  & $A^1G^{56}$      & $\alpha_{156}$ or $\alpha_{456}$ \\
    \hline
    $A^2G^{16}$    & $\alpha_{126}$ or $\alpha_{156}$  & $A^3G^{45}$      & $\alpha_{345}$ or $\alpha_{456}$ \\
    \hline
    \end{longtblr}
and
    \begin{equation}\label{eqn:Adet-homs}
    f_j(\overline{g}) \;=\; s^{j^*}_j \left( \laproj_{j^*} \left( -\displaystyle\sum_{i \in A} \varphi_i(g_i) \right) \right) \, .
    \end{equation}
for $j \in B$.
\end{lem}

\begin{proof}
It suffices to verify that each element of the standard generating sets for these subgroups satisfy Equation~\eqref{eqn:Adet-homs}. There are two cases.

(\emph{Case I: } $\overline{g} \in \overline{Y}_i$ for some $i \in A$). In this case, $\varphi_i(\overline{y}) = 0$ for all $1 \leq i \leq 2n$. In particular, $\sum_{i \in A} \varphi_i(g_i) = 0$, and so the right side of \eqref{eqn:Adet-homs} gives the element $e \in G_j$, which agrees with $\grpproj_j(\overline{y})$ because $j \not\in A$ and $\overline{y}$ entries are $e$ in the coordinates of $B$. (Indeed, they are $e$ for all but one entry in $A$, but we don't need this much.) 

(\emph{Case II: } $\overline{g} \in \overline{Z}_{j,i}^j$ for some $i \in A$, $j \in B$). In this case, $\overline{z} \in \overline{Z}^j_{i,j}$ has two coordinate entries which are not $e$; there is a copy of $s^{j^*}_i(b)$ in the $i$-factor and a copy of $s^{j^*}_j(-b)$ in the $j$-factor for some basis element $b \in B_{j^*}$ of the group $A_{j^*}$. Thus,
    \[
    \sum_{i \in A} \varphi_i(g_i) \;=\; 0 + \cdots + 0 + b
    \]
is a sum of $(n-1)$ zeros and a copy of $b$. Therefore, the first line of the right side of \eqref{eqn:Adet-nonstd} gives $s^{j^*}_j(-b)$, which agrees with $\grpproj_j(\overline{z})$.
\end{proof}

Here is the main result describing algebraic and combinatorial properties of  the standard face subgroups. 

\begin{prop}[Table of the standard face subgroups of $\Ker(\Phi)$]\label{thm:table_std}
The following table summarizes the isomorphism type, subgroup pattern, and linear algebra encoding for each of the thirteen standard face subgroups of $\Ker(\Phi)$.
    \begin{center}
    \begin{longtblr}{
        colspec = {|c|c|c|c|},
        rowhead = 1, rows={abovesep=5pt,belowsep=5pt},
    }
    \hline
    {Subgroup \\ Label} & {Isomorphism \\ Type} & {Subgroup \\ Pattern} &  {Linear \\ Algebra} \\
    \hline
    \hline
    $G^{123}$ & $G_{\alpha_1} \times G_{\alpha_2} \times G_{\alpha_3}$ & $(G_{\alpha_1}, G_{\alpha_2}, G_{\alpha_3}, A^{\alpha_1}_4, A^{\alpha_2}_5, A^{\alpha_3}_6)$ & $\begin{pmatrix} 1 & 2 & 3 \\ 4 & 5 & 6 \end{pmatrix}$ \\
    \hline
    $G^{126}$ & $G_{\alpha_6} \times G_{\alpha_1} \times G_{\alpha_2}$ & $(G_{\alpha_1}, G_{\alpha_2}, A^{\alpha_3}_{3}, A^{\alpha_1}_{4}, A^{\alpha_2}_{5}, G_{\alpha_6})$ & $\begin{pmatrix} 1 & 2 & 3 \\ 4 & 5 & 3 \end{pmatrix}$ \\
    \hline
    $G^{156}$ & $G_{\alpha_1} \times G_{\alpha_5} \times G_{\alpha_6}$ & $(G_{\alpha_1}, A^{\alpha_2}_2, A^{\alpha_3}_{3}, A^{\alpha_1}_4, G_{\alpha_5}, G_{\alpha_6})$ & $\begin{pmatrix} 1 & 2 & 3 \\ 4 & 2 & 3 \end{pmatrix}$ \\
    \hline
    $G^{234}$ & $G_{\alpha_2} \times G_{\alpha_3} \times G_{\alpha_4}$ & $(A^{\alpha_1}_{1}, G_{\alpha_2}, G_{\alpha_3}, G_{\alpha_4}, A^{\alpha_2}_{5}, A^{\alpha_3}_{6})$ & $\begin{pmatrix} 1 & 2 & 3 \\ 1 & 5 & 6 \end{pmatrix}$ \\
    \hline
    $G^{345}$ & $G_{\alpha_3} \times G_{\alpha_4} \times G_{\alpha_5}$ & $(A^{\alpha_1}_1, A^{\alpha_2}_2, G_{\alpha_3}, G_{\alpha_4}, G_{\alpha_5}, A^{\alpha_3}_6)$ & $\begin{pmatrix} 1 & 2 & 3 \\ 1 & 2 & 6 \end{pmatrix}$ \\
    \hline
    $G^{456}$ & $G_{\alpha_4} \times G_{\alpha_5} \times G_{\alpha_6}$ & $(A^{\alpha_1}_1, A^{\alpha_2}_2, A^{\alpha_3}_{3}, G_{\alpha_4},  G_{\alpha_5}, G_{\alpha_6})$ & $\begin{pmatrix} 1 & 2 & 3 \\ 1 & 2 & 3 \end{pmatrix}$ \\
    \hline
    $A^3G^{12}$ & $G_{\alpha_1} \times G_{\alpha_2} \times A_{\alpha_3}$ & $(G_{\alpha_1}, G_{\alpha_2}, A^{\alpha_3}_3, A^{\alpha_1}_4, A^{\alpha_2}_5, A^{\alpha_3}_6)$ & $\begin{pmatrix} 1 & 2 & 3 \\ 1 & 2 & 3/6 \end{pmatrix}$ \\
    \hline
    $A^2G^{16}$ & $A_{\alpha_2} \times G_{\alpha_1} \times G_{\alpha_6}$ & $(G_{\alpha_1}, A^{\alpha_2}_2, A^{\alpha_3}_3, A^{\alpha_1}_4, A^{\alpha_2}_5, G_{\alpha_6})$ & $\begin{pmatrix} 1 & 2 & 3 \\ 4 & 2/5 & 3 \end{pmatrix}$ \\
    \hline
    $A^1G^{23}$ & $A_{\alpha_1} \times G_{\alpha_2} \times G_{\alpha_3}$ & $(A^{\alpha_1}_1, G_{\alpha_2}, G_{\alpha_3}, A^{\alpha_1}_4, A^{\alpha_2}_5, A^{\alpha_3}_6)$ & $\begin{pmatrix} 1 & 2 & 3 \\ 1/4 & 5 & 6 \end{pmatrix}$ \\
    \hline
    $A^2G^{34}$ & $A_{\alpha_2} \times G_{\alpha_3} \times G_{\alpha_4}$ & $(A^{\alpha_1}_1, A^{\alpha_2}_2, G_{\alpha_3}, G_{\alpha_4}, A^{\alpha_2}_5, A^{\alpha_3}_6)$ & $\begin{pmatrix} 1 & 2 & 3 \\ 1 & 2/5 & 6 \end{pmatrix}$ \\
    \hline
    $A^3G^{45}$ & $A_{\alpha_3} \times G_{\alpha_4} \times G_{\alpha_5}$ & $(A^{\alpha_1}_1, A^{\alpha_2}_2, A^{\alpha_3}_3, G_{\alpha_4}, G_{\alpha_5}, A^{\alpha_3}_6)$ & $\begin{pmatrix} 1 & 2 & 3 \\ 1 & 2 & 3/6 \end{pmatrix}$ \\
    \hline
    $A^1G^{56}$ & $G_{\alpha_5} \times G_{\alpha_6} \times A_{\alpha_1}$ & $(A^{\alpha_1}_1, A^{\alpha_2}_2, A^{\alpha_3}_{3}, A^{\alpha_1}_4, G_{\alpha_5}, G_{\alpha_6})$ & $\begin{pmatrix} 1 & 2 & 3 \\ 1/4 & 2 & 3 \end{pmatrix}$ \\
    \hline
    $\mathcal{L}$ & $\Z^m$ & $(A^{\alpha_1}_1, A^{\alpha_2}_2, A^{\alpha_3}_{3}, A^{\alpha_1}_4, A^{\alpha_2}_5, A^{\alpha_3}_6)$ & $\begin{pmatrix} 1 & 2 & 3 \\ 1/4 & 2/5 & 3/6 \end{pmatrix}$ \\
    \hline
    \end{longtblr}
    \end{center}
\end{prop}

\begin{proof}
Note that the {\em linear algebra encoding} simply records the positions of the $A^{\alpha_i}$ subgroups in the given subgroup pattern, so the entries in the third column of the table follow immediately from the information in the second column and the notation of  Definition~\ref{def:lin_algebra}.

({\em Row 13}): We start with the last row of the table since some of work is already done in Section~\ref{sec:alg-setup}. We have already established the isomorphism type of $\mathcal{L}$ in Lemma~\ref{lem:std-Zm}. The isomorphism in that proof is given by a restriction of the projection map $\grpproj_{\alpha_{123}}$. Instead, one could have taken $A = \alpha_{456}$ and $B = \alpha_{123}$ and obtained an isomorphism between $\mathcal{L}$ and $\Z^m$ via restriction of the projection $\grpproj_{\alpha_{456}}$. By composing the two projections above with projections onto the product of $\alpha_i$ factors for $1 \leq i \leq 6$, we establish the given subgroup pattern. 

({\em Rows 1 through 6}): Next, consider the subgroup $H = G^{123}$ of the first row of the table. It is generated by 
    \[
    \bigcup_{i\in \alpha_{123}} \left( \overline{Y}_i \;\cup\; \bigcup_{j=1}^n \overline{Z}^j_{i,j'} \right)
    \]
where $j'$ is the unique element of $\{j,j+n\} \cap \alpha_{123}^c$. Let $A = \alpha_{123}$ and $B = \alpha_{123}^c$. These generators are $A$-determinate with respect to the collection of homomorphisms
    \[
    f_{j'} \;=\; s^j_{j'} \left( -\laproj_j \left( \sum_{i \in \alpha_{123}} \varphi_i \right) \right)
    \]
for $1 \leq j \leq n$ (so that $j' = j+n$ belongs to the index set $\alpha_{123}^c = B$). By Lemma~\ref{lem:detgen}, the subgroup $H = G^{123}$ is $A$-determinate. By Lemma~\ref{lem:detinj}, the projection map $\grpproj_{\alpha_{123}}|_H$ is injective. For each $1 \leq i \leq n$, the generating set $\overline{Y}_i \cup \bigcup_{j=1}^n \overline{Z}^j_{i,j'}$ gives a copy of the kernel-section generating set 
    \[
    Y_i \cup Z^1_i \cup \cdots \cup Z^n_i
    \]
of $G_i$ in the $i$th factor. Therefore, $\grpproj_{\alpha_{123}}|_H$ is also surjective and so is an isomorphism of $H$ with $G_1 \times \cdots \times G_n = G_{\alpha_1} \times G_{\alpha_2} \times G_{\alpha_3}$. 

From the paragraph above, we conclude that the first three entries of the pattern of $H$ are $G_{\alpha_1}$, $G_{\alpha_2}$, and $G_{\alpha_3}$. Finally, Lemma~\ref{lem:Bpatterns} (with $A = \alpha_{123}$ and $B = \alpha_{123}^c$) establishes the $A^{\alpha_1}_4, A^{\alpha_2}_5, A^{\alpha_3}_6$ portions of the pattern, and so the first row is proven. 

The proofs for the other triple superscript groups work the same way. In the general case of the group $G^{pqr}$, one works with $A = \alpha_{pqr}$, $B = \alpha_{pqr}^c$ and $j'$ being the unique element of $\{j, j+n\} \cap \alpha_{pqr}^c$. 

({\em Rows 7 through 12}): Now, consider the subgroup $H = A^3G^{12}$ on line 7 of the table. It is generated by 
    \[
    \bigcup_{i \in \alpha_{12}} \left( \overline{Y}_i \;\cup\; \bigcup_{j=1}^n \bigcup_{j'} \overline{Z}^j_{i,j'} \right) 
    \]
where  $j' \in \{j, j+n\} \cap \alpha_{12}^c$. Note that this set is either a singleton (in the case $j \in \alpha_{12}$) or all of $\{j, j+n\}$ (in the case $j \in \alpha_3$). 

Choosing $A = \alpha_{123}$ (though one could also work with $A = \alpha_{126}$) and $B$ its complement, we see that all of these generators are uniformly $A$-determinate with respect to the collection of homomorphisms
    \[
    f_{j'} \;=\; s^j_{j'} \left( -\laproj_j \left( \sum_{i \in \alpha_{123}} \varphi_i \right) \right) ,
    \]
where $j'$ ranges over $B$ as $j$ ranges over $\{1, \ldots, n\}$. By Lemma~\ref{lem:detgen}, $H$ is $A$-determinate, and so by Lemma~\ref{lem:detinj}, the projection map $\grpproj_{\alpha_{123}}|_H$ is injective. 

For each $i \in \alpha_{12}$, the $\grpproj_{\alpha_{123}}|_H$-projection image of $\overline{Y}_i \cup \bigcup_{j=1}^n \overline{Z}^j_{i,j'}$ is the kernel-section generating set $Y_i \cup Z^1_i \cup \cdots \cup Z^n_i$ for $G_i$. Furthermore, the $\grpproj_{\alpha_{123}}|_H$-projection image of the generating sets $\overline{Z}_{i,j'}^j$ for $j'= j \in \alpha_3$ gives a generating set for $A_j$ in the $j$th coordinate. This gives a copy of $A_{\alpha_3}$ in the $\alpha_3$-coordinates. Therefore, $\grpproj_{\alpha_{123}}|_H$ gives a surjection to $G_{\alpha_1} \times G_{\alpha_2} \times A_{\alpha_3}$. This establishes the isomorphism type of $A^3G^{12}$. 

Projecting the isomorphic group $G_{\alpha_1} \times G_{\alpha_2} \times A_{\alpha_3}$ onto the $\alpha_i$-factors for $i=1,2,3$ establishes half of the pattern for $H$. The $B$-patterns Lemma~\ref{lem:Bpatterns} completes the pattern with $A^{\alpha_1}_4$, $A^{\alpha_2}_5$, and $A^{\alpha_3}_6$, establishing line 7 of the table. 

The other double superscript groups are handled similarly. One can use $A = \alpha_{234}$ (or $A = \alpha_{123}$) for the group $A^1G^{23}$; $A = \alpha_{345}$ (or $A = \alpha_{234}$) for the group $A^2G^{34}$; $A = \alpha_{456}$ (or $A = \alpha_{345}$) for the group $A^3G^{45}$; $A = \alpha_{561}$ (or $A = \alpha_{456}$) for the group $A^1G^{56}$; and $A = \alpha_{612}$ (or $A = \alpha_{561}$) for the group $A^2G^{61}$.  
\end{proof}

\subsection{The non-standard face subgroups of \texorpdfstring{$\Ker(\Phi)$}{the kernel}}\label{sub:subgroups.3}

Next, we define the twelve non-standard face subgroups of $\Ker(\Phi)$.  These definitions are inspired by the definitions of the subgroups associated to the corners of the algebraic square in \cite[Section~4]{KLI}. In contrast to the organization of the definition of standard face subgroups in Definition~\ref{defn:stdsub}, we group these definitions according to their common index sets $A$.

\begin{defn}[Twelve non-standard face subgroups of $\Ker(\Phi)$]\label{defn:nonstdsub}
The $N$-subscript is used to denote non-standard subgroups of $\Ker(\Phi)$.  

There are twelve non-standard face subgroups of $\Ker(\Phi)$: three of these are the groups $G^{14}_N$, $G^{25}_N$, and $G^{36}_N$; another six are the groups $A^iG^i_N$ and $A^iG^{i+3}_N$ for each $1 \leq i \leq 3$; and the final three are the {\em nonstandard linear algebra} subgroups $A^1A^1_N$, $A^2A^2_N$, and $A^3A^3_N$. We define these subgroups based on their generating sets as follows, grouped by the value of their indices $\!\!\pmod{3}$.

{\em Group 1}, where the indices are all equivalent to $1 \pmod{3}$:
    \begin{align*}
    G^{14}_N \;&=\; \left\langle \; \bigcup_{i\in\alpha_{14}} \left( \overline{Y}_i \;\cup\; \bigcup_{j\in\alpha_{23}} \overline{Z}^j_{i,j}  \,\cup\, \bigcup_{j\in\alpha_1} \overline{Z}^j_{i,j+|\alpha_{1234}|} \right) \; \right\rangle \, , \\
    A^1G^1_N \;&=\; \left\langle \; \bigcup_{i\in\alpha_1} \left( \overline{Y}_i \;\cup\; \bigcup_{j\in\alpha_{234}} \overline{Z}^{j^*}_{i,j}  \,\cup\, \bigcup_{j\in\alpha_1} \overline{Z}^j_{i,j+|\alpha_{1234}|} \right) \; \right\rangle \, , \\
    A^1G^4_N \;&=\; \left\langle \; \bigcup_{i\in\alpha_4} \left( \overline{Y}_i \;\cup\; \bigcup_{j\in\alpha_{234}} \overline{Z}^{j^*}_{i,j}  \,\cup\, \bigcup_{j\in\alpha_1} \overline{Z}^j_{i,j+|\alpha_{1234}|} \right) \; \right\rangle \, , \text{ and } \\
    A^{11}_N \;&=\; \left\langle \bigcup_{j\in\alpha_1} \left( \overline{Z}^j_{j,j+n} \;\cup\; \overline{Z}^j_{j,j+|\alpha_{1234}|} \right) \right\rangle \, .
    \end{align*}

{\em Group 2}, where the indices are all equivalent to $2 \pmod{3}$:
    \begin{align*}
    G^{25}_N \;&=\; \left\langle \; \bigcup_{i\in\alpha_{25}} \left( \overline{Y}_i \;\cup\; \bigcup_{j\in\alpha_{16}} \overline{Z}^{j^*}_{i,j}  \;\cup\; \bigcup_{j\in\alpha_2} \overline{Z}^j_{i,j+|\alpha_2|} \right) \; \right\rangle \, , \\
    A^2G^2_N \;&=\; \left\langle \; \bigcup_{i\in\alpha_2} \left( \overline{Y}_i \;\cup\; \bigcup_{j\in\alpha_{156}} \overline{Z}^{j^*}_{i,j} \;\cup\; \bigcup_{j\in\alpha_2} \overline{Z}^j_{i,j+|\alpha_2|} \right) \; \right\rangle \, , \\
    A^2G^5_N \;&=\; \left\langle \; \bigcup_{i\in\alpha_5} \left( \overline{Y}_i \;\cup\; \bigcup_{j\in\alpha_{156}} \overline{Z}^{j^*}_{i,j} \;\cup\; \bigcup_{j\in\alpha_2} \overline{Z}^j_{i,j+|\alpha_2|} \right) \; \right\rangle \, , \text{ and } \\
    A^{22}_N \;&=\; \left\langle \bigcup_{j \in \alpha_2} \left( \overline{Z}^j_{j,j+n} \;\cup\; \overline{Z}^j_{j,j+|\alpha_2|} \right) \right\rangle \, .
    \end{align*}

{\em Group 3}, where the indices are all equivalent to $3 \pmod{3}$:
    \begin{align*}
    G^{36}_N \;&=\; \left\langle \; \bigcup_{i \in \alpha_{36}} \left( \overline{Y}_i \;\cup\; \bigcup_{j\in\alpha_{45}} \overline{Z}^{j^*}_{i,j} \;\cup\; \bigcup_{j \in \alpha_3} \overline{Z}^j_{i,j-|\alpha_{12}|} \right) \; \right\rangle \, , \\
    A^3G^3_N \;&=\; \left\langle \bigcup_{i \in \alpha_3} \left( \overline{Y}_i \;\cup\; \bigcup_{j\in\alpha_{456}} \overline{Z}^{j^*}_{i,j}  \;\cup\; \bigcup_{j \in \alpha_3} \overline{Z}^j_{i,j-|\alpha_{12}|} \right) \right\rangle \, , \\
    A^3G^6_N \;&=\; \left\langle \bigcup_{i\in \alpha_6} \left( \overline{Y}_i \;\cup\; \bigcup_{j\in\alpha_{456}} \overline{Z}^{j^*}_{i,j}  \;\cup\; \bigcup_{j \in \alpha_3} \overline{Z}^j_{i,j-|\alpha_{12}|} \right) \right\rangle \, , \text{ and } \\
    A^{33}_N \;&=\; \left\langle \bigcup_{j \in \alpha_3} \left( \overline{Z}^j_{j,j+n} \;\cup\; \overline{Z}^j_{j,j-|\alpha_{12}|} \right) \right\rangle \, .
    \end{align*}
\end{defn}

Next, we establish the $A$-determinacy of the following non-standard subgroups. This is done in separate cases for each of Groups 1, 2, and 3, with each case having a common index set $A$. 

\begin{lem}\label{lem:detnstd}
\begin{subequations}\label{eqn:Adet-nonstd}
The subgroups $G^{14}_N$, $A^1G^1_N$, $A^1G^4_N$, and $A^{11}_N$ are $A$-determinate for
    \[
    A \;=\; \alpha_{14} \,\cup\, \alpha_{56} [ |\alpha_1|+1,|\alpha_{56}| ]
    \]
and homomorphisms $f_j$ ($1 \leq j \leq 2n$) given by
    \begin{equation}
    f_j(\overline{g}) \;=\; 
        \begin{cases}
        s^{j-|\alpha_2|}_j \left( \laproj_{j-|\alpha_2|} \left( -\sum_{i \in A} \varphi_i(g_i) \right)\right) & j \in \alpha_{34} [ 1, |\alpha_2| ] \\
        s^{j^*}_j \left( \laproj_{j^*} \left( -\sum_{i \in A} \varphi_i(g_i) \right)\right) & j \in \alpha_{16}
        \end{cases} \, . \label{eqn:Adet-nonstd-one}
    \end{equation}

Similarly, $G^{25}_N$, $A^2G^2_N$, $A^2G^5_N$, and $A^{22}_N$ are $A$-determinate for  
    \[
    A \;=\; \alpha_{25} \,\cup\, \alpha_{34} [ |\alpha_2|+1,|\alpha_{34}| ]
    \]
and
    \begin{equation}
    f_j(\overline{g}) \;=\;
        \begin{cases}
        s^{j-n-|\alpha_4|}_j \left( \laproj_{j-n-|\alpha_4|} \left( -\sum_{i \in A} \varphi_i(g_i) \right)\right) & j \in \alpha_{56} [ 1, |\alpha_1| ] \\
        s^j_j \left( \laproj_j \left( -\sum_{i \in A} \varphi_i(g_i) \right)\right) & j \in \alpha_{23}
        \end{cases} \, . \label{eqn:Adet-nonstd-two}
    \end{equation}

Finally, $G^{36}_N$, $A^3G^3_N$, $A^3G^6_N$, and $A^{33}_N$ are $A$-determinate for
    \[
    A \;=\; \alpha_{36} \,\cup\, \alpha_{12} [ |\alpha_{12}|, |\alpha_3|+1 ]
    \]
and
    \begin{equation}
    f_j(\overline{g}) \;=\;
        \begin{cases}
        s^{j+|\alpha_{12}|}_j \left( \laproj_{j+|\alpha_{12}|} \left( -\sum_{i \in A} \varphi_i(g_i) \right)\right) & j \in \alpha_{12} [ 1, |\alpha_3| ] \\  
        s^{j-n}_j \left( \laproj_{j-n} \left( -\sum_{i \in A} \varphi_i(g_i) \right)\right) & j \in \alpha_{45}
        \end{cases} \, . \label{eqn:Adet-nonstd-three}
    \end{equation}
\end{subequations}
\end{lem}

\begin{proof}
For the duration of this proof, we will refer to any of Equations~\eqref{eqn:Adet-nonstd-one}, ~\eqref{eqn:Adet-nonstd-two}, or ~\eqref{eqn:Adet-nonstd-three} as Equation~\eqref{eqn:Adet-nonstd}, depending on which equation would make sense in context. With that in mind, there are two cases to consider.
    
(\emph{Case I: } $\overline{g} \in \overline{Y}_i$ for some $i \in A$). In this case, $\varphi_i(\overline{g}) = 0$ for all $1 \leq i \leq 2n$. In particular, $\sum_{i \in A} \varphi_i(g_i) = 0$, and so the right side of~\eqref{eqn:Adet-nonstd} gives the element $e \in G_j$, which agrees with $f_j(\overline{g})$ because $j \not\in A$ and $\overline{Y}$ entries are $e$ on $B$.
        
(\emph{Case II: } $\overline{g} \in \overline{Z}_{j,i}^{j^*}$ for some $i \in A$, $j \in B$). Note that $\overline{z} \in \overline{Z}^{j^*}_{i,j}$ has two coordinate entries which are not $e$; there is (for some $1 \leq k \leq n$, determined by which of the three cases of the lemma we are in) a copy of $s^k_i(b)$ in the $i$-factor and a copy of $s^k_j(-b)$ in the $j$-factor, for some basis element $b \in B_j$ of the group $A_j$. Thus,
    \[
    \sum_{i \in A} \varphi_i(g_i) \; = \; 0 + \cdots + 0 + b
    \]
is a sum of $(n-1)$ zeros and a copy of $b$. Therefore, the first line of the right side of~\eqref{eqn:Adet-nonstd} gives $s^k_j(-b)$, which agrees with $\grpproj_j(\overline{g})$.
\end{proof}

\begin{lem}[$B$-pattern for nonstandard face subgroups]\label{lem:BpatternsNS} 
The following table summarizes the $B$-patterns for each of the nonstandard face subgroups of $\Ker(\Phi)$ for $B = B_1 \sqcup B_2$ as indicated.
    \begin{center}
    \begin{longtblr}{
        colspec = {|c|c|c|c|c|},
        rowhead = 1, rows={abovesep=5pt,belowsep=5pt},
    }
    \hline
    Subgroup $H$ & $B_1$ & $\grpproj_k(H)$, $k \in B_1$ & $B_2$ & $\grpproj_k(H)$, $k \in B_2$ \\
    \hline \hline
    $G^{14}_N$, $A^1G^1_N$, $A^1G^4_N$ & $\alpha_{23}$ & $A_k$ & $\alpha_{56} [ 1,|\alpha_1| ]$ & $A_{k-n-|\alpha_1|}$ \\
    \hline
    $A^{11}_N$ & $\alpha_{23}$ & $\{e\}$ & $\alpha_{56} [ 1,|\alpha_1| ]$ & $A_{k-n-|\alpha_1|}$ \\
    \hline 
    $G^{25}_N$, $A^2G^2_N$, $A^2G^5_N$ & $\alpha_{16}$ & $A_{k^*}$ & $\alpha_{34} [ 1,|\alpha_2| ]$ & $A_{k+|\alpha_2|}$ \\
    \hline
    $A^{22}_N$ & $\alpha_{16}$ & $\{e\}$ & $\alpha_{34} [ 1,|\alpha_2| ]$ & $A_{k+|\alpha_2|}$ \\
    \hline
    $G^{36}_N$, $A^3G^3_N$, $A^3G^6_N$ & $\alpha_{45}$ & $A_{k-n}$ & $\alpha_{12} [ 1,|\alpha_3| ]$ & $A_{k+n-|\alpha_3|}$ \\
    \hline
    $A^{33}_N$ & $\alpha_{45}$ & $\{e\}$ & $\alpha_{12} [ 1,|\alpha_3| ]$ & $A_{k+n-|\alpha_3|}$ \\
    \hline
    \end{longtblr}
    \end{center}
\end{lem} 

\begin{proof}
From Lemma~\ref{lem:detnstd}, we know the index set $A$ of $G^{36}_N$ is
    \[
    A \;=\; \alpha_{36} \,\cup\, \alpha_{12} [ |\alpha_{12}|,|\alpha_3|+1 ] \, .
    \]
Furthermore, for $i \in A$, $\grpproj_k(\overline{Y}_i) = \{e\}$ by the definition of the elements of $\overline{Y}_i$ (and since $k \not\in A$).

For $k \in \alpha_{45}$, we have 
    \begin{align*}
    \grpproj_k(G^{36}_N) \;&=\; \grpproj_k \left( \left\langle \; \bigcup_{i \in \alpha_{36}} \left( \overline{Y}_i \;\cup\; \bigcup_{j\in\alpha_{12}} \overline{Z}^j_{i,j+n}  \,\cup\, \bigcup_{j \in \alpha_3} \overline{Z}^j_{i,j+ |\alpha_3|-n} \right) \; \right\rangle \right) \\
    &=\; \left\langle \grpproj_k \left( \bigcup_{i \in \alpha_{36}} \left( \overline{Y}_i \;\cup\; \bigcup_{j\in\alpha_{12}} \overline{Z}^j_{i,j+n} \,\cup\, \bigcup_{j \in \alpha_3} \overline{Z}^j_{i,j+|\alpha_3|-n} \right) \right) \right\rangle \\
    &=\; \left\langle \grpproj_k\left( \bigcup_{i \in \alpha_{36}} \bigcup_{j\in\alpha_{12}} \overline{Z}^j_{i,j+n} \,\cup\, \bigcup_{j \in \alpha_3} \overline{Z}^j_{i,j+|\alpha_3|-n} \right) \right\rangle {\hbox{ (as $\grpproj_k(\overline{Y}_i) = \{e\}$ for all $i \in \alpha_{36}$)}} \\
    &=\; \left\langle \bigcup_{i \in \alpha_{36}} \bigcup_{j \in \alpha_{12}} \grpproj_k \left( \overline{Z}^j_{i,j+n} \right) \right\rangle {\hbox{ (since $\alpha_{45} \cap \{j+|\alpha_3|-n\} = \emptyset$ for all $j \in \alpha_3$)}} \\
    &=\; \left\langle \bigcup_{i \in \alpha_{36}} \grpproj_k \left( \overline{Z}^{k-n}_k \right) \right\rangle \\
    &=\; \left\langle \grpproj_k \left( \overline{Z}^{k-n}_k \right) \right\rangle \\
    &=\; A_{k-n} \, ,
    \end{align*}
where $\overline{Z}^i_k$ is as defined in the proof of Lemma~\ref{lem:Bpatterns}. 

Finally, for $k \in \alpha_{12} [ 1,|\alpha_3| ]$ we have
    \begin{align*}
    \grpproj_k(G^{36}_N) \;&=\; \grpproj_k \left( \left\langle \; \bigcup_{i \in \alpha_{36}} \left( \overline{Y}_i \;\cup\; \bigcup_{j\in\alpha_{12}} \overline{Z}^j_{i,j+n} \,\cup\, \bigcup_{j \in \alpha_3} \overline{Z}^j_{i,j+|\alpha_3|-n} \right) \; \right\rangle \right) \\
    &=\; \left\langle \grpproj_k \left( \bigcup_{i \in \alpha_{36}} \left( \overline{Y}_i \;\cup\; \bigcup_{j\in\alpha_{12}} \overline{Z}^j_{i,j+n} \,\cup\, \bigcup_{j \in \alpha_3} \overline{Z}^j_{i,j+|\alpha_3|-n} \right) \right) \right\rangle \\
    &=\; \left\langle \grpproj_k \left( \bigcup_{i \in \alpha_{36}} \bigcup_{j\in\alpha_{12}} \overline{Z}^j_{i,j+n} \,\cup\, \bigcup_{j \in \alpha_3} \overline{Z}^j_{i,j+|\alpha_3|-n} \right) \right\rangle {\hbox{ (since $\grpproj_k(\overline{Y}_i) = \{e\}$ for $i \in \alpha_{36}$)}} \\
    &=\; \left\langle \bigcup_{i \in \alpha_{36}} \bigcup_{j \in \alpha_{3}} \grpproj_k \left( \overline{Z}^j_{i,j+|\alpha_3|-n} \right) \right\rangle {\hbox{ (since $\alpha_{12} [ 1,|\alpha_3| ] \;\cap\; \{j+n\} = \emptyset$ for $j \in \alpha_{3}$)}} \\
    &=\; \left\langle \bigcup_{i \in \alpha_{36}} \grpproj_k \left( \overline{Z}^{k+n-|\alpha_3|}_k \right) \right\rangle \\
    &=\; \left\langle \grpproj_k \left( \overline{Z}^{k+n-|\alpha_3|}_k \right) \right\rangle \\
    &=\; A_{k+n-|\alpha_3|} \, .
    \end{align*}

Since the index sets do not change for $A^3G^3_N$ and $A^3G^6_N$, then for any $k \in \alpha_{45}$,
    \[
    \grpproj_k(A^3G^{3}_N) = \grpproj_k(A^3G^6_N) = A_{k-n} \, ,
    \]
and for any $k \in \alpha_{12} [ 1,|\alpha_3| ]$,
    \[
    \grpproj_k(A^3G^{3}_N) = \grpproj_k(A^3G^6_N) = A_{k+n-|\alpha_3|} \, .
    \]

By varying the index sets, similar arguments show that:
    \begin{itemize}[nolistsep]
    \item for any $k \in \alpha_{23}$,
        \[
        \grpproj_k(G^{14}_N) = \grpproj_k(A^1G^{1}_N) = \grpproj_k(A^1G^4_N) = A_{k} \, ,
        \]
    \item for any $k \in \alpha_{56} [ 1,|\alpha_1| ]$,
        \[
        \grpproj_k(G^{14}_N) = \grpproj_k(A^1G^1_N) = \grpproj_k(A^1G^4_N) = A_{k-n-|\alpha_1|} \, ,
        \]
    \item for any $k \in \alpha_{16}$,
        \[
        \grpproj_k(G^{25}_N) = \grpproj_k(A^2G^{2}_N) = \grpproj_k(A^2G^5_N) = A_{k^*} \, ,
        \]
    \item and for any $k \in \alpha_{34} [ 1,|\alpha_2| ]$,
        \[
        \grpproj_k(G^{25}_N) = \grpproj_k(A^2G^{2}_N) = \grpproj_k(A^2G^5_N) = A_{k-|\alpha_2|} \, .
        \]
    \end{itemize}

It now only remains to show the $B$-patterns for $A^{11}_N$, $A^{22}_N$, and $A^{33}_N$. From Lemma~\ref{lem:detnstd}, we know the index set $A$ of $A^{33}_N$ is 
    \[
    A \;=\; \alpha_{36} \,\cup\, \alpha_{12} [ |\alpha_{12},|\alpha_3|+1 ] \, ,
    \]
with complement
    \[
    B \;=\; \alpha_{45} \,\cup\, \alpha_{12} [ 1,|\alpha_3| ] \, .
    \]
        
For $k \in \alpha_{45}$, $\grpproj_k(A^{33}_N) = \{e\}$ since there are no $\overline{Z}$ sets in the $\alpha_{45}$ coordinates, as seen in the definition of $A^{33}_N$ in Definition~\ref{defn:nonstdsub}.

Let $k \in \alpha_{12} [ 1,|\alpha_3| ]$. Then
    \begin{align*}
    \grpproj_k(A^{33}_N) \;&=\; \grpproj_k \left( \left\langle \bigcup_{j\in\alpha_3} \left( \overline{Z}^j_{j,j-|\alpha_{12}|} \;\cup\; \overline{Z}^j_{j,j+n} \right) \right\rangle \right) \\
    &=\; \grpproj_k \left( \left\langle \bigcup_{j\in\alpha_3} \left( \overline{Z}^j_{j,j-|\alpha_{12}|} \right) \right\rangle \right) \\
    &=\; \left\langle \grpproj_k\left(\overline{Z}^{k+|\alpha_{12}|}_k\right) \right\rangle \\
    &=\; A_{k+|\alpha_{12}|} \, .
    \end{align*}
Similarly to before, varying the index sets and using similar arguments will show that for any $k \in \alpha_{23}$, $\grpproj_k(A^{11}_N) = \{e\}$ and $k \in \alpha_{56} [ 1,|\alpha_1| ]$ yields $\grpproj_k(A^{11}_N) = A_{k-n-|\alpha_1|}$. Finally, for any $k \in \alpha_{16}$, $\grpproj_k(A^{22}_N) = \{e\}$ and $k \in \alpha_{34} [ 1,|\alpha_2| ]$ gives $\grpproj_k(A^{22}_N) = A_{k+|\alpha_2|}$.
\end{proof}

The following proposition is an analogue of Proposition~\ref{thm:table_std}. It describes the algebraic and combinatorial properties of the non-standard face subgroups of $\Ker(\Phi)$. 

\begin{prop}[Table of non-standard face subgroups of $\Ker(\Phi)$]\label{thm:table_nonstd}
The following table summarizes the isomorphism type, subgroup pattern, and linear algebra encoding for each of the twelve non-standard face subgroups of $\Ker(\Phi)$.

    \begin{center}
    \begin{longtblr}{
        colspec = {|c|c|c|c|},
        rowhead = 1, rows={abovesep=5pt,belowsep=5pt},
    }
    \hline
    {Subgroup \\ Label} & {Isomorphism \\ Type} & {Subgroup \\ Pattern} & {Linear \\ Algebra} \\
    \hline
    \hline
    $G^{14}_N$ & $G_{\alpha_1} \times G_{\alpha_4}$ & $(G_{\alpha_1}, A^{\alpha_2}_2, A^{\alpha_3}_{3}, G_{\alpha_4}, A^{\alpha_1}_5 E)$ & $\begin{pmatrix} \dot{1} & 2 & 3 \\ 5 & 2 & 3 \end{pmatrix}$ \\
    \hline
    $G^{25}_N$ & $G_{\alpha_2} \times G_{\alpha_5}$ & $(A^{\alpha_1}_{1}, G_{\alpha_2}, A^{\alpha_2}_3 E, G_{\alpha_5}, A^{\alpha_3}_{6})$ & $\begin{pmatrix} 1 & \dot{2} & 3 \\ 1 & 3 & 6 \end{pmatrix}$ \\
    \hline
    $G^{36}_N$ & $G_{\alpha_3} \times G_{\alpha_6}$ & $(A^{\alpha_3}_1 E, G_{\alpha_3}, A^{\alpha_1}_4, A^{\alpha_2}_5, G_{\alpha_6})$ & $\begin{pmatrix} 1 & 2 & \dot{3} \\ 4 & 5 & 1 \end{pmatrix}$ \\
    \hline
    $A^1G^{1}_N$ & $A_{\alpha_1} \times G_{\alpha_1}$ & $(G_{\alpha_1}, A^{\alpha_2}_2, A^{\alpha_3}_{3}, A^{\alpha_1}_4, A^{\alpha_1}_5 E)$ & $\begin{pmatrix} \dot{1} & 2 & 3 \\ 4/5 & 2 & 3 \end{pmatrix}$ \\
    \hline
    $A^1G^{4}_N$ & $A_{\alpha_1} \times G_{\alpha_4}$ & $(A^{\alpha_1}_1, A^{\alpha_2}_2, A^{\alpha_3}_{3}, G_{\alpha_4}, A^{\alpha_1}_5 E)$ & $\begin{pmatrix} \dot{1} & 2 & 3 \\ 1/5 & 2 & 3 \end{pmatrix}$ \\
    \hline
    $A^2G^{2}_N$ & $A_{\alpha_2} \times G_{\alpha_2}$ & $(A^{\alpha_1}_{1}, G_{\alpha_2}, A^{\alpha_2}_3 E, A^{\alpha_2}_{5}, A^{\alpha_3}_{6})$ & $\begin{pmatrix} 1 & \dot{2} & 3 \\ 1 & 3/5 & 6 \end{pmatrix}$ \\
    \hline  
    $A^2G^{5}_N$ & $A_{\alpha_2} \times G_{\alpha_5}$ & $(A^{\alpha_1}_1, A^{\alpha_2}_2, A^{\alpha_2}_3 E, G_{\alpha_5}, A^{\alpha_3}_6)$ & $\begin{pmatrix} 1 & \dot{2} & 3 \\ 1 & 2/3 & 6 \end{pmatrix}$ \\
    \hline
    $A^3G^{3}_N$ & $A_{\alpha_3} \times G_{\alpha_3}$ & $(A^{\alpha_3}_1 E, G_{\alpha_3}, A^{\alpha_1}_4, A^{\alpha_2}_5, A^{\alpha_3}_6)$ & $\begin{pmatrix} 1 & 2 & \dot{3} \\ 4 & 5 & 1/6 \end{pmatrix}$ \\
    \hline
    $A^3G^{6}_N$ & $A_{\alpha_3} \times G_{\alpha_6}$ & $(A^{\alpha_3}_1 E, A^{\alpha_3}_{3}, A^{\alpha_1}_{4}, A^{\alpha_2}_{5}, G_{\alpha_6})$ & $\begin{pmatrix} 1 & 2 & \dot{3} \\ 4 & 5 & 1/3 \end{pmatrix}$ \\
    \hline
    $A^{11}_N$ & $(A_{\alpha_1})^2$ & $(A^{\alpha_1}_1, E, E, A^{\alpha_1}_4, A^{\alpha_1}_5 E)$ & $\begin{pmatrix} \dot{1} \\ 1/4/5 \end{pmatrix}$ \\
    \hline
    $A^{22}_N$ & $(A_{\alpha_2})^2$ & $(E, A^{\alpha_2}_2, A^{\alpha_2}_3 E, A^{\alpha_2}_5, E)$ & $\begin{pmatrix} \dot{2} \\ 2/3/5 \end{pmatrix}$ \\
    \hline
    $A^{33}_N$ & $(A_{\alpha_3})^2$ & $(A^{\alpha_3}_1 E, A^{\alpha_3}_3, E, E, A^{\alpha_3}_6)$ & $\begin{pmatrix} \dot{3} \\ 1/3/6 \end{pmatrix}$ \\
    \hline
    \end{longtblr}
    \end{center}
\end{prop}

\begin{proof}
{\em (Rows 1 through 3.)} The group $G^{36}_N$ is generated by
    \[
    \left\langle \; \bigcup_{i \in \alpha_{36}} \left( \overline{Y}_i \;\cup\; \bigcup_{j\in\alpha_{12}} \overline{Z}^j_{i,j+n} \,\cup\, \bigcup_{j \in \alpha_3} \overline{Z}^j_{i,j+|\alpha_3|-n} \right) \; \right\rangle \, .
    \]
Taking $A \;=\; \alpha_{36} \,\cup\, \alpha_{12} [ |\alpha_3|+1,|\alpha_{12}| ]$ and $B \;=\; \alpha_{45} \,\cup\, \alpha_{12} [ 1,|\alpha_3| ]$, we see that the generators of $G^{36}_N$ are $A$-determinate with respect to the collection of homomorphisms $f_j$ defined on the right side of the equation below:
    \begin{equation}
    f_j(\overline{g}) \;=\;
        \begin{cases}
        s^{j+|\alpha_{12}|}_j \left( \laproj_{j+|\alpha_{12}|} \left( -\sum_{i \in A} \varphi_i(g_i) \right)\right) & j \in \alpha_{12} [ 1,|\alpha_3| ] \\
        s^{j-n}_j \left( \laproj_{j-n} \left( -\sum_{i \in A} \varphi_i(g_i) \right)\right) & j \in \alpha_{45}
        \end{cases} \, ,
    \end{equation}
as proven in Lemma~\ref{lem:detnstd}. By Lemma~\ref{lem:detgen}, the subgroup $G^{36}_N$ is $A$-determinate. By Lemma~\ref{lem:detinj}, the projection map $\grpproj_{A}|_{G^{36}_N}$ is injective. For each $i \in \alpha_{36}$, the generating set
    \[
    \overline{Y}_i \;\cup\; \bigcup_{j\in\alpha_{12}} \overline{Z}^j_{i,j+n} \,\cup\, \bigcup_{j \in \alpha_3} \overline{Z}^j_{i,j+|\alpha_3|-n}
    \]
gives a copy of the kernel-section generating set of $G_i$ in the $i$th factor. Therefore, $\grpproj_{A}|_{G^{36}_N}$ is surjective, and so is an isomorphism of $G^{36}_N$ with
    \[
    E_{\alpha_{12} [ |\alpha_3|+1,|\alpha_{12}| ] } \times G_{\alpha_3} \times G_{\alpha_6} \cong G_{\alpha_3} \times G_{\alpha_6} \, .
    \]
Finally, from the isomorphism above and Lemma~\ref{lem:detnstd}, we get the required projection pattern for $G^{36}_N$ as in the first row of the table.

Varying the index sets, we get the isomorphism type and subgroup patterns for the other two nonstandard double superscript subgroups, $G^{25}_N$ and $G^{14}_N$.

{\em (Rows 4 through 9.)} The group $A^3G^{3}_N$ is generated by
    \[
    \left\langle \; \bigcup_{i \in \alpha_{3}} \left( \overline{Y}_i \;\cup\; \bigcup_{j\in\alpha_{12}} \overline{Z}^j_{i,j+n} \,\cup\, \bigcup_{j \in \alpha_3} \overline{Z}^j_{i,j+|\alpha_3|-n} \right) \; \right\rangle \, .
    \]
Taking $A \;=\; \alpha_{36} \,\cup\, \alpha_{12} [ |\alpha_3|+1,|\alpha_{12}| ]$ and $B \;=\; \alpha_{45} \,\cup\, \alpha_{12} [ 1,|\alpha_3| ]$, we recall that the generators of $A^3G^{3}_N$ are $A$-determinate with respect to the collection of homomorphisms $f_j$ below,
    \begin{equation}
    f_j(\overline{g}) \;=\;
        \begin{cases}
        s^{j+|\alpha_{12}|}_j \left( \laproj_{j+|\alpha_{12}|} \left( -\sum_{i \in A} \varphi_i(g_i) \right)\right) & j \in \alpha_{12} [ 1,|\alpha_3| ] \\
        s^{j-n}_j \left( \laproj_{j-n} \left( -\sum_{i \in A} \varphi_i(g_i) \right)\right) & j \in \alpha_{45}
        \end{cases} \, ,
    \end{equation}
as proven in Lemma~\ref{lem:detnstd}. By Lemma~\ref{lem:detgen}, the subgroup $A^3G^{3}_N$ is $A$-determinate. By Lemma~\ref{lem:detinj}, the projection map $\grpproj_{A}|_{A^3G^{3}_N}$ is injective. For each $i \in \alpha_{3}$, the generating set
    \[
    \overline{Y}_i \;\cup\; \bigcup_{j\in\alpha_{12}} \overline{Z}^j_{i,j+n} \,\cup\, \bigcup_{j \in \alpha_3} \overline{Z}^j_{i,j+|\alpha_3|-n}
    \]
gives a copy of the kernel-section generating set of $G_i$ in the $i$th factor. For each $i \in \alpha_{6}$, the generating set
    \[
    \bigcup_{j \in \alpha_3} \overline{Z}^j_{i,j+|\alpha_3|-n}
    \]
gives a copy of $A_{i-n}$ in the ($j+|\alpha_3|-n$)-th factor. Therefore, $\grpproj_{A}|_{A^3G^3_N}$ is also surjective, and so is an isomorphism of $A^3G^{3}_N$ with
    \[
    E_{\alpha_{12} [ |\alpha_3|+1,|\alpha_{12}| ]} \times G_{\alpha_3} \times A_{\alpha_3} \cong G_{\alpha_3} \times A_{\alpha_3} \, .
    \]
Finally, from the isomorphism above and Lemma~\ref{lem:BpatternsNS}, we get the required projection pattern for $A^3G^{3}_N$ as in the fourth row of the table.

Varying the index sets we get the isomorphism type and subgroup patterns for the other five nonstandard double superscript subgroups, namely $A^1G^{1}_N$, $A^2G^{2}_N$, $A^1G^{4}_N$, $A^2G^{5}_N$ and $A^3G^{6}_N$.

{\em (Rows 10 through 12.)} The group $A^{33}_N$ is generated by
    \[
    \left\langle \bigcup_{j \in \alpha_3} \left(\overline{Z}^j_{j,j-|\alpha_{12}|} \;\cup\; \overline{Z}^j_{j,j+n} \right) \right\rangle \, .
    \]
Taking $A \;=\; \alpha_{36} \,\cup\, \alpha_{12} [ |\alpha_3|+1,|\alpha_{12}| ]$ and $B \;=\; \alpha_{45} \,\cup\, \alpha_{12} [ 1,|\alpha_3| ]$, we know that the generators of $A^{33}_N$ are $A$-determinate with respect to the collection of homomorphisms $f_j$ given by
    \begin{equation}
    f_j(\overline{g}) \;=\;
        \begin{cases}
        s^{j+|\alpha_{12}|}_j \left( \laproj_{j+|\alpha_{12}|} \left( -\sum_{i \in A} \varphi_i(g_i) \right)\right) & j \in \alpha_{12} [ 1,|\alpha_3| ] \\
        s^{j-n}_j \left( \laproj_{j-n} \left( -\sum_{i \in A} \varphi_i(g_i) \right)\right) & j \in \alpha_{45}
    \end{cases} \, ,
    \end{equation}
as proven in Lemma~\ref{lem:detnstd}. Then by Lemma~\ref{lem:detgen}, the subgroup $A^{33}_N$ is $A$-determinate. By Lemma~\ref{lem:detinj}, the projection map $\grpproj_{A}|_{A^{33}_N}$ is injective. For each $i \in \alpha_{3}$, the generating set
    \[
    \left\langle \bigcup_{j \in \alpha_3} \left(\overline{Z}^j_{j,j-|\alpha_{12}|} \;\cup\; \overline{Z}^j_{j,j+n} \right) \right\rangle
    \]
gives a copy of $A_i$ in the $i$th factor and for each $i \in \alpha_{6}$, the same generating set gives a copy of $A_{i-n}$ in the ($i-n$)-th factor. Therefore, $\grpproj_{A}|_{A^{33}_N}$ is also surjective, and so is an isomorphism of $A^{33}_N$ with
    \[
    E_{\alpha_{12} [ |\alpha_3|+1,|\alpha_{12}| ]} \times A_{\alpha_3} \times A_{\alpha_3} \cong A_{\alpha_3} \times A_{\alpha_3} \, .
    \]
Finally, from the isomorphism above and Lemma~\ref{lem:BpatternsNS}, we get the required projection pattern for $A^{33}_N$ as in the tenth row of the table.

Varying the index sets gives the isomorphism type and subgroup patterns for $A^{11}_N$ and $A^{22}_N$.
\end{proof}

\subsection{The standard and non-standard edge subgroups of \texorpdfstring{$\Ker(\Phi)$}{the kernel}}\label{sub:subgroups.4}

In the previous subsections we defined and investigated the twenty-five standard and non-standard face subgroups of $\Ker(\Phi)$. In this subsection we do the same for the twenty-one standard and non-standard edge subgroups. 

\begin{defn}[The standard and non-standard edge subgroups of $\Ker(\Phi)$]{\label{def:edgegroups}} There are four families of edge subgroups of $\Ker(\Phi)$.

The first family of edge groups consists of non-standard groups which are isomorphic to a given $G_{\alpha_j}$ for $j \in \{1, \ldots , 6\}$. We shall see in Section~\ref{sec:con-algtri} that the particular type of non-standard linear algebra in a given subgroup is inherited from that of the non-standard face group whose face contains the given edge. Meanwhile, here are the definitions. 
    \begin{align*}
    G^1_N \;&=\; \left\langle \; \bigcup_{i\in\alpha_1} \left( \overline{Y}_i \;\cup\; \bigcup_{j\in\alpha_{23}} \overline{Z}^{j^*}_{i,j}  \,\cup\, \bigcup_{j\in\alpha_1} \overline{Z}^j_{i,j+|\alpha_{1234}|} \right) \; \right\rangle \, , \\
    G^2_N \;&=\; \left\langle \; \bigcup_{i\in\alpha_2} \left( \overline{Y}_i \;\cup\; \bigcup_{j\in\alpha_{16}} \overline{Z}^{j^*}_{i,j} \;\cup\; \bigcup_{j\in\alpha_2} \overline{Z}^j_{i,j+|\alpha_2|} \right) \; \right\rangle \, , \\
    G^3_N \;&=\; \left\langle \bigcup_{i \in \alpha_3} \left( \overline{Y}_i \;\cup\; \bigcup_{j\in\alpha_{45}} \overline{Z}^{j^*}_{i,j}  \;\cup\; \bigcup_{j \in \alpha_3} \overline{Z}^j_{i,j-|\alpha_{12}|} \right) \right\rangle \, , \\
    G^4_N \;&=\; \left\langle \; \bigcup_{i\in\alpha_4} \left( \overline{Y}_i \;\cup\; \bigcup_{j\in\alpha_{23}} \overline{Z}^{j^*}_{i,j}  \,\cup\, \bigcup_{j\in\alpha_1} \overline{Z}^j_{i,j+|\alpha_{1234}|} \right) \; \right\rangle \, , \\
    G^5_N \;&=\; \left\langle \; \bigcup_{i\in\alpha_5} \left( \overline{Y}_i \;\cup\; \bigcup_{j\in\alpha_{16}} \overline{Z}^{j^*}_{i,j} \;\cup\; \bigcup_{j\in\alpha_2} \overline{Z}^j_{i,j+|\alpha_2|} \right) \; \right\rangle \, , \; \mathrm{and} \\
    G^6_N \;&=\; \left\langle \bigcup_{i\in \alpha_6} \left( \overline{Y}_i \;\cup\; \bigcup_{j\in\alpha_{45}} \overline{Z}^{j^*}_{i,j}  \;\cup\; \bigcup_{j \in \alpha_3} \overline{Z}^j_{i,j-|\alpha_{12}|} \right) \right\rangle \, .
    \end{align*}

The second family of edge groups consists of standard subgroups which are isomorphic to $G_{\alpha_i}$ for $i \in \{1, \ldots , 6\}$. Specifically,
    \begin{align*}
    G^1 \;&=\; \left\langle \bigcup_{i \in \alpha_1}\left(\overline{Y}_i \, \cup \, \bigcup_{j \in \alpha_{234}}\overline{Z}^{j^\ast}_{i,j} \right) \right\rangle \, , \\
    G^2 \;&=\; \left\langle \bigcup_{i \in \alpha_2}\left(\overline{Y}_i \, \cup \, \bigcup_{j \in \alpha_{156}}\overline{Z}^{j^\ast}_{i,j} \right) \right\rangle \, , \\
    G^3 \;&=\; \left\langle \bigcup_{i \in \alpha_3}\left(\overline{Y}_i \, \cup \, \bigcup_{j \in \alpha_{456}}\overline{Z}^{j^\ast}_{i,j} \right) \right\rangle \, , \\
    G^4 \;&=\; \left\langle \bigcup_{i \in \alpha_4}\left(\overline{Y}_i \, \cup \, \bigcup_{j \in \alpha_{123}}\overline{Z}^{j^\ast}_{i,j} \right) \right\rangle \, , \\
    G^5 \;&=\; \left\langle \bigcup_{i \in \alpha_5}\left(\overline{Y}_i \, \cup \, \bigcup_{j \in \alpha_{126}}\overline{Z}^{j^\ast}_{i,j} \right) \right\rangle \, , \; \mathrm{and} \\
    G^6 \;&=\; \left\langle \bigcup_{i \in \alpha_6}\left(\overline{Y}_i \, \cup \, \bigcup_{j \in \alpha_{345}}\overline{Z}^{j^\ast}_{i,j} \right) \right\rangle \, .
    \end{align*}

The third family of edge subgroups are copies of $A_{\alpha_i}$ for $i \in \{1,2,3\}$. They are non-standard subgroups whose elements convert standard linear algebra into a given non-standard from or vise versa. In the notation for these subgroups, the $\alpha_i$ is indicated by a superscript $i$ and the standard to non-standard conversion are included in parentheses. These are:
    \begin{align*}
    A^1_N(1,5) \;&=\; \left\langle \bigcup_{j \in \alpha_1} \overline{Z}^j_{j,j+|\alpha_{1234}|} \right\rangle \, , \\
    A^1_N(4,5) \;&=\; \left\langle \bigcup_{j \in \alpha_4} \overline{Z}^{j^*}_{j,j+|\alpha_4|} \right\rangle \, , \\
    A^2_N(2,3) \;&=\; \left\langle \bigcup_{j \in \alpha_2} \overline{Z}^j_{j,j+|\alpha_2|} \right\rangle \, , \\
    A^2_N(3,5) \;&=\; \left\langle \bigcup_{j \in \alpha_5} \overline{Z}^{j^*}_{j,j-|\alpha_{34}|} \right\rangle \;=\; \left\langle \bigcup_{j \in \alpha_5} \overline{Z}^{j^*}_{j,j+|\alpha_{1256}|} \right\rangle \, , \\
    A^3_N(1,3) \;&=\; \left\langle \bigcup_{j \in \alpha_3} \overline{Z}^j_{j,j-|\alpha_{12}|} \right\rangle \;=\; \left\langle \bigcup_{j \in \alpha_3} \overline{Z}^j_{j,j+|\alpha_{3456}|} \right\rangle \, , \; \mathrm{and} \\
    A^3_N(1,6) \;&=\; \left\langle  \bigcup_{j \in \alpha_6}\overline{Z}^{j^*}_{j,j-|\alpha_{12345}|} \right\rangle \;=\; \left\langle \bigcup_{j \in \alpha_6}\overline{Z}^{j^*}_{j,j+|\alpha_6|} \right\rangle \, . 
    \end{align*}

Finally, the fourth family of edge subgroups are also copies of $A_{\alpha_i}$ for $i \in \{1, 2, 3\}$; however, these are standard subgroups (in contrast to the subgroups in the third class) whose elements convert standard linear algebra between its two different forms. The notation is similar to that for the third class. For each $i \in \{1, 2, 3\}$, we have
    \[
    A^i(i,i+3) \;=\; \left\langle \bigcup_{j \in \alpha_i} \overline{Z}^j_{j,j+n} \right\rangle \, .
    \]
\end{defn}

We record the isomorphism type, subgroup pattern and linear algebra encoding of the edge subgroups defined in Definition~\ref{def:edgegroups} in the following proposition. 

\begin{prop}[Table of the edge subgroups of $\Ker(\Phi)$]\label{lem:edge-subgroups}
The following table summarizes the isomorphism type, subgroup pattern, and linear algebra encoding for each of the twenty-one standard and non-standard edge subgroups of $\Ker(\Phi)$
    \begin{center}
        \begin{longtblr}{
            colspec = {|c|c|c|c|},
            rowhead = 1, rows={abovesep=5pt,belowsep=5pt},
        }
        \hline
        {Subgroup \\ Label} & {Isomorphism \\ Type} & {Subgroup \\ Pattern} & {Linear \\ Algebra} \\
        \hline
        \hline
        $G^1$ & $G_{\alpha_1}$ & $(G_{\alpha_1}, A^{\alpha_2}_{2}, A^{\alpha_3}_3, A^{\alpha_1}_4, E, E )$ & $\begin{pmatrix} 1 & 2 & 3 \\ 4 & 2 & 3 \end{pmatrix}$ \\
        \hline
        $G^2$ & $G_{\alpha_2}$ & $(A^{\alpha_1}_{1}, G_{\alpha_2}, E, E, A^{\alpha_2}_5, A^{\alpha_3}_{6})$ & $\begin{pmatrix} 1 & 2 & 3 \\ 1 & 5 & 6 \end{pmatrix}$ \\
        \hline
        $G^3$ & $G_{\alpha_3}$ & $(E, E, G_{\alpha_3}, A^{\alpha_1}_4, A^{\alpha_2}_{5}, A^{\alpha_3}_6)$ & $\begin{pmatrix} 1 & 2 & 3 \\ 4 & 5 & 6 \end{pmatrix}$ \\
        \hline
        $G^4$ & $G_{\alpha_4}$ & $(A^{\alpha_1}_1, A^{\alpha_2}_2, A^{\alpha_3}_3, G_{\alpha_4}, E, E)$ & $\begin{pmatrix} 1 & 2 & 3 \\ 1 & 2 & 3 \end{pmatrix}$ \\
        \hline
        $G^5$ & $G_{\alpha_5}$ & $(A^{\alpha_1}_{1}, A^{\alpha_2}_2, E, E,  G_{\alpha_5}, A^{\alpha_3}_{6})$  & $\begin{pmatrix} 1 & 2 & 3 \\ 1 & 2 & 6 \end{pmatrix}$ \\
        \hline
        $G^6$ & $G_{\alpha_6}$ & $(E, E, A^{\alpha_3}_{3}, A^{\alpha_1}_4, A^{\alpha_2}_5, G_{\alpha_6})$ & $\begin{pmatrix} 1 & 2 & 3 \\ 4 & 5 & 3 \end{pmatrix}$ \\
        \hline  
        $G^{1}_N$ & $G_{\alpha_1}$ & $(G_{\alpha_1}, A^{\alpha_2}_2, A^{\alpha_3}_{3}, E, A^{\alpha_1}_5 E)$ & $\begin{pmatrix} \dot{1} & 2 & 3 \\ 5 & 2 & 3 \end{pmatrix}$ \\
        \hline
        $G^{2}_N$ & $G_{\alpha_2}$ & $(A^{\alpha_1}_1, G_{\alpha_2}, A^{\alpha_2}_3 E, E, A^{\alpha_3}_{6})$ & $\begin{pmatrix} 1 & \dot{2} & 3 \\ 1 & 3 & 6 \end{pmatrix}$ \\
        \hline
        $G^{3}_N$ & $G_{\alpha_3}$ & $(A^{\alpha_3}_1 E, G_{\alpha_3}, A^{\alpha_1}_4, A^{\alpha_2}_5, E)$ & $\begin{pmatrix} 1 & 2 & \dot{3} \\ 4 & 5 & 1 \end{pmatrix}$ \\
        \hline
        $G^4_N$ & $G_{\alpha_4}$ & $(E, A^{\alpha_2}_2, A^{\alpha_3}_{3}, G_{\alpha_4}, A^{\alpha_1}_5 E)$ & $\begin{pmatrix} \dot{1} & 2 & 3 \\ 5 & 2 & 3 \end{pmatrix}$ \\
        \hline
        $G^5_N$ & $G_{\alpha_5}$ & $(A^{\alpha_1}_1, E, A^{\alpha_2}_3 E, G_{\alpha_5}, A^{\alpha_3}_6)$ & $\begin{pmatrix} 1 & \dot{2} & 3 \\ 1 & 3 & 6 \end{pmatrix}$ \\
        \hline
        $G^6_N$ & $G_{\alpha_6}$ & $(A^{\alpha_3}_1 E, E, A^{\alpha_1}_{4}, A^{\alpha_2}_{5}, G_{\alpha_6})$ & $\begin{pmatrix} 1 & 2 & \dot{3} \\ 4 & 5 & 1 \end{pmatrix}$ \\
        \hline
        $A^1_N(1,5)$ & $A_{\alpha_1}$ & $(A^{\alpha_1}_1, E, E, E, A^{\alpha_1}_5 E)$ & $\begin{pmatrix} \dot{1} \\ 1/5 \end{pmatrix}$ \\
        \hline
        $A^1_N(4,5)$ & $A_{\alpha_1}$ & $(E, E, E, A^{\alpha_1}_4, A^{\alpha_1}_5 E)$ & $\begin{pmatrix} \dot{1} \\ 4/5 \end{pmatrix}$ \\
        \hline
        $A^1(1,4)$ & $A_{\alpha_1}$ & $(A^{\alpha_1}_1, E, E, A^{\alpha_1}_4, E, E)$ & $\begin{pmatrix} 1 \\ 1/4 \end{pmatrix}$ \\
        \hline
        $A^2_N(2,3)$ & $A_{\alpha_2}$ & $(E, A^{\alpha_2}_2, A^{\alpha_2}_3 E, E, E)$ & $\begin{pmatrix} \dot{2} \\ 2/3 \end{pmatrix}$ \\
        \hline
        $A^2_N(3,5)$ & $A_{\alpha_2}$ & $(E, E, A^{\alpha_2}_3, E, A^{\alpha_2}_5 E)$ & $\begin{pmatrix} \dot{2} \\ 3/5 \end{pmatrix}$ \\
        \hline
        $A^2(2,5)$ & $A_{\alpha_2}$ & $(E, A^{\alpha_2}_2, E, E, A^{\alpha_2}_5, E)$ & $\begin{pmatrix} 2 \\ 2/5 \end{pmatrix}$ \\
        \hline
        $A^3_N(1,3)$ & $A_{\alpha_3}$ & $(A^{\alpha_3}_1 E, A^{\alpha_3}_3, E, E, E)$ & $\begin{pmatrix} \dot{3} \\ 1/3 \end{pmatrix}$ \\
        \hline
        $A^3_N(1,6)$ & $A_{\alpha_3}$ & $(A^{\alpha_3}_1 E, E, E, E, A^{\alpha_3}_6)$ & $\begin{pmatrix} \dot{3} \\ 1/6 \end{pmatrix}$ \\
        \hline
        $A^3(3,6)$ & $A_{\alpha_3}$ & $(E, E, A^{\alpha_3}_3, E, E, A^{\alpha_3}_6)$ & $\begin{pmatrix} 3 \\ 3/6 \end{pmatrix}$ \\
        \hline
        \end{longtblr}
    \end{center}
\end{prop}

\begin{proof}
    The proof parallels those of Propositions~\ref{thm:table_std} and ~\ref{thm:table_nonstd}, and is left to the reader.
\end{proof}

In Section~\ref{sec:con-algtri}, the face and edge subgroups above will be assigned to various faces and edges of a decorated 2-complex called the {\em algebraic triangle}. See Figure~\ref{fig:tri-2cell-label} and Figure~\ref{fig:tri-1cell-label} below for details. The next lemma shows that each edge group is a subgroup of the adjacent 2-cell groups.

\begin{lem}[Edge groups are subgroups of adjacent face groups]\label{lem:edgeface}
Each edge group in Definition~\ref{def:edgegroups} is a subgroup of the corresponding adjacent face groups of Definition~\ref{defn:stdsub} and Definition~\ref{defn:nonstdsub}. 
\end{lem}

\begin{proof}
To see this, one needs to simply verify that each generator for a given edge group appears (or is a word in) the generating set of the corresponding adjacent face groups. By way of example, we show that $G^3 \leq A^3G^3_N$ and $G^3 \leq G^{123}$ explicitly below, and leave the other verifications to the reader. 

The generating set for $G^3$ is
    \[
    \bigcup_{i \in \alpha_3}\left( \overline{Y}_i \,\cup\, \bigcup_{j \in \alpha_{456}}\overline{Z}^{j^\ast}_{i,j} \right) \, ,
    \]
which is a subset of the generating set 
    \[
    \bigcup_{i \in \alpha_3}\left( \overline{Y}_i \,\cup\, \bigcup_{j \in \alpha_{456}}\overline{Z}^{j^\ast}_{i,j} \,\cup\, \bigcup_{j \in \alpha_3} \overline{Z}^j_{i, j-|\alpha_{12}|} \right)
    \]
for $A^3G^3_N$ of Definition~\ref{defn:nonstdsub}. This establishes the inclusion $G^3 \leq A^3G^3_N$.

Similarly, the generating set for $G^3$ is also seen to be a subset of the generating set 
    \[
    \bigcup_{i \in \alpha_{123}}\left( \overline{Y}_i \,\cup\, \bigcup_{j \in \alpha_{456}}\overline{Z}^{j^\ast}_{i,j} \right)
    \]
for the group $G^{123}$ of Definition~\ref{defn:stdsub}, establishing the reverse inclusion. 
\end{proof}

\begin{rem}[Non-standard and standard edge groups contained in the $AG_N$ face groups]
In the proof of Lemma~\ref{lem:edgeface} above, it was shown that the non-standard group $A^3G^3_N$ contains the standard edge group $G^3$. The reader can verify that $A^3G^3_N$ also contains the non-standard edge group $G^3_N$ as a subgroup. The other edge group of $A^3G^3_N$ is $A^3_N(1,6)$ which converts between non-standard $\begin{pmatrix} \dot{3} \\ 1 \end{pmatrix}$ and the standard $\begin{pmatrix} 3 \\ 6 \end{pmatrix}$ linear algebra. This dual standard/non-standard nature of $A^3G^3_N$ (and of the other $AG_N$ subgroups) enables one to interpolate between the non-standard corner cells of the algebraic triangle and the standard cells in the central region -- see Figure~\ref{fig:tri-2cell-label} for details. 
\end{rem}


\section{Construction of the algebraic triangle}\label{sec:con-algtri}

In this section, we define and investigate properties of the {\em algebraic triangle}. The algebraic triangle is a decorated 2-complex consisting of a triangle which is subdivided into twenty-five subregions (as shown in Figure~\ref{fig:tri-2cell-label}), together with an assignment of particular subgroups of $\Ker(\Phi)$ (those defined in Section~\ref{sec:subgroups}) to the 1-cells and 2-cells of the 2-complex. 
The algebraic triangle provides a template for building certain special van Kampen diagrams in $\Ker(\Phi)$, which in turn tesselate (via a triangular Farey diagram construction) a van Kampen diagram for a general word in the generators of $\Ker(\Phi)$ which represents the identity.

\begin{rem}[Dropping bar notation for elements of $G = \prod_{j=1}^{2n}G_j$] 
To aid in readability of expressions, we drop the bar notation for elements of $G = \prod_{j=1}^{2n}G_j$ from this point on. So, for example,  an element $a \in \Ker(\Phi)$ is  a $2n$-tuple
$$
a \; =\; (a_1, \ldots , a_{2n}). 
$$

We retain the bar notation for sets of vectorized kernel- and section-generators as before; namely, $\overline{Y}_i$ and $\overline{Z}^j_{i,k}$.
\end{rem}

\begin{defn}[Vector kernel-section generators of $\Ker(\Phi)$]\label{def:kersecgens}
Define $\overline{X}$ to be the union of the generating sets for each of the twenty-five face subgroups of $\Ker(\Phi)$. In Corollary~\ref{cor:kerfingen}, we show that $\overline{X}$ generates $\Ker(\Phi)$.
\end{defn}

\begin{defn}[Algebraic triangle]{\label{def:abstract-alg-tri}}
Given integers $m \geq n \geq 3$, $2n$ groups $G_i$, and homomorphisms $\varphi_i: G_i \to \Z^{m}$,  satisfing the conditions of the basic set up in  Definition~\ref{def:setup}, the {\em algebraic triangle} corresponding to this data consists of the following:
    \begin{itemize}
    \item a $2$-dimensional disk, given a combinatorial $2$-complex structure with thirty-six $0$-cells, sixty $1$-cells, and twenty-five $2$-cells, as shown in Figure~\ref{fig:tri-2cell-label} below, and   
    \item a labeling of each $1$-cell (resp.\ $2$-cell) by a specific subgroup of $\Ker(\Phi)$ from Section~\ref{sec:subgroups}, as shown in Figure~\ref{fig:tri-1cell-label} (resp.\ Figure~\ref{fig:tri-2cell-label}) below.
    \end{itemize}
\end{defn}

The algebraic triangle is a decorated 2-complex which is used to build van Kampen diagrams for special types of loops in $\Ker(\Phi)$ as follows. 

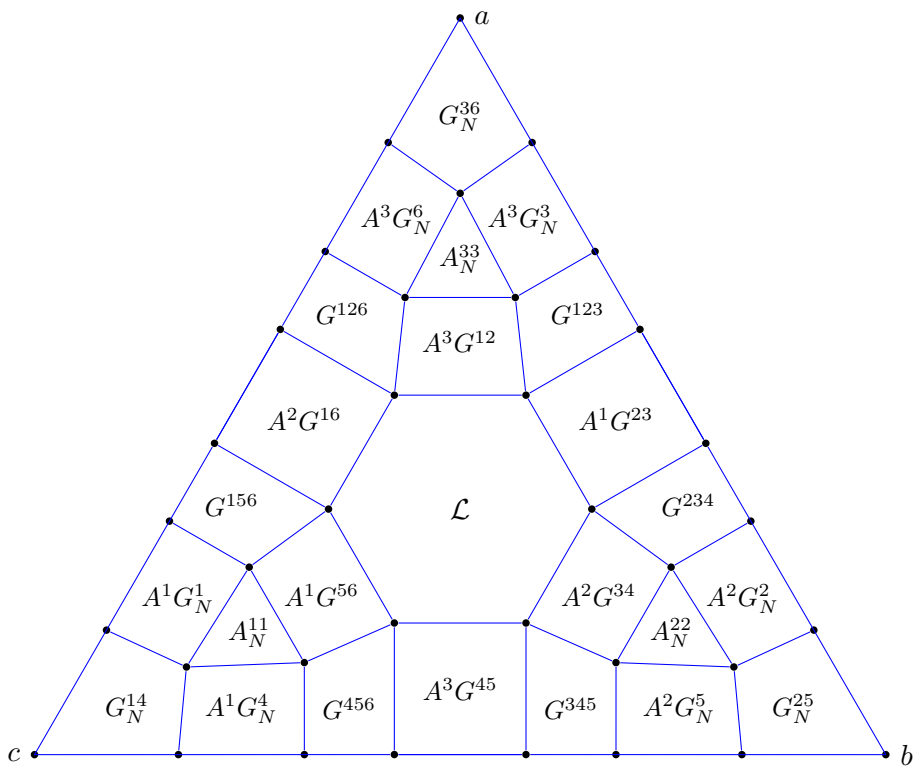
\begin{figure}[h]
    \begin{tikzpicture}[scale = 1.75]
    \foreach \i in {0,...,5}{
    \node[auto=center,style={circle,fill=black},inner sep=1pt] (hexagon-\i) at ($({\i*60}:1)$){};
    }
        
    \node[auto=center,style={circle,fill=black},inner sep=1pt] (A1) at ($(hexagon-1)+(30:1)$){};
    \node[auto=center,style={circle,fill=black},inner sep=1pt] (A2) at ($(hexagon-2)+(150:1)$){};
    \node[auto=center,style={circle,fill=black},inner sep=1pt] (B1) at ($(hexagon-3)+(150:1)$){};
    \node[auto=center,style={circle,fill=black},inner sep=1pt] (B2) at ($(hexagon-4)+(270:1)$){};
    \node[auto=center,style={circle,fill=black},inner sep=1pt] (C1) at ($(hexagon-5)+(270:1)$){};
    \node[auto=center,style={circle,fill=black},inner sep=1pt] (C2) at ($(hexagon-0)+(30:1)$){};
    
    \draw [name path=B1--A2,color=blue] (B1) -- (A2);
    \draw [name path=C2--A1,color=blue] (C2) -- (A1);
    \coordinate (pointA) at (intersection of B1--A2 and C2--A1);
    \node[auto=center,style={circle,fill=black},inner sep=1pt,label=right:$a$] at (pointA) {};
    
    \draw [name path=A2--B1,color=blue] (A2) -- (B1);
    \draw [name path=C1--B2,color=blue] (C1) -- (B2);
    \coordinate (pointB) at (intersection of A2--B1 and C1--B2);
    \node[auto=center,style={circle,fill=black},inner sep=1pt,label=left:$c$] at (pointB) {};
    
    \draw [name path=B2--C1,color=blue] (B2) -- (C1);
    \draw [name path=A1--C2,color=blue] (A1) -- (C2);
    \coordinate (pointC) at (intersection of B2--C1 and A1--C2);
    \node[auto=center,style={circle,fill=black},inner sep=1pt,label=right:$b$] at (pointC) {};
    
    \draw [color=blue] (hexagon-0) -- (hexagon-1);
    \draw [color=blue] (hexagon-2) -- (hexagon-3);
    \draw [color=blue] (hexagon-4) -- (hexagon-5);
    
    \node[auto=center,style={circle,fill=black},inner sep=1pt] (AB1) at ($(pointA)!{0.4}!(A2)$){};
    \node[auto=center,style={circle,fill=black},inner sep=1pt] (AB2) at ($(pointA)!{0.75}!(A2)$){};
    \node[auto=center,style={circle,fill=black},inner sep=1pt] (AC1) at ($(pointA)!{0.4}!(A1)$){};
    \node[auto=center,style={circle,fill=black},inner sep=1pt] (AC2) at ($(pointA)!{0.75}!(A1)$){};
    \node[auto=center,style={circle,fill=black},inner sep=1pt] (AIL) at ($(AB2)+(-30:0.7)$){};
    \node[auto=center,style={circle,fill=black},inner sep=1pt] (AIR) at ($(AC2)+(210:0.7)$){};
    \node[auto=center,style={circle,fill=black},inner sep=1pt] (AIU) at ($(pointA)+(270:1.333)$){};
    
    \draw[color=blue] (pointA) -- (A2) -- (hexagon-2) -- (hexagon-1) -- (A1) -- (pointA);
    \draw[color=blue] (AB2) -- (AIL) -- (AIR) -- (AC2);
    \draw[color=blue] (AB1) -- (AIU) -- (AC1);
    \draw[color=blue] (hexagon-2) -- (AIL) -- (AIU) -- (AIR) -- (hexagon-1);
    
    \node[auto=center,style={circle,fill=black},inner sep=1pt] (BA1) at ($(pointB)!{0.4}!(B1)$){};
    \node[auto=center,style={circle,fill=black},inner sep=1pt] (BA2) at ($(pointB)!{0.75}!(B1)$){};
    \node[auto=center,style={circle,fill=black},inner sep=1pt] (BC1) at ($(pointB)!{0.4}!(B2)$){};
    \node[auto=center,style={circle,fill=black},inner sep=1pt] (BC2) at ($(pointB)!{0.75}!(B2)$){};
    \node[auto=center,style={circle,fill=black},inner sep=1pt] (BIR) at ($(BA2)+(-30:0.7)$){};
    \node[auto=center,style={circle,fill=black},inner sep=1pt] (BIL) at ($(BC2)+(90:0.7)$){};
    \node[auto=center,style={circle,fill=black},inner sep=1pt] (BIU) at ($(pointB)+(30:1.333)$){};
    
    \draw[color=blue] (pointB) -- (B2) -- (hexagon-4) -- (hexagon-3) -- (B1) -- (pointB);
    \draw[color=blue] (BA2) -- (BIR) -- (BIL) -- (BC2);
    \draw[color=blue] (BA1) -- (BIU) -- (BC1);
    \draw[color=blue] (hexagon-4) -- (BIL) -- (BIU) -- (BIR) -- (hexagon-3);
    
    \node[auto=center,style={circle,fill=black},inner sep=1pt] (CB1) at ($(pointC)!{0.4}!(C1)$){};
    \node[auto=center,style={circle,fill=black},inner sep=1pt] (CB2) at ($(pointC)!{0.75}!(C1)$){};
    \node[auto=center,style={circle,fill=black},inner sep=1pt] (CA1) at ($(pointC)!{0.4}!(C2)$){};
    \node[auto=center,style={circle,fill=black},inner sep=1pt] (CA2) at ($(pointC)!{0.75}!(C2)$){};
    \node[auto=center,style={circle,fill=black},inner sep=1pt] (CIR) at ($(CB2)+(90:0.7)$){};
    \node[auto=center,style={circle,fill=black},inner sep=1pt] (CIL) at ($(CA2)+(210:0.7)$){};
    \node[auto=center,style={circle,fill=black},inner sep=1pt] (CIU) at ($(pointC)+(150:1.333)$){};
    
    \draw[color=blue] (pointC) -- (C2) -- (hexagon-0) -- (hexagon-5) -- (C1) -- (pointC);
    \draw[color=blue] (CB2) -- (CIR) -- (CIL) -- (CA2);
    \draw[color=blue] (CB1) -- (CIU) -- (CA1);
    \draw[color=blue] (hexagon-0) -- (CIL) -- (CIU) -- (CIR) -- (hexagon-5);

    \node at (0,0) {$\mathcal{L}$};
    \node at ($(pointA)+(0,-0.75)$) {\small $G^{36}_N$};
    \node at ($(pointB)+(0.7,0.35)$) {\small $G^{14}_N$};
    \node at ($(pointC)+(-0.7,0.35)$) {\small $G^{25}_N$};
    \node at ($(hexagon-1)!{0.55}!(AIL)$) {\small $A^3G^{12}$};
    \node at ($(hexagon-4)!{0.5}!(BIR)$) {\small $A^1G^{56}$};
    \node at ($(hexagon-5)!{0.5}!(CIL)$) {\small $A^2G^{34}$};
    \node at ($(hexagon-0)!{0.5}!(A1)$) {\small $A^1G^{23}$};
    \node at ($(hexagon-3)!{0.5}!(A2)$) {\small $A^2G^{16}$};
    \node at ($(hexagon-4)!{0.5}!(C1)$) {\small $A^3G^{45}$};
    \node at ($(A1)!{0.5}!(AIR)$) {\small $G^{123}$};
    \node at ($(A2)!{0.5}!(AIL)$) {\small $G^{126}$};
    \node at ($(B1)!{0.5}!(BIR)$) {\small $G^{156}$};
    \node at ($(B2)!{0.5}!(BIL)$) {\small $G^{456}$};
    \node at ($(C1)!{0.5}!(CIR)$) {\small $G^{345}$};
    \node at ($(C2)!{0.5}!(CIL)$) {\small $G^{234}$};
    \node at ($(AC1)!{0.5}!(AIR)$) {\small $A^3G^3_N$};
    \node at ($(AB1)!{0.5}!(AIL)$) {\small $A^3G^6_N$};
    \node at ($(BA1)!{0.5}!(BIR)$) {\small $A^1G^1_N$};
    \node at ($(BC1)!{0.5}!(BIL)$) {\small $A^1G^4_N$};
    \node at ($(CB1)!{0.5}!(CIR)$) {\small $A^2G^5_N$};
    \node at ($(CA1)!{0.5}!(CIL)$) {\small $A^2G^2_N$};
    \node at ($(AIU)+(0,-0.5)$) {\small $A^{33}_N$};
    \node at ($(BIR)+(0,-0.5)$) {\small $A^{11}_N$};
    \node at ($(CIL)+(0,-0.5)$) {\small $A^{22}_N$};
    \end{tikzpicture}
\caption{The algebraic triangle $\Delta$ --- cell structure and labeling of the 2-cells by subgroups of $\Ker(\Phi)$.}
\label{fig:tri-2cell-label}
\end{figure}
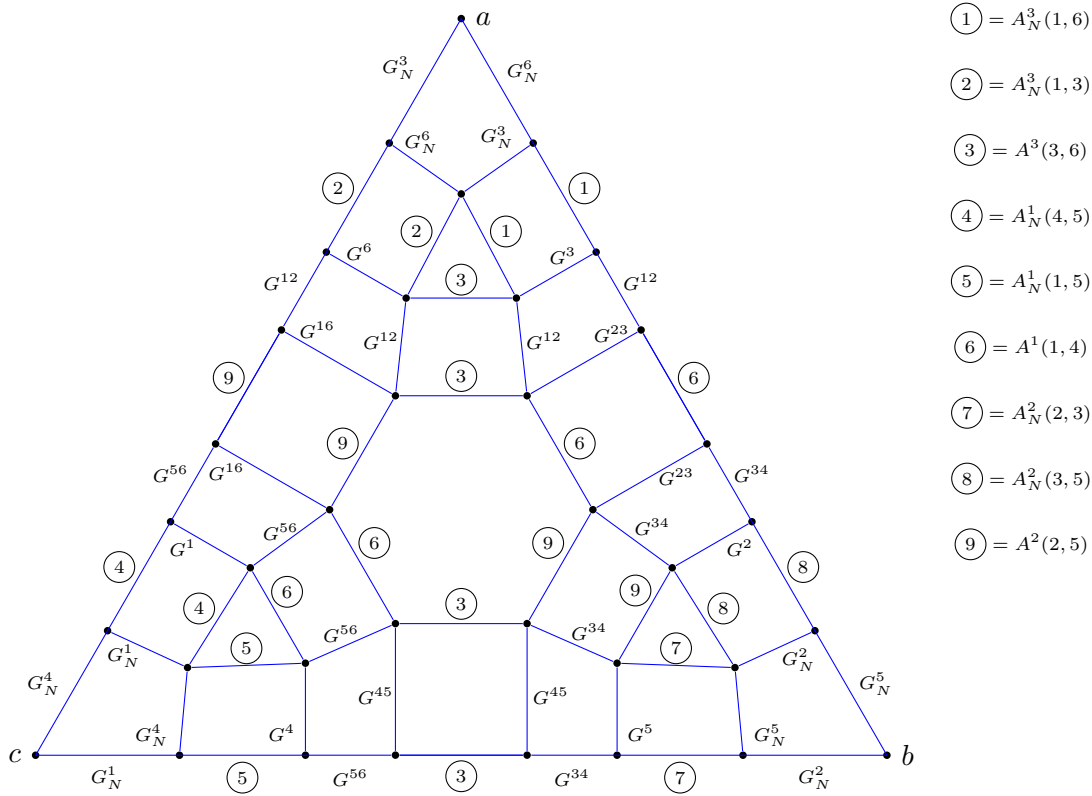
\begin{figure}[h]
\centering
    \begin{tikzpicture}[scale = 1.75]
    \foreach \i in {0,...,5}{
    \node[auto=center,style={circle,fill=black},inner sep=1pt] (hexagon-\i) at ($({\i*60}:1)$){};
    }
        
    \node[auto=center,style={circle,fill=black},inner sep=1pt] (A1) at ($(hexagon-1)+(30:1)$){};
    \node[auto=center,style={circle,fill=black},inner sep=1pt] (A2) at ($(hexagon-2)+(150:1)$){};
    \node[auto=center,style={circle,fill=black},inner sep=1pt] (B1) at ($(hexagon-3)+(150:1)$){};
    \node[auto=center,style={circle,fill=black},inner sep=1pt] (B2) at ($(hexagon-4)+(270:1)$){};
    \node[auto=center,style={circle,fill=black},inner sep=1pt] (C1) at ($(hexagon-5)+(270:1)$){};
    \node[auto=center,style={circle,fill=black},inner sep=1pt] (C2) at ($(hexagon-0)+(30:1)$){};
    
    \draw [name path=B1--A2,color=blue] (B1) -- (A2);
    \draw [name path=C2--A1,color=blue] (C2) -- (A1);
    \coordinate (pointA) at (intersection of B1--A2 and C2--A1);
    \node[auto=center,style={circle,fill=black},inner sep=1pt,label=right:$a$] at (pointA) {};
    
    \draw [name path=A2--B1,color=blue] (A2) -- (B1);
    \draw [name path=C1--B2,color=blue] (C1) -- (B2);
    \coordinate (pointB) at (intersection of A2--B1 and C1--B2);
    \node[auto=center,style={circle,fill=black},inner sep=1pt,label=left:$c$] at (pointB) {};
    
    \draw [name path=B2--C1,color=blue] (B2) -- (C1);
    \draw [name path=A1--C2,color=blue] (A1) -- (C2);
    \coordinate (pointC) at (intersection of B2--C1 and A1--C2);
    \node[auto=center,style={circle,fill=black},inner sep=1pt,label=right:$b$] at (pointC) {};
    
    \draw [color=blue] (hexagon-0) -- (hexagon-1);
    \draw [color=blue] (hexagon-2) -- (hexagon-3);
    \draw [color=blue] (hexagon-4) -- (hexagon-5);
    
    \node[auto=center,style={circle,fill=black},inner sep=1pt] (AB1) at ($(pointA)!{0.4}!(A2)$){};
    \node[auto=center,style={circle,fill=black},inner sep=1pt] (AB2) at ($(pointA)!{0.75}!(A2)$){};
    \node[auto=center,style={circle,fill=black},inner sep=1pt] (AC1) at ($(pointA)!{0.4}!(A1)$){};
    \node[auto=center,style={circle,fill=black},inner sep=1pt] (AC2) at ($(pointA)!{0.75}!(A1)$){};
    \node[auto=center,style={circle,fill=black},inner sep=1pt] (AIL) at ($(AB2)+(-30:0.7)$){};
    \node[auto=center,style={circle,fill=black},inner sep=1pt] (AIR) at ($(AC2)+(210:0.7)$){};
    \node[auto=center,style={circle,fill=black},inner sep=1pt] (AIU) at ($(pointA)+(270:1.333)$){};
    
    \draw[color=blue] (pointA) -- (A2) -- (hexagon-2) -- (hexagon-1) -- (A1) -- (pointA);
     \draw[color=blue] (AB2) -- (AIL) -- (AIR) -- (AC2);
    \draw[color=blue] (AB1) -- (AIU) -- (AC1);
    \draw[color=blue] (hexagon-2) -- (AIL) -- (AIU) -- (AIR) -- (hexagon-1);
    
    \node[auto=center,style={circle,fill=black},inner sep=1pt] (BA1) at ($(pointB)!{0.4}!(B1)$){};
    \node[auto=center,style={circle,fill=black},inner sep=1pt] (BA2) at ($(pointB)!{0.75}!(B1)$){};
    \node[auto=center,style={circle,fill=black},inner sep=1pt] (BC1) at ($(pointB)!{0.4}!(B2)$){};
    \node[auto=center,style={circle,fill=black},inner sep=1pt] (BC2) at ($(pointB)!{0.75}!(B2)$){};
    \node[auto=center,style={circle,fill=black},inner sep=1pt] (BIR) at ($(BA2)+(-30:0.7)$){};
    \node[auto=center,style={circle,fill=black},inner sep=1pt] (BIL) at ($(BC2)+(90:0.7)$){};
    \node[auto=center,style={circle,fill=black},inner sep=1pt] (BIU) at ($(pointB)+(30:1.333)$){};
    
    \draw[color=blue] (pointB) -- (B2) -- (hexagon-4) -- (hexagon-3) -- (B1) -- (pointB);
    \draw[color=blue] (BA2) -- (BIR) -- (BIL) -- (BC2);
    \draw[color=blue] (BA1) -- (BIU) -- (BC1);
    \draw[color=blue] (hexagon-4) -- (BIL) -- (BIU) -- (BIR) -- (hexagon-3);
    
    \node[auto=center,style={circle,fill=black},inner sep=1pt] (CB1) at ($(pointC)!{0.4}!(C1)$){};
    \node[auto=center,style={circle,fill=black},inner sep=1pt] (CB2) at ($(pointC)!{0.75}!(C1)$){};
    \node[auto=center,style={circle,fill=black},inner sep=1pt] (CA1) at ($(pointC)!{0.4}!(C2)$){};
    \node[auto=center,style={circle,fill=black},inner sep=1pt] (CA2) at ($(pointC)!{0.75}!(C2)$){};
    \node[auto=center,style={circle,fill=black},inner sep=1pt] (CIR) at ($(CB2)+(90:0.7)$){};
    \node[auto=center,style={circle,fill=black},inner sep=1pt] (CIL) at ($(CA2)+(210:0.7)$){};
    \node[auto=center,style={circle,fill=black},inner sep=1pt] (CIU) at ($(pointC)+(150:1.333)$){};
    
    \draw[color=blue] (pointC) -- (C2) -- (hexagon-0) -- (hexagon-5) -- (C1) -- (pointC);
    \draw[color=blue] (CB2) -- (CIR) -- (CIL) -- (CA2);
    \draw[color=blue] (CB1) -- (CIU) -- (CA1);
    \draw[color=blue] (hexagon-0) -- (CIL) -- (CIU) -- (CIR) -- (hexagon-5);

    \node (i) at ($(pointA)+(4.25,0)$) {\tiny $\circled{1} = A^3_N(1,6)$};
    \node (ii) at ($(i)+(0,-0.5)$) {\tiny $\circled{2} = A^3_N(1,3)$};
    \node (iii) at ($(ii)+(0,-0.5)$) {\tiny $\circled{3} = A^3(3,6)$};

    \node (iv) at ($(iii)+(0,-0.5)$) {\tiny $\circled{4} = A^1_N(4,5)$};
    \node (v) at ($(iv)+(0,-0.5)$) {\tiny $\circled{5} = A^1_N(1,5)$};
    \node (vi) at ($(v)+(0,-0.5)$) {\tiny $\circled{6} = A^1(1,4)$};

    \node (vii) at ($(vi)+(0,-0.5)$) {\tiny $\circled{7} = A^2_N(2,3)$};
    \node (viii) at ($(vii)+(0,-0.5)$) {\tiny $\circled{8} = A^2_N(3,5)$};
    \node (ix) at ($(viii)+(0,-0.5)$) {\tiny $\circled{9} = A^2(2,5)$};
    
    \node at ($(pointA)!{0.5}!(AC1)+(0.2,0.07)$) {\tiny $G^6_N$};
    \node at ($(AC1)!{0.5}!(AC2)+(0.15,0.07)$) {\tiny $\circled{1}$};
    \node at ($(AC2)!{0.5}!(A1)+(0.17,0.07)$) {\tiny $G^{12}$};
    \node at ($(A1)!{0.5}!(C2)+(0.15,0.07)$) {\tiny $\circled{6}$};
    \node at ($(C2)!{0.5}!(CA2)+(0.17,0.07)$) {\tiny $G^{34}$};
    \node at ($(CA2)!{0.5}!(CA1)+(0.15,0.07)$) {\tiny $\circled{8}$};
    \node at ($(CA1)!{0.5}!(pointC)+(0.19,0.07)$) {\tiny $G^5_N$};
    \node at ($(pointA)!{0.5}!(AB1)+(-0.2,0.1)$) {\tiny $G^3_N$};
    \node at ($(AB1)!{0.5}!(AB2)+(-0.15,0.07)$) {\tiny $\circled{2}$};
    \node at ($(AB2)!{0.5}!(A2)+(-0.17,0.07)$) {\tiny $G^{12}$};
    \node at ($(A2)!{0.5}!(B1)+(-0.15,0.07)$) {\tiny $\circled{9}$};
    \node at ($(B1)!{0.5}!(BA2)+(-0.17,0.07)$) {\tiny $G^{56}$};
    \node at ($(BA2)!{0.5}!(BA1)+(-0.15,0.07)$) {\tiny $\circled{4}$};
    \node at ($(BA1)!{0.5}!(pointB)+(-0.21,0.07)$) {\tiny $G^4_N$};
    \node at ($(pointB)!{0.5}!(BC1)+(0,-0.17)$) {\tiny $G^1_N$};
    \node at ($(BC1)!{0.5}!(BC2)+(0,-0.17)$) {\tiny $\circled{5}$};
    \node at ($(BC2)!{0.5}!(B2)+(0,-0.17)$) {\tiny $G^{56}$};
    \node at ($(B2)!{0.5}!(C1)+(0,-0.16)$) {\tiny $\circled{3}$};
    \node at ($(C1)!{0.5}!(CB2)+(0,-0.17)$) {\tiny $G^{34}$};
    \node at ($(CB2)!{0.5}!(CB1)+(0,-0.17)$) {\tiny $\circled{7}$};
    \node at ($(CB1)!{0.5}!(pointC)+(0,-0.17)$) {\tiny $G^2_N$};

    \node at ($(AC1)!{0.5}!(AIU)+(0,0.24)$) {\tiny $G^3_N$};
    \node at ($(AB1)!{0.5}!(AIU)+(-0.03,0.2)$) {\tiny $G^6_N$};
    \node at ($(AIU)!{0.5}!(AIR)+(0.13,0.11)$) {\tiny $\circled{1}$};
    \node at ($(AIU)!{0.5}!(AIL)+(-0.13,0.11)$) {\tiny $\circled{2}$};
    \node at ($(AIL)!{0.5}!(AIR)+(0,0.14)$) {\tiny $\circled{3}$};
    \node at ($(AIR)!{0.5}!(AC2)+(0.05,0.15)$) {\tiny $G^3$};
    \node at ($(AIL)!{0.5}!(AB2)+(-0.05,0.18)$) {\tiny $G^6$};
    \node at ($(A2)!{0.5}!(hexagon-2)+(-0.16,0.25)$) {\tiny $G^{16}$};
    \node at ($(A1)!{0.5}!(hexagon-1)+(0.21,0.22)$) {\tiny $G^{23}$};
    \node at ($(hexagon-1)!{0.5}!(AIR)+(0.17,0.05)$) {\tiny $G^{12}$};
    \node at ($(hexagon-2)!{0.5}!(AIL)+(-0.15,0.05)$) {\tiny $G^{12}$};
    \node at ($(hexagon-1)!{0.5}!(hexagon-2)+(0,0.15)$) {\tiny $\circled{3}$};

    \node at ($(BA1)!{0.5}!(BIU)+(-0.18,-0.05)$) {\tiny $G^1_N$};
    \node at ($(BC1)!{0.5}!(BIU)+(-0.24,-0.18)$) {\tiny $G^4_N$};
    \node at ($(BIR)!{0.5}!(BA2)+(-0.21,-0.02)$) {\tiny $G^1$};
    \node at ($(hexagon-3)!{0.5}!(B1)+(-0.35,0.04)$) {\tiny $G^{16}$};
    \node at ($(BIU)!{0.5}!(BIR)+(-0.15,0.07)$) {\tiny $\circled{4}$};
    \node at ($(hexagon-3)!{0.5}!(BIR)+(-0.07,0.05)$) {\tiny $G^{56}$};
    \node at ($(BIL)!{0.5}!(BIU)+(0,0.13)$) {\tiny $\circled{5}$};
    \node at ($(BC2)!{0.5}!(BIL)+(-0.18,-0.18)$) {\tiny $G^4$};
    \node at ($(BIL)!{0.5}!(BIR)+(0.07,0.18)$) {\tiny $\circled{6}$};
    \node at ($(hexagon-3)!{0.5}!(hexagon-4)+(0.09,0.18)$) {\tiny $\circled{6}$};
    \node at ($(hexagon-4)!{0.5}!(BIL)+(-0.07,0.07)$) {\tiny $G^{56}$};
    \node at ($(hexagon-4)!{0.5}!(B2)+(-0.15,-0.05)$) {\tiny $G^{45}$};

    \node at ($(CA1)!{0.5}!(CIU)+(0.18,-0.08)$) {\tiny $G^2_N$};
    \node at ($(CB1)!{0.5}!(CIU)+(0.24,-0.18)$) {\tiny $G^5_N$};
    \node at ($(CIL)!{0.5}!(CA2)+(0.21,-0.02)$) {\tiny $G^2$};
    \node at ($(hexagon-5)!{0.5}!(C1)+(0.18,-0.05)$) {\tiny $G^{45}$};
    \node at ($(CIU)!{0.5}!(CIL)+(0.15,0.07)$) {\tiny $\circled{8}$};
    \node at ($(hexagon-5)!{0.5}!(CIR)+(0.13,0.09)$) {\tiny $G^{34}$};
    \node at ($(CIR)!{0.5}!(CIU)+(0,0.13)$) {\tiny $\circled{7}$};
    \node at ($(CB2)!{0.5}!(CIR)+(0.18,-0.18)$) {\tiny $G^5$};
    \node at ($(CIL)!{0.5}!(CIR)+(-0.07,0.19)$) {\tiny $\circled{9}$};
    \node at ($(hexagon-5)!{0.5}!(hexagon-0)+(-0.09,0.18)$) {\tiny $\circled{9}$};
    \node at ($(hexagon-0)!{0.5}!(CIL)+(0.15,0.09)$) {\tiny $G^{34}$};
    \node at ($(hexagon-0)!{0.5}!(C2)+(0.2,0)$) {\tiny $G^{23}$};

    \node at ($(hexagon-2)!{0.5}!(hexagon-3)+(-0.15,0.07)$) {\tiny $\circled{9}$};
    \node at ($(hexagon-0)!{0.5}!(hexagon-1)+(0.15,0.07)$) {\tiny $\circled{6}$};
    \node at ($(hexagon-4)!{0.5}!(hexagon-5)+(0,0.15)$) {\tiny $\circled{3}$};
    \end{tikzpicture}
\caption{The algebraic triangle $\Delta$ --- labeling of the $1$-cells by subgroups of $\Ker(\Phi)$.}
\label{fig:tri-1cell-label}
\end{figure}

\begin{defn}[Actualization of the algebraic triangle]\label{def:actual-alg-tri}
Let $(a,b,c)$ be an ordered triple of elements of $\Ker(\Phi)$. The algebraic triangle $\Delta$ forms a template for a particular van Kampen diagram, denoted $\Delta(a,b,c)$ and called the \emph{actualization of $\Delta$ with respect to the triple $(a,b,c)$}, in the following way.

\begin{enumerate}
    \item Let $a$, $b$, and $c$  label the three vertices of the triangle as shown in Figure~\ref{fig:tri-2cell-label}. So $a$ at the top, $b$ on the bottom right, and $c$ on the bottom left. 
    \item Use the sequence of 7 edge segments from $a$ to $b$ to construct an edge-path, $\partial\Delta(a,b)$, of vector kernel-section generators of these 7 sedge subgroups of $\Ker(\Phi)$ between these two elements. The detailed procedure for building such an edge-path is described in Lemma~\ref{lem:triside-labdist}. Note that since $a$ and $b$ are arbitrary, this proves that union of the vector kernel-section generating sets for these 7 edge groups forms a generating set for $\Ker(\Phi)$; this is the content of Corollary~\ref{cor:kerfingen}. We can view $\partial\Delta(a,b)$ as a particular edge path from $a$ to $b$ in the Cayley graph of $\Ker(\Phi)$. 

    In a similar fashion, use the 7 edge segments from $b$ to $c$ to build an edge path, $\partial\Delta(b,c)$, from $b$ to $c$, and use the remaining 7 edge segments from $c$ to $a$ to build an edge path $\partial\Delta(c,a)$ from $c$ to $a$. 
    \item The union $\partial\Delta(a,b) \cup \partial\Delta(b,c)\cup \partial\Delta(c,a)$ gives a loop in the Cayley graph of $\Ker(\Phi)$. Lemma~\ref{lem:triinner-labdist} describes how to subdivide the loop   $\partial\Delta(a,b) \cup \partial\Delta(b,c)\cup \partial\Delta(c,a)$  into 25  loops, indexed by the 2-cells of $\Delta$, each of which is contained in the Cayley graph of the corresponding  face group. Since the face groups are finitely presented, each of the 25 sub-loops can be filled by a van Kamepn diagram over its face subgroup. 
    \item The union of these 25 van Kampen diagrams gives a van Kampen diagram for the loop 
    $\partial\Delta(a,b) \cup \partial\Delta(b,c)\cup \partial\Delta(c,a)$. We call this van Kampen diagram the {\em actualization  of the algebraic triangle $\Delta$} and denote it by 
   $\Delta(a,b,c)$. 
 \end{enumerate}  
   Note that $\Delta(a,b,c)$ is unique up to choices of the edge paths constructed in  Lemma~\ref{lem:triside-labdist} and Lemma~\ref{lem:triinner-labdist}, and to the choices of van Kampen fillings chosen in the 25 face groups. 
\end{defn}

\begin{figure}[h]
\centering
    \begin{tikzpicture}[scale = 1.75]
    \foreach \i in {0,...,5}{
    \node[auto=center,style={circle,fill=black},inner sep=1pt] (hexagon-\i) at ($({\i*60}:1)$){};
    }
        
    \node[auto=center,style={circle,fill=black},inner sep=1pt] (A1) at ($(hexagon-1)+(30:1)$){};
    \node[auto=center,style={circle,fill=black},inner sep=1pt] (A2) at ($(hexagon-2)+(150:1)$){};
    \node[auto=center,style={circle,fill=black},inner sep=1pt] (B1) at ($(hexagon-3)+(150:1)$){};
    \node[auto=center,style={circle,fill=black},inner sep=1pt] (B2) at ($(hexagon-4)+(270:1)$){};
    \node[auto=center,style={circle,fill=black},inner sep=1pt] (C1) at ($(hexagon-5)+(270:1)$){};
    \node[auto=center,style={circle,fill=black},inner sep=1pt] (C2) at ($(hexagon-0)+(30:1)$){};
    
    \draw [name path=B1--A2,color=blue] (B1) -- (A2);
    \draw [name path=C2--A1,color=blue] (C2) -- (A1);
    \coordinate (pointA) at (intersection of B1--A2 and C2--A1);
    \node[auto=center,style={circle,fill=black},inner sep=1pt] at (pointA) {};
    
    \draw [name path=A2--B1,color=blue] (A2) -- (B1);
    \draw [name path=C1--B2,color=blue] (C1) -- (B2);
    \coordinate (pointB) at (intersection of A2--B1 and C1--B2);
    \node[auto=center,style={circle,fill=black},inner sep=1pt] at (pointB) {};
    
    \draw [name path=B2--C1,color=blue] (B2) -- (C1);
    \draw [name path=A1--C2,color=blue] (A1) -- (C2);
    \coordinate (pointC) at (intersection of B2--C1 and A1--C2);
    \node[auto=center,style={circle,fill=black},inner sep=1pt] at (pointC) {};
    
    \draw [color=blue] (hexagon-0) -- (hexagon-1);
    \draw [color=blue] (hexagon-2) -- (hexagon-3);
    \draw [color=blue] (hexagon-4) -- (hexagon-5);
    
    \node[auto=center,style={circle,fill=black},inner sep=1pt] (AB1) at ($(pointA)!{0.4}!(A2)$){};
    \node[auto=center,style={circle,fill=black},inner sep=1pt] (AB2) at ($(pointA)!{0.75}!(A2)$){};
    \node[auto=center,style={circle,fill=black},inner sep=1pt] (AC1) at ($(pointA)!{0.4}!(A1)$){};
    \node[auto=center,style={circle,fill=black},inner sep=1pt] (AC2) at ($(pointA)!{0.75}!(A1)$){};
    \node[auto=center,style={circle,fill=black},inner sep=1pt] (AIL) at ($(AB2)+(-30:0.7)$){};
    \node[auto=center,style={circle,fill=black},inner sep=1pt] (AIR) at ($(AC2)+(210:0.7)$){};
    \node[auto=center,style={circle,fill=black},inner sep=1pt] (AIU) at ($(pointA)+(270:1.333)$){};
    
    \draw[color=blue] (pointA) -- (A2) -- (hexagon-2) -- (hexagon-1) -- (A1) -- (pointA);
    \draw[color=blue] (AB2) -- (AIL) -- (AIR) -- (AC2);
    \draw[color=blue] (AB1) -- (AIU) -- (AC1);
    \draw[color=blue] (hexagon-2) -- (AIL) -- (AIU) -- (AIR) -- (hexagon-1);
    
    \node[auto=center,style={circle,fill=black},inner sep=1pt] (BA1) at ($(pointB)!{0.4}!(B1)$){};
    \node[auto=center,style={circle,fill=black},inner sep=1pt] (BA2) at ($(pointB)!{0.75}!(B1)$){};
    \node[auto=center,style={circle,fill=black},inner sep=1pt] (BC1) at ($(pointB)!{0.4}!(B2)$){};
    \node[auto=center,style={circle,fill=black},inner sep=1pt] (BC2) at ($(pointB)!{0.75}!(B2)$){};
    \node[auto=center,style={circle,fill=black},inner sep=1pt] (BIR) at ($(BA2)+(-30:0.7)$){};
    \node[auto=center,style={circle,fill=black},inner sep=1pt] (BIL) at ($(BC2)+(90:0.7)$){};
    \node[auto=center,style={circle,fill=black},inner sep=1pt] (BIU) at ($(pointB)+(30:1.333)$){};
    
    \draw[color=blue] (pointB) -- (B2) -- (hexagon-4) -- (hexagon-3) -- (B1) -- (pointB);
    \draw[color=blue] (BA2) -- (BIR) -- (BIL) -- (BC2);
    \draw[color=blue] (BA1) -- (BIU) -- (BC1);
    \draw[color=blue] (hexagon-4) -- (BIL) -- (BIU) -- (BIR) -- (hexagon-3);
    
    \node[auto=center,style={circle,fill=black},inner sep=1pt] (CB1) at ($(pointC)!{0.4}!(C1)$){};
    \node[auto=center,style={circle,fill=black},inner sep=1pt] (CB2) at ($(pointC)!{0.75}!(C1)$){};
    \node[auto=center,style={circle,fill=black},inner sep=1pt] (CA1) at ($(pointC)!{0.4}!(C2)$){};
    \node[auto=center,style={circle,fill=black},inner sep=1pt] (CA2) at ($(pointC)!{0.75}!(C2)$){};
    \node[auto=center,style={circle,fill=black},inner sep=1pt] (CIR) at ($(CB2)+(90:0.7)$){};
    \node[auto=center,style={circle,fill=black},inner sep=1pt] (CIL) at ($(CA2)+(210:0.7)$){};
    \node[auto=center,style={circle,fill=black},inner sep=1pt] (CIU) at ($(pointC)+(150:1.333)$){};
    
    \draw[color=blue] (pointC) -- (C2) -- (hexagon-0) -- (hexagon-5) -- (C1) -- (pointC);
    \draw[color=blue] (CB2) -- (CIR) -- (CIL) -- (CA2);
    \draw[color=blue] (CB1) -- (CIU) -- (CA1);
    \draw[color=blue] (hexagon-0) -- (CIL) -- (CIU) -- (CIR) -- (hexagon-5);

    \node at (0,0) {$\mathcal{L}$};
    \node at ($(pointA)+(0,-0.75)$) {$G^{36}_N$};
    \node at ($(hexagon-1)!{0.55}!(AIL)$) {$A^3G^{12}$};
    \node at ($(hexagon-0)!{0.5}!(A1)$) {$A^1G^{23}$};
    \node at ($(A1)!{0.5}!(AIR)$) {$G^{123}$};
    \node at ($(A2)!{0.5}!(AIL)$) {$G^{126}$};
    \node at ($(AC1)!{0.5}!(AIR)$) {$A^3G^3_N$};
    \node at ($(AB1)!{0.5}!(AIL)$) {$A^3G^6_N$};
    \node at ($(AIU)+(0,-0.5)$) {$A^{33}_N$};

    \node at ($(pointA)+(0.15,0)$) {\tiny \textcolor{red}{$P_1$}};
    \node at ($(AC1)+(0.15,0)$) {\tiny \textcolor{red}{$P_2$}};
    \node at ($(AC2)+(0.15,0)$) {\tiny \textcolor{red}{$P_3$}};
    \node at ($(A1)+(0.15,0)$) {\tiny \textcolor{red}{$P_4$}};
    \node at ($(C2)+(0.15,0)$) {\tiny \textcolor{red}{$P_5$}};
    \node at ($(CA2)+(0.15,0)$) {\tiny \textcolor{red}{$P_6$}};
    \node at ($(CA1)+(0.15,0)$) {\tiny \textcolor{red}{$P_7$}};
    \node at ($(pointC)+(0.15,0)$) {\tiny \textcolor{red}{$P_8$}};
    \node at ($(hexagon-0)+(-0.2,0)$) {\tiny \textcolor{red}{$P_{17}$}};
    \node at ($(hexagon-1)+(0.2,-0.05)$) {\tiny \textcolor{red}{$P_{16}$}};
    \node at ($(AIR)+(0.05,0.2)$) {\tiny \textcolor{red}{$P_{13}$}};
    \node at ($(AIU)+(0,0.2)$) {\tiny \textcolor{red}{$P_{10}$}};
    \node at ($(AB1)+(-0.2,0)$) {\tiny \textcolor{red}{$P_{9}$}};
    \node at ($(AB2)+(-0.2,0)$) {\tiny \textcolor{red}{$P_{11}$}};
    \node at ($(AIL)+(-0.07,0.18)$) {\tiny \textcolor{red}{$P_{12}$}};
    \node at ($(hexagon-2)+(-0.2,0)$) {\tiny \textcolor{red}{$P_{15}$}};
    \node at ($(A2)+(-0.2,0)$) {\tiny \textcolor{red}{$P_{14}$}};

    \node at ($(hexagon-5)+(-0.1,0.1)$) {\tiny \textcolor{red}{$P_{18}$}};
    \node at ($(hexagon-4)+(0.1,0.1)$) {\tiny \textcolor{red}{$P_{19}$}};
    \node at ($(hexagon-3)+(0.2,0)$) {\tiny \textcolor{red}{$P_{20}$}};
    \end{tikzpicture}
\caption{A labeling of  vertices of the algebraic triangle corresponding to the right side and to one third of the $2$-cells. This labeling is used in the proofs of Lemma~\ref{lem:triside-labdist} and of Lemma~\ref{lem:triinner-labdist}.}
\label{fig:triangle-vertices}
\end{figure}
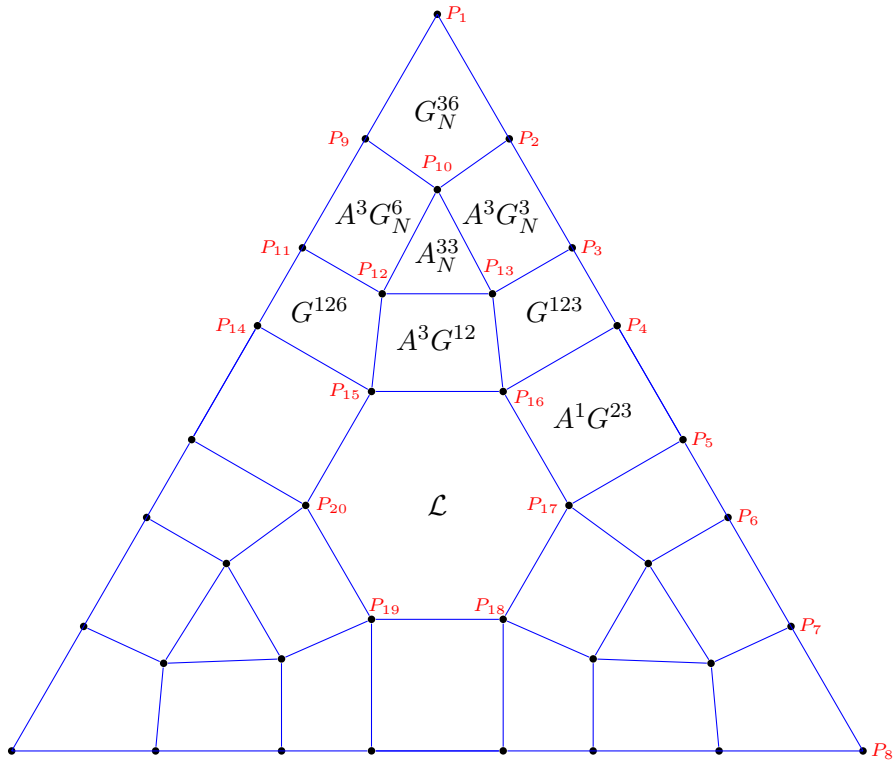


In order to streamline the process of specifying the  vertices $P_1, \ldots, P_{20}$ in Figure~\ref{fig:triangle-vertices}, we introduce some notation here.

\begin{defn}[Padded elements]\label{defn:padded}
Given an element $g \in G$ and a subset $A \subseteq \{1, \dots, 2n\}$, we define the {\em padded element} $g_A$ to be the unique element of $G$ such that
    \[
    \grpproj_i(g_A) = \begin{cases}
        \grpproj_i(g) & i \in A, \\
        e & \text{ else}
    \end{cases}.
    \]
\end{defn}

\begin{exmp}
As a concrete example, if $g = (g_1, g_2, \dots, g_{2n})$, then $g_{\{1,2\}} = (g_1, g_2, e, \dots, e)$. In the lemmas that follow, $A$ is some label sequence, e.g., $A= \alpha_6$ or $A = \alpha_{12}$. 

Note that there is a `calculus' of identities involving padded elements and label sequences. For example, if $a, b \in G$, then 
    \[
    a(a_{\alpha_6}^{-1}b_{\alpha_6}) \; =\;  (aa_{\alpha_6}^{-1})b_{\alpha_6} \; =\; a_{\alpha_{12345}}b_{\alpha_6} \, .
    \]
\end{exmp}

\begin{defn}[Streamlined encoding for group elements]\label{defn:la_elements} Suppose that  $A \subseteq \{1, \ldots, 2n\}$ of length $n$ is as in Definition~\ref{def:determinate} (for the  standard and non-standard face and edge groups) and  has  complement $B$. For any  $g \in G$ the element $g_A$ (defined in part (1) above) can be converted into an element of $\Ker(\Phi)$ by multiplying by the $B$-section lift of $-\Phi(g_A)$. 
This element is expressed as a product of $j$-section lifts for $j \in B$. The streamlined notation is to write this product as 
$[-\Phi(g_A)]$ times the appropriate $(2 \times 3)$ linear algebra matrix. 
\end{defn}

\begin{exmp}[Linear algebra encoding for elements in $G^6_N$]\label{exmp:G6N}
    For example, in the edge group $G^6_N$ the linear algebra encoding matrix is $\begin{pmatrix}
    1 & 2 & \dot{3}\\ 4& 5& 1
\end{pmatrix}$
and so the subset $B$ is as follows
$$
B \; =\; \alpha_{4} \, \cup \, \alpha_5 \, \cup \, \alpha_{12}[1, |\alpha_3|]. 
$$ 
Given $g \in G$, the $B$-section lift of $-\Phi(g_A)$ is 
$$
\prod_{j \in \alpha_{12}}s^j_{j+n}(-\Phi(g_A))\prod_{j \in \alpha_3}s^j_{j -|\alpha_{12}|}(-\Phi(g_A))\; =\; 
[-\Phi(g_A)]\begin{pmatrix}
    1 & 2 & \dot{3}\\ 4& 5& 1
\end{pmatrix} . 
$$
The explicit product of section lifts is on the left side, and the streamlined encoding is on the right side. 
\end{exmp}

The next example shows how the streamlined encoding will be used in the proof of  Lemma~\ref{lem:triside-labdist} and Lemma~\ref{lem:triinner-labdist}  below.

\begin{exmp}[Application of streamlined encoding in $G^6_N$]\label{exmp:G6N2}  Suppose $a, b \in \Ker(\Phi)$ and we wish to convert the $\alpha_6$-coordinates of $a$ into those of $b$ using multiplication by an element of $G^6_N$. Here is how one achieves this. 

This change is achieved by multiplying the last $|\alpha_6|$ coordinates of $a$ by the element $a_{\alpha_6}^{-1}b_{\alpha_6}$. However, the element $a_{\alpha_6}^{-1}b_{\alpha_6}$ is likely not to belong to $\Ker(\Phi)$, and so we have to be careful in order to keep the factor in $G^6_N \leq \Ker(\Phi)$. 
Write the element $a_{\alpha_6}^{-1}b_{\alpha_6}$ as a product of words in the kernel-section generators of $G_j$ (where $j \in \alpha_6$), and now replace these kernel-section generators by vectorized versions from the group $G^6_N$. 

The new element is 
$$
a(a_{\alpha_6}^{-1}b_{\alpha_6})[-\Phi(a_{\alpha_6}^{-1}b_{\alpha_6})]\begin{pmatrix}
    1 & 2 & \dot{3}\\ 4& 5& 1
\end{pmatrix} 
$$
which lies in $\Ker(\Phi)$. Note that  the multiplying factor
$$
(a_{\alpha_6}^{-1}b_{\alpha_6})[-\Phi(a_{\alpha_6}^{-1}b_{\alpha_6})]\begin{pmatrix}
    1 & 2 & \dot{3}\\ 4& 5& 1
\end{pmatrix} 
$$
is an element of $G^6_N$, and that the linear algebra term 
$$ 
[-\Phi(a_{\alpha_6}^{-1}b_{\alpha_6})]\begin{pmatrix}
    1 & 2 & \dot{3}\\ 4& 5& 1
\end{pmatrix} 
$$
is the streamlined form of the product expression in Example~\ref{exmp:G6N} for $g_A = a_{\alpha_6}^{-1}b_{\alpha_6}$. 
    
\end{exmp}

We now state the lemma behind the first part  of the construction of the actualization $\Delta(a,b,c)$.  

\begin{lem}[Filling 0-spheres with 7-segment edge paths]\label{lem:triside-labdist}
Given an ordered triple $(a,b,c)$ of elements of $\Ker(\Phi)$, and an assignment of $a$, $b$, $c$ to vertices of the algebraic triangle as in Definition~\ref{def:actual-alg-tri}, there exist edge paths $\partial{\Delta}(a,b)$ from $a$ to $b$, $\partial{\Delta}(b,c)$ from $b$ to $c$, and $\partial{\Delta}(c,a)$ from $c$ to $a$. Each of these edge paths is composed of 7 segments (indexed by the 7 boundary edges of the algebraic triangle from the first coordinate to the second), each of which  is a geodesic path in the Cayley graph of the corresponding edge group.
\end{lem}


    

\begin{proof}
We give a detailed construction of the 7-segment path $\partial\Delta(a,b)$ here. The constructions of the paths $\partial\Delta(b,c)$ and $\partial\Delta(c,a)$  differ primarily in notation, and details are left to the reader. 

Intuitively, one needs to convert the $a$-entries (in 6 groups indexed by the $\alpha_i$) into $b$-entries. We work from each end of the right side of the algebraic triangle (in Figure~\ref{fig:tri-1cell-label}) towards the middle as follows. We use the labeling of vertices shown in Figure~\ref{fig:triangle-vertices}. 
\begin{itemize}
    \item The edge adjacent to $a$ is used to convert the $\alpha_6$-entries of $a$ into those of $b$. This is possible in the non-standard edge subgroup $G^6_N$. 
    \item The second edge is responsible for converting a non-standard linear algebra arrangement $\begin{pmatrix}
    \dot{3}\\1
\end{pmatrix}$ into the standard  $\begin{pmatrix}
    3\\6
\end{pmatrix}$ arrangement.  This is achieved in the non-standard edge group $A^3_N(1,3)$. 
\item The third edge is used to  convert the $\alpha_{12}$-entries of $a$ (appropriately modified by the effects of the the previous two edges) into those of $b$. This is achieved in the standard edge group $G^{12}$. 
\item Similarly, working from the $b$-vertex, one uses $G^5_N$ to change the $\alpha_5$-entries of $b$ into those of $a$. 
\item Next, the non-standard edge group $A^2_N(2,3)$ is used to change non-standard linear algebra $\begin{pmatrix}
    \dot{2}\\3
\end{pmatrix}$ into standard linear algebra $\begin{pmatrix}
    2\\5
\end{pmatrix}$.
\item Thirdly, one changes the $\alpha_{34}$-entries of (the modified version of) $b$ into those of $a$ using the standard edge group $G^{34}$. 
\item Finally, the middle edge group $A^1(1,4)$ effects a standard linear algebra conversion $\begin{pmatrix}
    1\\1/4
\end{pmatrix}$, and one checks that the resulting group element agrees with the element obtained in the third bullet item above.  
\end{itemize}

Here are the details. Label the end points of these segments by 
    \[
    a \;=\; P_1 \, , \, P_2 \, , \, \ldots \, , \, P_7 \, , \, P_8 \;=\; b 
    \]
as indicated in Figure~{\ref{fig:triangle-vertices}}. 

{\em The segment $\overline{P_1P_2}$.} The purpose of this segment is to convert the last $|\alpha_6|$ entries of $a$ into those of $b$. This is achieved using a word in the generators of  $G^6_{N}$ of length $d_{\alpha_6}(a,b)$. Following Example~\ref{exmp:G6N2}, the endpoints of this segment are $P_1 = a$ and 
    \[
    P_2 \;=\; a_{\alpha_{12345}}b_{\alpha_6}[-\Phi(a_{\alpha_6}^{-1}b_{\alpha_6})]{\scalebox{0.75}{$\begin{pmatrix} 1 & 2 & \dot{3} \\ 4 & 5 & 1 \end{pmatrix}$}} \, .
    \]
This segment has length $\ell_{1,2} = d_{\alpha_6}(a,b)$.



{\em The segment $\overline{P_2P_3}$.} The purpose of this segment is to change the nonstandard linear algebra \scalebox{0.75}{$\begin{pmatrix} \dot{3} \\  1 \end{pmatrix}$} of the element $P_2$ into standard linear algebra \scalebox{0.75}{$\begin{pmatrix} 3 \\  6 \end{pmatrix}$} . This is achieved by using at most $d_{\alpha_6}(a,b)$ generators of $A^3_N(1,6)$. This upper bound is seen by noticing that the  non-standard \scalebox{0.75}{$\begin{pmatrix} \dot{3} \\  1 \end{pmatrix}$} elements arose as a result of the $\overline{P_1P_2}$  segment, and so their number is bounded by $\ell_{1,2} = d_{\alpha_6}(a,b)$.   The new endpoint is
    \[
    P_3 \;=\; a_{\alpha_{12345}}b_{\alpha_6}\left[-\Phi\left(a_{\alpha_6}^{-1}b_{\alpha_6}\right)\right]{\scalebox{0.75}{$\begin{pmatrix} 1 & 2 & 3 \\ 4 & 5 & 6 \end{pmatrix}$}} \, . 
    \]
This segment has length $\ell_{2,3} \leq \ell_{1,2} \leq d_{\alpha_6}(a,b)$.

{\em The segment $\overline{P_3P_4}$.} The purpose of this segment is to change  the first $|\alpha_{12}|$ entries of the element $P_3$ into those of $b$. Since the $\alpha_{12}$-entries of $P_3$ are identical with those of $a$, this is achieved using a word in the generators of $G^{12}$ of length $d_{\alpha_{12}}(a,b)$. The new endpoint is
    \[
    P_4 \;=\; a_{\alpha_{345}}b_{\alpha_{126}}\left[-\smashoperator{\sum_{i \in \{1,2,6\}}}\Phi\left(a_{\alpha_i}^{-1}b_{\alpha_i}\right)\right]{\scalebox{0.75}{$\begin{pmatrix} 1 & 2 & 3 \\ 4 & 5 & 6 \end{pmatrix}$}} \, .
    \]
This segment has length $\ell_{3,4} = d_{\alpha_{12}}(a,b)$.

{\em The segment $\overline{P_8P_7}$.} The purpose of this  segment is to change  the $\alpha_5$-entries  of $b$ into thiose of $a$. This is achieved using a word in the generators of $G^{5}_{N}$ of length $d_{\alpha_{5}}(a,b)$. The endpoints of this segment are $P_8 = b$ and 
    \[
    P_7 \;=\; b_{\alpha_{12346}}a_{\alpha_5}\left[-\Phi\left(b_{\alpha_5}^{-1}a_{\alpha_5}\right)\right]{\scalebox{0.75}{$\begin{pmatrix} 1 & \dot{2} & 3 \\ 1 & 3 & 6 \end{pmatrix}$}} \, .
    \]
This segment has length $\ell_{7,8} = d_{\alpha_5}(a,b)$.

{\em The segment $\overline{P_7P_6}$.} The purpose of this segment is to change the nonstandard linear algebra \scalebox{0.75}{$\begin{pmatrix} \dot{2} \\ 3 \end{pmatrix}$} of the element $P_7$ into standard linear algebra \scalebox{0.75}{$\begin{pmatrix}  2 \\ 5 \end{pmatrix}$}. An argument similar to the one used  for segment $\overline{P_2P_3}$ shows that is is achieved by using at most $\ell_{7,8} \leq d_{\alpha_5}(a,b)$ generators of $A^2_N(3,5)$. The new endpoint is
    \[
    P_6 \;=\; b_{\alpha_{12346}}a_{\alpha_5}\left[-\Phi\left(b_{\alpha_5}^{-1}a_{\alpha_5}\right)\right]{\scalebox{0.75}{$\begin{pmatrix} 1 & 2 & 3 \\ 1 & 5 & 6 \end{pmatrix}$}} \, .
    \]
This segment has length $\ell_{6,7} \leq \ell_{7,8} \leq d_{\alpha_5}(a,b)$.

{\em The segment $\overline{P_6P_5}$.} The purpose of this segment is to change  the $\alpha_{34}$-entries of  the element $P_6$ into those of  $a$. Since these entries of $P_6$ agree with those of $b$, this is achieved by using a word in the generators of $G^{34}$ of length $d_{\alpha_{34}}(a,b)$. The new endpoint is
    \[
    P_5 \;=\; b_{\alpha_{126}}a_{\alpha_{345}}\left[-\smashoperator{\sum_{i \in \{3,4,5\}}}\Phi\left(b_{\alpha_i}^{-1}a_{\alpha_i}\right)\right]{\scalebox{0.75}{$\begin{pmatrix} 1 & 2 & 3 \\ 1 & 5 & 6 \end{pmatrix}$}} \, .
    \]
This segment has length $\ell_{5,6} = d_{\alpha_{34}}(a,b)$.

{\em The segment $\overline{P_5P_4^{'}}$.} The purpose of this segment is to  change the standard linear algebra \scalebox{0.75}{$\begin{pmatrix} 1 \\ 1 \end{pmatrix}$} of the element $P_5$ into standard linear algebra \scalebox{0.75}{$\begin{pmatrix} 1 \\ 4  \end{pmatrix}$} of the element $P_4$. This is achieved by using at most $d_{\alpha_{345}}(a,b)$ generators of $A^1(1,4)$. This  bound comes from estimating the number of \scalebox{0.75}{$\begin{pmatrix} 1 \\ 1 \end{pmatrix}$} generators that can arise along the path $P_8 - P_7 -P_6 -P_5$ above.  The new endpoint is
    \[
    P_4^{'} = b_{\alpha_{126}}a_{\alpha_{345}}\left[-\smashoperator{\sum_{i \in \{3,4,5\}}}\Phi\left(b_{\alpha_i}^{-1}a_{\alpha_i}\right)\right]{\scalebox{0.75}{$\begin{pmatrix} 1 & 2 & 3 \\ 4 & 5 & 6 \end{pmatrix}$}} \, .
    \]
This segment has length $\ell_{4,5} \leq d_{\alpha_{345}}(a,b)$.

{\em Showing $P_4^{'} = P_4$.} We finish our argument by showing that $P_4 = P_4^{'}$. First, the padded elements commute 
$$
b_{\alpha_{126}}a_{\alpha_{345}} \; =\; a_{\alpha_{345}}b_{\alpha_{126}}
$$
because they are non-trivial on complementary ranges, 
and so we only have to check the linear algebra terms. 
For the linear algebra, notice that 
$$
\sum_{j \in \{3,4,5\}}\Phi(b_{\alpha_j}^{-1}a_{\alpha_j}) \; =\; -\sum_{j \in \{3,4,5\}}\Phi(a_{\alpha_j}^{-1}b_{\alpha_j}) \; =\; \sum_{i \in \{1,2,6\} }\Phi(a_{\alpha_i}^{-1}b_{\alpha_i})
$$
and so $P'_4 = P_4$. 
The first equality in the displayed equation above is obtained by taking inverses. The second equality follows from the identity 
$$
\sum_{i \in \{1, \ldots , 6\}}\Phi(a_{\alpha_i}^{-1}b_{\alpha_i}) \; =\; 0
$$
which is a long hand version of $\Phi(a^{-1}b) = 0$. This holds since $a, b \in \Ker(\Phi)$ imply $a^{-1}b \in \Ker(\Phi)$. 

One also observes that $\ell_{4,5} \leq d_{\alpha_{126}}(a,b)$. We will work with the bound $d_{\alpha_{345}}(a,b)$ as proved before.
\end{proof}

Since the $a,b \in \Ker(\Phi)$ in the argument above are arbitrary, this shows that the subset of $\overline{X}$ in Definition~\ref{def:kersecgens} consisting of the union of the generating sets of the 7 edge groups corresponding to the edges  on the right side of the triangle generates $\Ker(\Phi)$. A similar argument shows that the subsets corresponding to the left side or to the base of the triangle also generate $\Ker(\Phi)$. This proves  the following. 

\begin{cor}[Finite generation of $\Ker(\Phi)$]\label{cor:kerfingen}
The set $\overline{X}$ of Definition~\ref{def:kersecgens} generates $\Ker(\Phi)$.
\end{cor}

\begin{rem}\label{record1}[Record of  segment length bounds for Lemma~\ref{lem:triside-labdist}] We record the upper bounds for the  segment lengths obtained in the proof of Lemma~\ref{lem:triside-labdist} for the reader's convenience. 

    \begin{center}
    \begin{longtblr}{
        colspec = {|c|c||c|c|},
        rowhead = 1, rows={abovesep=2pt,belowsep=2pt},
    }
    \hline
    {Segment Length} & {Upper Bound} & {Segment Length} & {Upper Bound} \\
    \hline
    \hline
    $\ell_{1,2}$ & $d_{\alpha_6} (a,b)$ & $\ell_{2,3}$ & $d_{\alpha_6} (a,b)$ \\
    \hline
    $\ell_{3,4}$ & $d_{\alpha_{12}} (a,b)$ & $\ell_{4,5}$ & $d_{\alpha_{345}} (a,b)$ \\
    \hline
    $\ell_{5,6}$ & $d_{\alpha_{34}} (a,b)$ & $\ell_{6,7}$  & $d_{\alpha_{5}} (a,b)$ \\
    \hline
    $\ell_{7,8}$  & $d_{\alpha_{5}} (a,b)$ & \hphantom{xyz} &  \\
    \hline
    \end{longtblr}
    \end{center}
    
\end{rem}

\begin{rem}[1-fillings of 0-spheres and 2-fillings of loops] The preceding lemma is a one-dimensional result. It describes how to build specific $1$-fillings of the $0$-spheres $\{a,b\}$, $\{b,c\}$, and $\{c,a\}$, and provides an upper bound estimate for the lengths of these fillings.  

The next lemma is two-dimensional. It describes how to build a $2$-filling (van Kampen diagram filling) of the $1$-cycle (loop) $\partial\Delta(a,b,c)$, defined below. This filling is obtained by first decomposing the loop into a product of 25 loops and establishing upper bounds on the lengths of each of these loops. These loops are filled in the Cayley complexes of the 25 face groups and upper bounds on their filling area are provided by the isomorphism types of these face groups.  
\end{rem}

\begin{lem}[Twenty-five loop decomposition and van Kampen filling of $\partial\Delta(a,b,c)$]\label{lem:triinner-labdist}
Given an ordered triple $(a,b,c)$ of elements of $ \Ker(\Phi)$ and edge paths $\partial{\Delta}(a,b)$, $\partial{\Delta}(b,c)$, and $\partial{\Delta}(c,a)$ constructed as in Lemma~\ref{lem:triside-labdist}, the edge group labels in Figure~\ref{fig:tri-1cell-label} give a procedure for constructing a subdivision of the loop
    \[
    \partial{\Delta}(a,b,c) \coloneqq \partial{\Delta}(a,b) \cup \partial{\Delta}(b,c) \cup \partial{\Delta}(c,a)
    \]
into twenty-five loops indexed by the $2$-cells of the algebraic triangle $\Delta$, each of which is an edge path loop made up of geodesic segments in the Cayley graph of the corresponding face group. 

Furthermore, the lengths of these geodesic segments can be bounded by $4D$, where $D \coloneqq d_G(a,b) + d_G(b,c) + d_G(c,a)$.
\end{lem}

\begin{proof}
We  describe how to fill in one-third of the diagram in Figure~\ref{fig:triangle-vertices} explicitly, and appeal to the $3$-fold rotational symmetry of the setup in order to fill in the remaining two-thirds. In particular, we consider the region corresponding to the top seven $2$-cells in Figure~\ref{fig:triangle-vertices}, the $2$-cell labeled by $A^1G^{23}$, and the central hexagonal $2$-cell.

Vertices $P_1, \ldots, P_8$ are already specified in the proof of Lemma~\ref{lem:triside-labdist}, as well as  the geodesic edge-paths connecting them. We continue  labeling  vertices and building edge-paths below. Our discussion is indexed by the $2$-cells, working from top to bottom.  


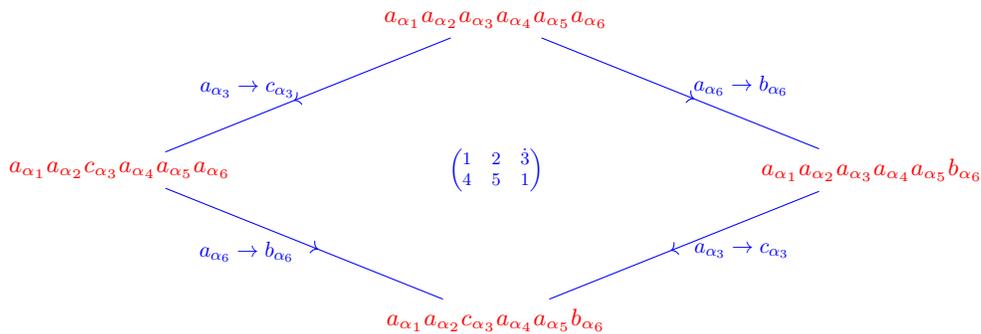
\begin{figure}[h]
\centering
    \begin{tikzpicture}
    \node (P1) at (0,4) {\footnotesize\textcolor{red}{$a_{\alpha_1}a_{\alpha_2}a_{\alpha_3}a_{\alpha_4}a_{\alpha_5}a_{\alpha_6}$}};
    \node (P2) at (5,2) {\footnotesize\textcolor{red}{$a_{\alpha_1}a_{\alpha_2}a_{\alpha_3}a_{\alpha_4}a_{\alpha_5}b_{\alpha_6}$}};
    \node (P9) at (-5,2) {\footnotesize\textcolor{red}{$a_{\alpha_1}a_{\alpha_2}c_{\alpha_3}a_{\alpha_4}a_{\alpha_5}a_{\alpha_6}$}};
    \node (P10) at (0,0) {\footnotesize\textcolor{red}{$a_{\alpha_1}a_{\alpha_2}c_{\alpha_3}a_{\alpha_4}a_{\alpha_5}b_{\alpha_6}$}};
    \node (LAlabel) at (0,2) {\footnotesize\textcolor{blue}{$\scalebox{0.75}{$\begin{pmatrix} 1 & 2 & \dot{3} \\ 4 & 5 & 1 \end{pmatrix}$}$}};
    
    \begin{scope}[color=blue,decoration={
    markings,
    mark=at position 0.55 with {\arrow{>}}}
    ]
        \draw[postaction={decorate}] (P1) -- (P2) node[midway, above right=-1ex and 0.25ex] {\scalebox{0.75}{$a_{\alpha_6} \to b_{\alpha_6}$}};
        \draw[postaction={decorate}] (P1) -- (P9) node[midway, above left=-1ex and 0.25ex] {\scalebox{0.75}{$a_{\alpha_3} \to c_{\alpha_3}$}};
        \draw[postaction={decorate}] (P2) -- (P10) node[midway, below right=-1ex and 0ex] {\scalebox{0.75}{$a_{\alpha_3} \to c_{\alpha_3}$}};
        \draw[postaction={decorate}] (P9) -- (P10) node[midway, below left=-1ex and 0ex] {\scalebox{0.75}{$a_{\alpha_6} \to b_{\alpha_6}$}};
    \end{scope}
\end{tikzpicture}
\caption{Vertex labels, coordinate changes, and linear algebra framework for the $G^{36}_N$ square of $\Delta(a,b,c)$.}
\label{fig:vertex-labels-g36n}
\end{figure}

\begin{rem}[Notation in figures\ref{fig:vertex-labels-g36n} through~\ref{fig:vertex-labels-a1g23-L}]
The vertex labels are not complete. They are simply products of padded elements. Their purpose is to help the reader keep track of which of the six $\alpha_i$-segments have entries from $a$ or $b$ or $c$. The precise group elements corresponding to each vertex is given in the details of the proof below. 

The edge labels consist of two types; either a change of entries in some $\alpha_i$-segment, or a change in linear algebra. The former are denoted by, for example, $a_{\alpha_i} \to b_{\alpha_i}$. The reader should keep in mind that such changes are effected by elements of $\Ker(\Phi)$ and there is a  linear algebra choice in the lifts of section generators. This linear algebra choice is indicated by an accompanying $(2\times 3)$-matrix; sometimes one such linear algebra choice works for all edges of the diagram, in which case the associated $(2 \times 3)$-matrix is typeset in the center of the region. 

The other purpose of an edge is to effect a change in linear algebra from non-standard to standard or vice versa, or between two standard choices. These are represented by an arrow between two $(2 \times 1)$ column vectors. 

\end{rem}

({\em $G^{36}_N$ square $2$-cell.}) The square $2$-cell indexed by the subgroup $G^{36}_N$ has vertices $P_1$, $P_2$, $P_9$, and $P_{10}$ as shown in Figure~\ref{fig:vertex-labels-g36n}. Vertices $P_1$ and $P_2$ were computed explicitly in the proof of Lemma~\ref{lem:triside-labdist}; vertex $P_9$ also follows from that lemma, but is not explicitly computed there. We do so here. 
    \begin{align*}
    P_1 &= a \;=\; a_{\alpha_{123456}} \, , \\
    P_2 &= a_{\alpha_{12345}}b_{\alpha_6} \left[-\Phi(a_{\alpha_6}^{-1}b_{\alpha_6})\right]{\scalebox{0.75}{$\begin{pmatrix} 1 & 2 & \dot{3} \\ 4 & 5 & 1 \end{pmatrix}$}} \\   
    &= P_1 \cdot a_{\alpha_6}^{-1}b_{\alpha_6} \left[-\Phi(a_{\alpha_6}^{-1}b_{\alpha_6})\right]{\scalebox{0.75}{$\begin{pmatrix} 1 & 2 & \dot{3} \\ 4 & 5 & 1 \end{pmatrix}$}} \, , \\
    P_9 &= a_{\alpha_{12345}}c_{\alpha_3} \left[-\Phi(a_{\alpha_3}^{-1}c_{\alpha_3})\right]{\scalebox{0.75}{$\begin{pmatrix} 1 & 2 & \dot{3} \\ 4 & 5 & 1 \end{pmatrix}$}} \\   
    &= P_1 \cdot a_{\alpha_3}^{-1}c_{\alpha_3} \left[-\Phi(a_{\alpha_3}^{-1}c_{\alpha_3})\right]{\scalebox{0.75}{$\begin{pmatrix} 1 & 2 & \dot{3} \\ 4 & 5 & 1 \end{pmatrix}$}} \, , \\
    P_{10} &= a_{\alpha_{1245}}b_{\alpha_6}c_{\alpha_3} \left[-\Phi\left(a_{\alpha_3}^{-1}c_{\alpha_3}\right)-\Phi\left(a_{\alpha_6}^{-1}b_{\alpha_6}\right)\right]{\scalebox{0.75}{$\begin{pmatrix} 1 & 2 & \dot{3} \\ 4 & 5 & 1 \end{pmatrix}$}} \\
    &= P_2 \cdot a_{\alpha_3}^{-1}c_{\alpha_3} \left[-\Phi(a_{\alpha_3}^{-1}c_{\alpha_3})\right]{\scalebox{0.75}{$\begin{pmatrix} 1 & 2 & \dot{3} \\ 4 & 5 & 1 \end{pmatrix}$}} \\
    &= P_9 \cdot a_{\alpha_6}^{-1}b_{\alpha_6} \left[-\Phi(a_{\alpha_6}^{-1}b_{\alpha_6})\right]{\scalebox{0.75}{$\begin{pmatrix} 1 & 2 & \dot{3} \\ 4 & 5 & 1 \end{pmatrix}$}} \, .
    \end{align*}
The fact that the two paths $P_1 \to P_2 \to P_{10}$ and $P_1 \to P_9 \to P_{10}$ form a closed loop corresponds to the commutation relation 
    \begin{align*}
    &(a_{\alpha_6}^{-1}b_{\alpha_6}) \left[-\Phi(a_{\alpha_6}^{-1}b_{\alpha_6})\right] {\scalebox{0.75}{$\begin{pmatrix} 1 & 2 & \dot{3} \\ 4 & 5 & 1 \end{pmatrix}$}}(a_{\alpha_3}^{-1}c_{\alpha_3}) \left[-\Phi(a_{\alpha_3}^{-1}c_{\alpha_3})\right]{\scalebox{0.75}{$\begin{pmatrix} 1 & 2 & \dot{3} \\ 4 & 5 & 1 \end{pmatrix}$}} \\
    &= (a_{\alpha_3}^{-1}c_{\alpha_3}) \left[-\Phi(a_{\alpha_3}^{-1}c_{\alpha_3})\right] {\scalebox{0.75}{$\begin{pmatrix} 1 & 2 & \dot{3} \\ 4 & 5 & 1 \end{pmatrix}$}}(a_{\alpha_6}^{-1}b_{\alpha_6}) \left[-\Phi(a_{\alpha_6}^{-1}b_{\alpha_6})\right] {\scalebox{0.75}{$\begin{pmatrix} 1 & 2 & \dot{3} \\ 4 & 5 & 1 \end{pmatrix}$}} \, ,
    \end{align*}
which holds in $G_N^{36} \cong G_{\alpha_3} \times G_{\alpha_6}$. Concretely, the $(a_{\alpha_3}^{-1}c_{\alpha_3})$ term in the first row commutes to the left past the linear algebra and the $(a_{\alpha_6}^{-1}b_{\alpha_6})$ terms because both of those terms have identity elements in their $\alpha_3$-coordinates and the padded element $(a_{\alpha_3}^{-1}c_{\alpha_3})$ has the identity element in every coordinate outside of the $\alpha_3$-coordinates. Next, the two linear algebra terms commute, since they are lifts of commuting elements from $\Z^m$. Finally, the $(a_{\alpha_6}^{-1}c_{\alpha_6})$ term commutes to the right past the $(a_{\alpha_3}^{-1}c_{\alpha_3})$-linear algebra term for the same reason. We do not spell out the relations this explicitly for the other squares below. 

There is a van Kampen diagram for this commutation identity in the Cayley $2$-complex of $G^{36}_N$. We take the $a$-translate of this van Kampen diagram corresponding to the coset $aG^{23}_N \subseteq \Ker(\Phi)$; this is one of twenty-five pieces which make up the van Kampen diagram for $\partial\Delta(a,b,c)$ in the Cayley $2$-complex of $\Ker(\Phi)$. We shall drop explicit reference to these coset translates in our discussion of the remaining 2-cells below.

({\em Segment bounds in $G^{36}_N$ square $2$-cell.}) In Lemma~\ref{lem:triside-labdist}, one gets $\ell_{1,2} \leq d_{\alpha_6}(a,b)$ (see Remark~\ref{lem:triside-labdist}). By similar reasoning, we get $\ell_{9,10} \leq d_{\alpha_6}(a,b)$. Again, using generators form $G^3_N$, one goes from $P_1$ to $P_9$ and from $P_2$ to $P_{10}$. This provides the following the upper bounds: $\ell_{1,9} \leq d_{\alpha_3}(a,c)$ and $\ell_{2,10} \leq d_{\alpha_3}(a,c)$.

\begin{figure}[h]
\centering
    \begin{tikzpicture}[scale=0.95]
    \node (P10) at (0,3) {\footnotesize\textcolor{red}{$a_{\alpha_1}a_{\alpha_2}c_{\alpha_3}a_{\alpha_4}a_{\alpha_5}b_{\alpha_6}$}};
    \node (P12) at (-3,0) {\footnotesize\textcolor{red}{$a_{\alpha_1}a_{\alpha_2}c_{\alpha_3}a_{\alpha_4}a_{\alpha_5}b_{\alpha_6}$}};
    \node (P13) at (3,0) {\footnotesize\textcolor{red}{$a_{\alpha_1}a_{\alpha_2}c_{\alpha_3}a_{\alpha_4}a_{\alpha_5}b_{\alpha_6}$}};
    
    \node (P2) at ($(P10)+(30:3)$) {\footnotesize\textcolor{red}{$a_{\alpha_1}a_{\alpha_2}a_{\alpha_3}a_{\alpha_4}a_{\alpha_5}b_{\alpha_6}$}};
    \node (P3) at ($(P13)+(30:3)$) {\footnotesize\textcolor{red}{$a_{\alpha_1}a_{\alpha_2}a_{\alpha_3}a_{\alpha_4}a_{\alpha_5}b_{\alpha_6}$}};
    \node (P9) at ($(P10)+(150:3)$) {\footnotesize\textcolor{red}{$a_{\alpha_1}a_{\alpha_2}c_{\alpha_3}a_{\alpha_4}a_{\alpha_5}a_{\alpha_6}$}};
    \node (P11) at ($(P12)+(150:3)$) {\footnotesize\textcolor{red}{$a_{\alpha_1}a_{\alpha_2}c_{\alpha_3}a_{\alpha_4}a_{\alpha_5}a_{\alpha_6}$}};

    \begin{scope}[color=blue,decoration={
    markings,
    mark=at position 0.55 with {\arrow{>}}}
    ]
        \draw[postaction={decorate}] (P2) -- (P3) node[midway, above right=-1ex and 0ex] {\scalebox{0.5}{$\begin{pmatrix} \dot{3} \\ 1 \end{pmatrix} \to \begin{pmatrix} 3 \\ 6 \end{pmatrix}$}};
        \draw[postaction={decorate}] (P2) -- (P10) node[midway, below right=-1ex and 0ex] {$\substack{$\scalebox{0.75}{$a_{\alpha_3} \to c_{\alpha_3}$}$ \\ $\scalebox{0.5}{$\begin{pmatrix} 1 & 2 & \dot{3} \\ 4 & 5 & 1 \end{pmatrix}$}$}$};
        \draw[postaction={decorate}] (P9) -- (P10) node[midway, below left=-1ex and 0ex] {$\substack{$\scalebox{0.75}{$a_{\alpha_6} \to b_{\alpha_6}$}$ \\ $\scalebox{0.5}{$\begin{pmatrix} 1 & 2 & \dot{3} \\ 4 & 5 & 1 \end{pmatrix}$}$}$};
        \draw[postaction={decorate}] (P3) -- (P13) node[midway, below right=-1ex and 0ex] {$\substack{$\scalebox{0.75}{$a_{\alpha_3} \to c_{\alpha_3}$}$ \\ $\scalebox{0.5}{$\begin{pmatrix} 1 & 2 & 3 \\ 4 & 5 & 3 \end{pmatrix}$}$}$};
        \draw[postaction={decorate}] (P10) -- (P13) node[midway, above right=0.5ex and -2.5ex] {\scalebox{0.5}{$\begin{pmatrix} \dot{3} \\ 1 \end{pmatrix} \to \begin{pmatrix} 3 \\ 6 \end{pmatrix}$}};
        \draw[postaction={decorate}] (P9) -- (P11) node[midway, above left=-1ex and 0ex] {\scalebox{0.5}{$\begin{pmatrix} \dot{3} \\ 1 \end{pmatrix} \to \begin{pmatrix} 3 \\ 3 \end{pmatrix}$}};
        \draw[postaction={decorate}] (P11) -- (P12) node[midway, below left=-1ex and 0ex] {$\substack{$\scalebox{0.75}{$a_{\alpha_6} \to b_{\alpha_6}$}$ \\ $\scalebox{0.5}{$\begin{pmatrix} 1 & 2 & 3 \\ 4 & 5 & 6 \end{pmatrix}$}$}$};
        \draw[postaction={decorate}] (P10) -- (P12) node[midway, above left=0.5ex and -2.5ex] {\scalebox{0.5}{$\begin{pmatrix} \dot{3} \\ 1 \end{pmatrix} \to \begin{pmatrix} 3 \\ 3 \end{pmatrix}$}};
        \draw[postaction={decorate}] (P12) -- (P13) node[midway, below] {\scalebox{0.5}{$\begin{pmatrix} 3 \\ 3 \end{pmatrix} \to \begin{pmatrix} 3 \\ 6 \end{pmatrix}$}};
    \end{scope}
    \end{tikzpicture}
\caption{Vertex labels, coordinate changes, and linear algebra framework for the $A^3G^3_N$ and $A^3G^6_N$ squares and the $A^{33}_N$ triangle of $\Delta(a,b,c)$.}
\label{fig:vertex-labels-a3g3n-a3g6n-a33n}
\end{figure}
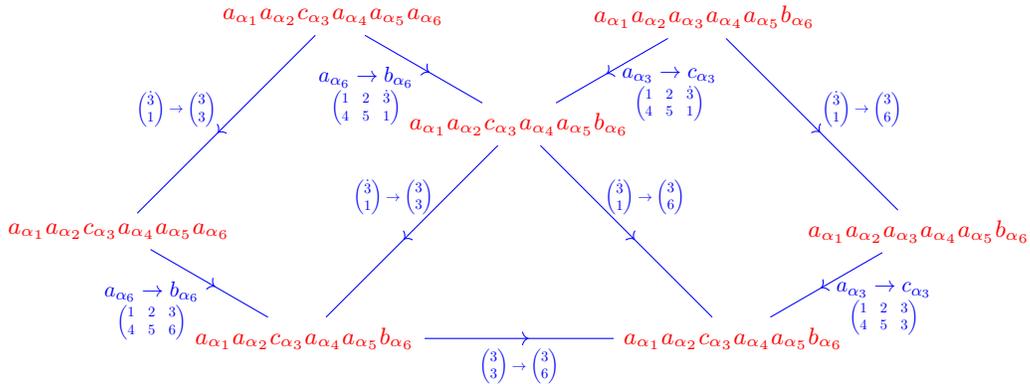

({\em $A^3G^3_N$ and $A^3G^6_N$ square $2$-cells.}) As seen in Figure~\ref{fig:vertex-labels-a3g3n-a3g6n-a33n} above, the square $2$-cell indexed by the subgroup $A^3G^3_N$ has the following four vertices: $P_2$ and $P_{10}$, which are described above, together with  
    \begin{align*}
    P_3 &= a_{\alpha_{12345}}b_{\alpha_6} \left[-\Phi\left(a_{\alpha_6}^{-1}b_{\alpha_6}\right)\right] {\scalebox{0.75}{$\begin{pmatrix} 1 & 2 & 3 \\ 4 & 5 & 6 \end{pmatrix}$}} \\
    &= P_2 \cdot \left[ \left[\Phi\left(a_{\alpha_6}^{-1}b_{\alpha_6}\right)\right]{\scalebox{0.75}{$\begin{pmatrix} \dot{3} \\ 1 \end{pmatrix}$}} +  \left[-\Phi\left(a_{\alpha_6}^{-1}b_{\alpha_6}\right)\right] {\scalebox{0.75}{$\begin{pmatrix} 3 \\ 6 \end{pmatrix}$}} \right] \, \text{ and } \\
    P_{13} &= a_{\alpha_{1245}}b_{\alpha_6}c_{\alpha_3} \left[-\Phi\left(a_{\alpha_3}^{-1}c_{\alpha_3}\right)-\Phi\left(a_{\alpha_6}^{-1}b_{\alpha_6}\right)\right] {\scalebox{0.75}{$\begin{pmatrix} 1 & 2 & 3 \\ 4 & 5 & 6 \end{pmatrix}$}} \\
    &= P_3 \cdot a_{\alpha_3}^{-1}c_{\alpha_3} \left[-\Phi(a_{\alpha_3}^{-1}c_{\alpha_3})\right] {\scalebox{0.75}{$\begin{pmatrix} 1 & 2 & 3 \\ 4 & 5 & 6 \end{pmatrix}$}} \\
    &= P_{10} \cdot \left[ \left[\Phi\left(a_{\alpha_3}^{-1}c_{\alpha_3}\right)\right]{\scalebox{0.75}{$\begin{pmatrix} \dot{3} \\ 1 \end{pmatrix}$}} + \left[-\Phi\left(a_{\alpha_3}^{-1}c_{\alpha_3}\right)\right] {\scalebox{0.75}{$\begin{pmatrix} 3 \\ 6 \end{pmatrix}$}} + \left[\Phi\left(a_{\alpha_6}^{-1}b_{\alpha_6}\right)\right] {\scalebox{0.75}{$\begin{pmatrix} \dot{3} \\ 1 \end{pmatrix}$}} + \left[-\Phi\left(a_{\alpha_6}^{-1}b_{\alpha_6}\right)\right] {\scalebox{0.75}{$\begin{pmatrix} 3 \\ 6 \end{pmatrix}$}}\right] \, .
    \end{align*}
The fact that the two paths $P_2 \to P_{10} \to P_{13}$ and $P_2 \to P_3 \to P_{13}$ are coterminal corresponds to multiplication by 
$a_{\alpha_3}^{-1}c_{\alpha_3}$ together with the linear algebra relation 
    \begin{align*}
    \left[-\Phi(a_{\alpha_3}^{-1}c_{\alpha_3})\right] {\scalebox{0.75}{$\begin{pmatrix} 1 & 2 & \dot{3} \\ 4 & 5 & 1 \end{pmatrix}$}} &+ \left[\Phi\left(a_{\alpha_3}^{-1}c_{\alpha_3}\right)\right] {\scalebox{0.75}{$\begin{pmatrix} \dot{3} \\ 1 \end{pmatrix}$}} \\
    &+ \left[-\Phi\left(a_{\alpha_3}^{-1}c_{\alpha_3}\right)\right] {\scalebox{0.75}{$\begin{pmatrix} 3 \\ 6 \end{pmatrix}$}} + \left[\Phi\left(a_{\alpha_6}^{-1}b_{\alpha_6}\right)\right] {\scalebox{0.75}{$\begin{pmatrix} \dot{3} \\ 1 \end{pmatrix}$}} + \left[-\Phi\left(a_{\alpha_6}^{-1}b_{\alpha_6}\right)\right] {\scalebox{0.75}{$\begin{pmatrix} 3 \\ 6 \end{pmatrix}$}} \\
    &= \left[\Phi\left(a_{\alpha_6}^{-1}b_{\alpha_6}\right)\right] {\scalebox{0.75}{$\begin{pmatrix} \dot{3} \\ 1 \end{pmatrix}$}} +  \left[-\Phi\left(a_{\alpha_6}^{-1}b_{\alpha_6}\right)\right] {\scalebox{0.75}{$\begin{pmatrix} 3 \\ 6 \end{pmatrix}$}} + \left[-\Phi(a_{\alpha_3}^{-1}c_{\alpha_3})\right] {\scalebox{0.75}{$\begin{pmatrix} 1 & 2 & 3 \\ 4 & 5 & 6 \end{pmatrix}$}} \, ,
    \end{align*}
which holds in $A^3G^3_N \cong A_{\alpha_3} \times G_{\alpha_3}$.

Similarly, the square $2$-cell indexed by the subgroup $A^3G^6_N$ has the following four vertices: $P_9$ and $P_{10}$ as described above, and 
    \begin{align*}
    P_{11} &= a_{\alpha_{12456}}c_{\alpha_3} \left[-\Phi\left(a_{\alpha_3}^{-1}c_{\alpha_3}\right)\right] {\scalebox{0.75}{$\begin{pmatrix} 1 & 2 & 3 \\ 4 & 5 & 3 \end{pmatrix}$}} \, \text{ and } \\
    P_{12} &= a_{\alpha_{1245}}b_{\alpha_6}c_{\alpha_3}\left[-\Phi\left(a_{\alpha_3}^{-1}c_{\alpha_3}\right)-\Phi\left(a_{\alpha_6}^{-1}b_{\alpha_6}\right)\right]{\scalebox{0.75}{$\begin{pmatrix} 1 & 2 & 3 \\ 4 & 5 & 3 \end{pmatrix}$}} \\
    &= P_{10} \cdot \left[ \left[\Phi\left(a_{\alpha_3}^{-1}c_{\alpha_3}\right) + \Phi\left(a_{\alpha_6}^{-1}b_{\alpha_6}\right)\right]{\scalebox{0.75}{$\begin{pmatrix}  \dot{3} \\ 1 \end{pmatrix}$}} - \left[\Phi\left(a_{\alpha_3}^{-1}c_{\alpha_3}\right) + \Phi\left(a_{\alpha_6}^{-1}b_{\alpha_6}\right)\right]{\scalebox{0.75}{$\begin{pmatrix}  3 \\ 3 \end{pmatrix}$}} \right] \, .
    \end{align*}  
This $2$-cell is a mirror image of the $A^3G^3_N$ $2$-cell above (specifically, one replaces $a_{\alpha_6}$ by $b_{\alpha_6}$ using non-standard linear algebra $\begin{pmatrix} 1&2&\dot{3}\\4&5&1
\end{pmatrix}$ and also using standard  linear algebra $\begin{pmatrix} 1&2&3\\4&5&3
\end{pmatrix}$), and the boundary relation holds in $A^3G^6_N$ in an analogous fashion. 

({\em Segment bounds in $A^3G^3_N$ and $A^3G^6_N$ square $2$-cells.}) One gets $\ell_{2,3} \leq d_{\alpha_6}(a,b)$ from Lemma~\ref{lem:triside-labdist}. The purpose of the $P_{10,13}$ segment is to change the nonstandard linear algebra \scalebox{0.75}{$\begin{pmatrix} \dot{3} \\  1 \end{pmatrix}$} of the element $P_{10}$ into standard linear algebra \scalebox{0.75}{$\begin{pmatrix} 3 \\  6 \end{pmatrix}$} of $P_{13}$. This is achieved by using at most $d_{\alpha_3} (a,c) + d_{\alpha_6} (a,b)$ generators of $A^3_N(1,6)$. Hence, $\ell_{10,13} \leq d_{\alpha_3} (a,c) + d_{\alpha_6} (a,b)$. Finally, using generators from $G^3$, one converts the $\alpha_3$ coordinates from $a$ to $c$, which establishes $\ell_{3,13} \leq d_{\alpha_3}(a,c)$. This completes all edges bounding the $A^3G^3_N$ square. As described above, since $A^3G^6_N$ is a mirror image of $A^3G^3_N$ cell, one obtains the upper bound (recorded in Remark~\ref{rem:record-vertlength}) on the segments by similar reasoning. 

({\em $A^{33}_N$ triangular $2$-cell.}) The triangular $2$-cell indexed by $A^{33}_N$ has the following vertices: $P_{10}$, $P_{12}$, and $P_{13}$, all of which are described above.

The fact that the two paths $P_{10} \to P_{12} \to P_{13}$ and $P_{10} \to P_{13}$ are coterminal corresponds to the purely linear algebra relation
    \begin{align*}
    &\left[\Phi\left(a_{\alpha_3}^{-1}c_{\alpha_3}\right) + \Phi\left(a_{\alpha_6}^{-1}b_{\alpha_6}\right)\right] {\scalebox{0.75}{$\begin{pmatrix}  \dot{3} \\ 1 \end{pmatrix}$}} - \left[\Phi\left(a_{\alpha_3}^{-1}c_{\alpha_3}\right) + \Phi\left(a_{\alpha_6}^{-1}b_{\alpha_6}\right)\right] {\scalebox{0.75}{$\begin{pmatrix} 3 \\ 3 \end{pmatrix}$}} \\
    &+\left[\Phi\left(a_{\alpha_3}^{-1}c_{\alpha_3}\right) + \Phi\left(a_{\alpha_6}^{-1}b_{\alpha_6}\right)\right] {\scalebox{0.75}{$\begin{pmatrix} 3 \\ 3 \end{pmatrix}$}} - \left[\Phi\left(a_{\alpha_3}^{-1}c_{\alpha_3}\right) + \Phi\left(a_{\alpha_6}^{-1}b_{\alpha_6}\right)\right] {\scalebox{0.75}{$\begin{pmatrix} 3 \\ 6 \end{pmatrix}$}} \\
    &= \left[\Phi\left(a_{\alpha_3}^{-1}c_{\alpha_3}\right) + \Phi\left(a_{\alpha_6}^{-1}b_{\alpha_6}\right)\right] {\scalebox{0.75}{$\begin{pmatrix}  \dot{3} \\ 1 \end{pmatrix}$}} - \left[\Phi\left(a_{\alpha_3}^{-1}c_{\alpha_3}\right) + \Phi\left(a_{\alpha_6}^{-1}b_{\alpha_6}\right)\right] {\scalebox{0.75}{$\begin{pmatrix} 3 \\ 6 \end{pmatrix}$}} \, ,
    \end{align*}
which holds in the group $A^{33}_N \cong A_{\alpha_3} \times A_{\alpha_3}$.

({\em Segment bounds in $A^{33}_N$ triangular $2$-cell.}) We have established the upper bounds for the segments $P_{10,13}$ and $P_{10,12}$. The purpose of the $P_{12,13}$ segment is to change the standard linear algebra \scalebox{0.75}{$\begin{pmatrix} 3 \\  3 \end{pmatrix}$} of the element $P_{12}$ into standard linear algebra \scalebox{0.75}{$\begin{pmatrix} 3 \\  6 \end{pmatrix}$} of $P_{13}$ which is achieved by using at most $d_{\alpha_3} (a,c) + d_{\alpha_6} (a,b)$ generators of $A^3(3,6)$.

\begin{figure}[h]
\centering
    \begin{tikzpicture}
    \node (P3) at (10,7) {\tiny \textcolor{red}{$a_{\alpha_1}a_{\alpha_2}a_{\alpha_3}a_{\alpha_4}a_{\alpha_5}b_{\alpha_6}$}};
    \node (P4) at (10,4) {\tiny \textcolor{red}{$b_{\alpha_1}b_{\alpha_2}a_{\alpha_3}a_{\alpha_4}a_{\alpha_5}b_{\alpha_6}$}};
    \node (P11) at (0,7) {\tiny \textcolor{red}{$a_{\alpha_1}a_{\alpha_2}c_{\alpha_3}a_{\alpha_4}a_{\alpha_5}a_{\alpha_6}$}};
    \node (P12) at (2.5,5) {\tiny \textcolor{red}{$a_{\alpha_1}a_{\alpha_2}c_{\alpha_3}a_{\alpha_4}a_{\alpha_5}b_{\alpha_6}$}};
    \node (P13) at (7.5,5) {\tiny \textcolor{red}{$a_{\alpha_1}a_{\alpha_2}c_{\alpha_3}a_{\alpha_4}a_{\alpha_5}b_{\alpha_6}$}};
    \node (P14) at (0,4) {\tiny \textcolor{red}{$c_{\alpha_1}c_{\alpha_2}c_{\alpha_3}a_{\alpha_4}a_{\alpha_5}a_{\alpha_6}$}};
    \node (P15) at (2.5,2) {\tiny \textcolor{red}{$b_{\alpha_1}c_{\alpha_2}c_{\alpha_3}a_{\alpha_4}a_{\alpha_5}b_{\alpha_6}$}};
    \node (P16) at (7.5,2) {\tiny \textcolor{red}{$b_{\alpha_1}c_{\alpha_2}c_{\alpha_3}a_{\alpha_4}a_{\alpha_5}b_{\alpha_6}$}};
    \node (LAlabel1) at (1.25,4.5) {\footnotesize\textcolor{blue}{$\scalebox{0.7}{$\begin{pmatrix} 1 & 2 & 3 \\ 4 & 5 & 3 \end{pmatrix}$}$}};
    \node (LAlabel2) at (8.75,4.5) {\footnotesize\textcolor{blue}{$\scalebox{0.7}{$\begin{pmatrix} 1 & 2 & 3 \\ 4 & 5 & 6 \end{pmatrix}$}$}};
    
    \begin{scope}[color=blue,decoration={
    markings,
    mark=at position 0.55 with {\arrow{>}}}
    ]
        \draw[postaction={decorate}] (P3) -- (P4) node[midway, right] {$\substack{\scalebox{0.75}{$a_{\alpha_1} \to b_{\alpha_1}$} \\ \scalebox{0.75}{$a_{\alpha_2} \to b_{\alpha_2}$}}$};
        \draw[postaction={decorate}] (P3) -- (P13) node[midway, above left=-1ex and 0ex] {\scalebox{0.75}{$a_{\alpha_3} \to c_{\alpha_3}$}};
        \draw[postaction={decorate}] (P11) -- (P12) node[midway, above right=-1ex and 0ex] {\scalebox{0.75}{$a_{\alpha_6} \to b_{\alpha_6}$}};
        \draw[postaction={decorate}] (P12) -- (P13) node[midway, below] {\scalebox{0.5}{$\begin{pmatrix} 3 \\ 3 \end{pmatrix} \to \begin{pmatrix} 3 \\ 6 \end{pmatrix}$}};
        \draw[postaction={decorate}] (P11) -- (P14) node[midway, left] {$\substack{\scalebox{0.75}{$a_{\alpha_1} \to c_{\alpha_1}$} \\ \scalebox{0.75}{$a_{\alpha_2} \to c_{\alpha_2}$}}$};
        \draw[postaction={decorate}] (P14) -- (P15) node[midway, below left=-1ex and 0ex] {$\substack{\scalebox{0.75}{$c_{\alpha_1} \to b_{\alpha_1}$} \\ \scalebox{0.75}{$a_{\alpha_6} \to b_{\alpha_6}$}}$};
        \draw[postaction={decorate}] (P12) -- (P15) node[midway, right] {$\substack{\scalebox{0.75}{$a_{\alpha_1} \to b_{\alpha_1}$} \\ \scalebox{0.75}{$a_{\alpha_2} \to c_{\alpha_2}$}}$};
        \draw[postaction={decorate}] (P15) -- (P16) node[midway, below] {\scalebox{0.5}{$\begin{pmatrix} 3 \\ 3 \end{pmatrix} \to \begin{pmatrix} 3 \\ 6 \end{pmatrix}$}};
        \draw[postaction={decorate}] (P13) -- (P16) node[midway, left] {$\substack{\scalebox{0.75}{$a_{\alpha_1} \to b_{\alpha_1}$} \\ \scalebox{0.75}{$a_{\alpha_2} \to c_{\alpha_2}$}}$};
        \draw[postaction={decorate}] (P4) -- (P16) node[midway, below right=-1ex and 0ex] {$\substack{\scalebox{0.75}{$b_{\alpha_2} \to c_{\alpha_2}$} \\ \scalebox{0.75}{$a_{\alpha_3} \to c_{\alpha_3}$}}$};
    \end{scope}
    \end{tikzpicture}
\caption{Vertex labels, coordinate changes, and linear algebra framework for the $G^{123}$, $G^{126}$, and $A^3G^{12}$ squares of $\Delta(a,b,c)$.}
\label{fig:vertex-labels-g126-g123-a3g12}
\end{figure}

({\em $G^{123}$ and $G^{126}$ square $2$-cells.}) The square $2$-cell indexed by the subgroup $G^{123}$ has the following four vertices: $P_3$ and $P_{13}$ described above, plus
    \begin{align*}
    P_4 &= a_{\alpha_{345}}b_{\alpha_{126}}\left[-\smashoperator{\sum_{i \in \{1,2,6\}}}\Phi\left(a_{\alpha_i}^{-1}b_{\alpha_i}\right)\right]{\scalebox{0.75}{$\begin{pmatrix} 1 & 2 & 3 \\ 4 & 5 & 6 \end{pmatrix}$}} \\
    &= a_{\alpha_{345}}b_{\alpha_{126}}\left[-\smashoperator{\sum_{i \in \{3,4,5\}}}\Phi\left(b_{\alpha_i}^{-1}a_{\alpha_i}\right)\right]{\scalebox{0.75}{$\begin{pmatrix} 1 & 2 & 3 \\ 4 & 5 & 6 \end{pmatrix}$}}  \, \text{ and } \\
    P_{16} &= a_{\alpha_{45}}b_{\alpha_{16}}c_{\alpha_{23}}\left[-\smashoperator{\sum_{i \in \{2,3\}}}\Phi\left(b_{\alpha_i}^{-1}c_{\alpha_i}\right)-\smashoperator{\sum_{i \in \{4,5\}}} \Phi\left(b_{\alpha_i}^{-1}a_{\alpha_i}\right)\right]{\scalebox{0.75}{$\begin{pmatrix} 1 & 2 & 3 \\ 4 & 5 & 6 \end{pmatrix}$}} \, .
    \end{align*}
The fact that the two paths $P_3 \to P_4 \to P_{16}$ and $P_3 \to P_{13} \to P_{16}$ are coterminal corresponds to the relation 
    \begin{align*}
    &(a_{\alpha_{12}}^{-1}b_{\alpha_{12}}) \left[-\smashoperator{\sum_{i \in \{1,2 \}}}\Phi\left(a_{\alpha_i}^{-1}b_{\alpha_i}\right)\right]{\scalebox{0.75}{$\begin{pmatrix} 1 & 2 & 3 \\ 4 & 5 & 6 \end{pmatrix}$}} \\
    &\cdot
    (b_{\alpha_{2}}^{-1}c_{\alpha_{2}}) \left[-\Phi\left(b_{\alpha_2}^{-1}c_{\alpha_2}\right)\right]{\scalebox{0.75}{$\begin{pmatrix} 1 & 2 & 3 \\ 4 & 5 & 6 \end{pmatrix}$}}
    \cdot
    (a_{\alpha_{3}}^{-1}c_{\alpha_{3}}) \left[-\Phi\left(a_{\alpha_3}^{-1}c_{\alpha_3}\right)\right]{\scalebox{0.75}{$\begin{pmatrix} 1 & 2 & 3 \\ 4 & 5 & 6 \end{pmatrix}$}} \\
    &= (a_{\alpha_{3}}^{-1}c_{\alpha_{3}}) \left[-\Phi\left(a_{\alpha_3}^{-1}c_{\alpha_3}\right)\right]{\scalebox{0.75}{$\begin{pmatrix} 1 & 2 & 3 \\ 4 & 5 & 6 \end{pmatrix}$}}
    \cdot \\
    &(a_{\alpha_{1}}^{-1}b_{\alpha_{1}}) \left[-\Phi\left(a_{\alpha_1}^{-1}b_{\alpha_1}\right)\right]{\scalebox{0.75}{$\begin{pmatrix} 1 & 2 & 3 \\ 4 & 5 & 6 \end{pmatrix}$}}
    \cdot
    (a_{\alpha_{2}}^{-1}c_{\alpha_{2}}) \left[-\Phi\left(a_{\alpha_2}^{-1}c_{\alpha_2}\right)\right]{\scalebox{0.75}{$\begin{pmatrix} 1 & 2 & 3 \\ 4 & 5 & 6 \end{pmatrix}$}} \, ,
    \end{align*}
which holds in the group $G^{123} \cong G_{\alpha_1} \times G_{\alpha_2} \times G_{\alpha_3}$.

Next, the square $2$-cell indexed by the subgroup $G^{126}$ has the following four vertices: $P_{11}$ and $P_{12}$ above, and
    \begin{align*}
    P_{14} &= a_{\alpha_{456}}c_{\alpha_{123}} \left[-\smashoperator{\sum_{i \in \{1,2,3\}}} \Phi\left(a_{\alpha_i}^{-1}c_{\alpha_i}\right)\right] {\scalebox{0.75}{$\begin{pmatrix} 1 & 2 & 3 \\ 4 & 5 & 3 \end{pmatrix}$}} \, \text{ and } \\
    P_{15} &= a_{\alpha_{45}}b_{\alpha_{16}}c_{\alpha_{23}}\left[-\smashoperator{\sum_{i \in \{2,3\}}}\Phi\left(a_{\alpha_i}^{-1}c_{\alpha_i}\right)-\smashoperator{\sum_{i \in \{1,6\}}} \Phi\left(a_{\alpha_i}^{-1}b_{\alpha_i}\right)\right] {\scalebox{0.75}{$\begin{pmatrix} 1 & 2 & 3 \\ 4 & 5 & 3 \end{pmatrix}$}} \, .
    \end{align*}
This $2$-cell is a mirror image of the $G^{123}$ $2$-cell above (the underlying combinatorics involves interchanging the roles of $b$ and $c$ and also the roles of 3 and 6), and the boundary relation holds in $G^{126}$ in an analogous fashion.

({\em Segment bounds in $G^{123}$ and $G^{126}$ square $2$-cells.}) For the $G^{123}$ square 2-cell, we have already established the upper bounds for $P_{3,4}$ and $P_{3,13}$. The $P_{4,16}$ segment converts the $\alpha_3$ and $\alpha_2$ coordinates from $a$ to $c$ and from $b$ to $c$ respectively using generators from $G^{23}$. Thus, $\ell_{4,16} \leq d_{\alpha_3}(a,c) + d_{\alpha_2}(b,c)$. Finally, the $P_{13,16}$ segment converts the $\alpha_1$ and $\alpha_2$ coordinates from $a$ to $b$ and from $a$ to $c$ respectively using generators from $G^{12}$. Thus, $\ell_{13,16} \leq d_{\alpha_1}(a,b) + d_{\alpha_2}(a,c)$. Being a mirror image, parallel arguments hold true for the $G^{126}$ square $2$-cell.

({\em $A^3G^{12}$ square $2$-cell.}) The square $2$-cell indexed by the subgroup $A^3G^{12}$ has the following four vertices: $P_{12}$, $P_{13}$, $P_{16}$ and $P_{15}$, all described previously. The fact that the two paths $P_{12} \to P_{13} \to P_{16}$ and $P_{12} \to P_{15} \to P_{16}$ are coterminal corresponds to the relation
    \begin{align*}
    &\left[\left(\Phi\left(a_{\alpha_3}^{-1}c_{\alpha_3}\right) + \Phi\left(a_{\alpha_6}^{-1}b_{\alpha_6}\right)\right){\scalebox{0.75}{$\begin{pmatrix}  3 \\ 3 \end{pmatrix}$}} - \Phi\left(a_{\alpha_3}^{-1}c_{\alpha_3}\right) + \left(\Phi\left(a_{\alpha_6}^{-1}b_{\alpha_6}\right)\right) {\scalebox{0.75}{$\begin{pmatrix}  3 \\ 6 \end{pmatrix}$}} \right] \\
    &\cdot (a_{\alpha_{1}}^{-1}b_{\alpha_{1}}) \left[-\Phi\left(a_{\alpha_1}^{-1}b_{\alpha_1}\right)\right]{\scalebox{0.75}{$\begin{pmatrix} 1 & 2 & 3 \\ 4 & 5 & 6 \end{pmatrix}$}} \cdot (a_{\alpha_{2}}^{-1}c_{\alpha_{2}}) \left[-\Phi\left(a_{\alpha_2}^{-1}c_{\alpha_2}\right)\right]{\scalebox{0.75}{$\begin{pmatrix} 1 & 2 & 3 \\ 4 & 5 & 6 \end{pmatrix}$}} \\
    &=(a_{\alpha_{1}}^{-1}b_{\alpha_{1}}) \left[-\Phi\left(a_{\alpha_1}^{-1}b_{\alpha_1}\right)\right]{\scalebox{0.75}{$\begin{pmatrix} 1 & 2 & 3 \\ 4 & 5 & 3 \end{pmatrix}$}} \cdot (a_{\alpha_{2}}^{-1}c_{\alpha_{2}}) \left[-\Phi\left(a_{\alpha_2}^{-1}c_{\alpha_2}\right)\right]{\scalebox{0.75}{$\begin{pmatrix} 1 & 2 & 3 \\ 4 & 5 & 3 \end{pmatrix}$}} \\
    &\cdot\left( \left[ \;\, \smashoperator{\sum_{i \in \{2,3\}}}\Phi\left(a_{\alpha_i}^{-1}c_{\alpha_i}\right) + \smashoperator{\sum_{i \in \{1,6\}}} \Phi\left(a_{\alpha_i}^{-1}b_{\alpha_i}\right)\right]{\scalebox{0.75}{$\begin{pmatrix} 3 \\ 3 \end{pmatrix}$}} - \left[ \, \, \, \smashoperator{\sum_{i \in \{2,3\}}}\Phi\left(a_{\alpha_i}^{-1}c_{\alpha_i}\right) + \smashoperator{\sum_{i \in \{1,6\}}} \Phi\left(a_{\alpha_i}^{-1}b_{\alpha_i}\right)\right]{\scalebox{0.75}{$\begin{pmatrix} 3 \\ 6 \end{pmatrix}$}} \right) \, ,
    \end{align*}
which holds in the group $A^3G^{12} \cong G_{\alpha_1} \times G_{\alpha_2} \times A_{\alpha_3}$. 

({\em Three segment bounds in $A^3G^{12}$ square $2$-cell.}) Upper bounds for $P_{12,15}$, $P_{12,13}$ and $P_{13,16}$ are already established before. We postpone the computation of the upper bound for $P_{15,16}$ after the discussion of $\mathcal{L}$ hexagonal $2$-cell.

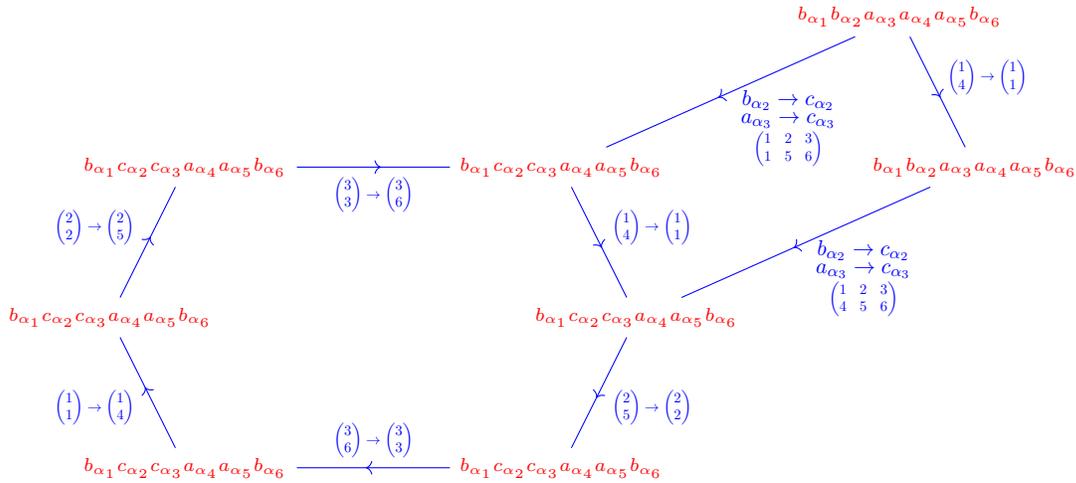
\begin{figure}[h]
\centering
    \begin{tikzpicture}
    \node (P4) at (12,4) {\tiny \textcolor{red}{$b_{\alpha_1}b_{\alpha_2}a_{\alpha_3}a_{\alpha_4}a_{\alpha_5}b_{\alpha_6}$}};
    \node (P5) at (13,2) {\tiny \textcolor{red}{$b_{\alpha_1}b_{\alpha_2}a_{\alpha_3}a_{\alpha_4}a_{\alpha_5}b_{\alpha_6}$}};
    \node (P15) at (2.5,2) {\tiny \textcolor{red}{$b_{\alpha_1}c_{\alpha_2}c_{\alpha_3}a_{\alpha_4}a_{\alpha_5}b_{\alpha_6}$}};
    \node (P16) at (7.5,2) {\tiny \textcolor{red}{$b_{\alpha_1}c_{\alpha_2}c_{\alpha_3}a_{\alpha_4}a_{\alpha_5}b_{\alpha_6}$}};
    \node (P17) at (8.5,0) {\tiny \textcolor{red}{$b_{\alpha_1}c_{\alpha_2}c_{\alpha_3}a_{\alpha_4}a_{\alpha_5}b_{\alpha_6}$}};
    \node (P18) at (1.5,0) {\tiny \textcolor{red}{$b_{\alpha_1}c_{\alpha_2}c_{\alpha_3}a_{\alpha_4}a_{\alpha_5}b_{\alpha_6}$}};
    \node (P19) at (2.5,-2) {\tiny \textcolor{red}{$b_{\alpha_1}c_{\alpha_2}c_{\alpha_3}a_{\alpha_4}a_{\alpha_5}b_{\alpha_6}$}};
    \node (P20) at (7.5,-2) {\tiny \textcolor{red}{$b_{\alpha_1}c_{\alpha_2}c_{\alpha_3}a_{\alpha_4}a_{\alpha_5}b_{\alpha_6}$}};

    \begin{scope}[color=blue,decoration={
    markings,
    mark=at position 0.55 with {\arrow{>}}}
    ]
        \draw[postaction={decorate}] (P4) -- (P5) node[midway, above right=-1ex and 0ex] {\scalebox{0.5}{$\begin{pmatrix} 1 \\ 4 \end{pmatrix} \to \begin{pmatrix} 1 \\ 1 \end{pmatrix}$}};
        \draw[postaction={decorate}] (P5) -- (P17) node[midway, below right=-1ex and 0ex] {$\substack{\scalebox{0.75}{$b_{\alpha_2} \to c_{\alpha_2}$} \\ \scalebox{0.75}{$a_{\alpha_3} \to c_{\alpha_3}$} \\ \scalebox{0.5}{$\begin{pmatrix} 1 & 2 & 3 \\ 4 & 5 & 6 \end{pmatrix}$}}$};
        \draw[postaction={decorate}] (P15) -- (P16) node[midway, below] {\scalebox{0.5}{$\begin{pmatrix} 3 \\ 3 \end{pmatrix} \to \begin{pmatrix} 3 \\ 6 \end{pmatrix}$}};
        \draw[postaction={decorate}] (P4) -- (P16) node[midway, below right=-1ex and 0ex] {$\substack{\scalebox{0.75}{$b_{\alpha_2} \to c_{\alpha_2}$} \\ \scalebox{0.75}{$a_{\alpha_3} \to c_{\alpha_3}$} \\ \scalebox{0.5}{$\begin{pmatrix} 1 & 2 & 3 \\ 1 & 5 & 6 \end{pmatrix}$}}$};
        \draw[postaction={decorate}] (P16) -- (P17) node[midway, above right=-1ex and 0.25ex] {\scalebox{0.5}{$\begin{pmatrix} 1 \\ 4 \end{pmatrix} \to \begin{pmatrix} 1 \\ 1 \end{pmatrix}$}};
        \draw[postaction={decorate}] (P17) -- (P20) node[midway, below right=-1ex and 0.25ex] {\scalebox{0.5}{$\begin{pmatrix} 2 \\ 5 \end{pmatrix} \to \begin{pmatrix} 2 \\ 2 \end{pmatrix}$}};
        \draw[postaction={decorate}] (P20) -- (P19) node[midway, above] {\scalebox{0.5}{$\begin{pmatrix} 3 \\ 6 \end{pmatrix} \to \begin{pmatrix} 3 \\ 3 \end{pmatrix}$}};
        \draw[postaction={decorate}] (P19) -- (P18) node[midway, below left=-1ex and 0.25ex] {\scalebox{0.5}{$\begin{pmatrix} 1 \\ 1 \end{pmatrix} \to \begin{pmatrix} 1 \\ 4 \end{pmatrix}$}};
        \draw[postaction={decorate}] (P18) -- (P15) node[midway, above left=-1ex and 0.25ex] {\scalebox{0.5}{$\begin{pmatrix} 2 \\ 2 \end{pmatrix} \to \begin{pmatrix} 2 \\ 5 \end{pmatrix}$}};
    \end{scope}
    \end{tikzpicture}
\caption{Vertex labels, coordinate changes, and linear algebra framework for the $A^1G^{23}$ square and $\mathcal{L}$ hexagon of $\Delta(a,b,c)$.}
\label{fig:vertex-labels-a1g23-L}
\end{figure}

({\em $A^1G^{23}$ square $2$-cell.}) The square $2$-cell indexed by the subgroup $A^1G^{23}$ has the following four vertices: $P_4$ and $P_{16}$, described above, and
    \begin{align*}
    P_5 &= a_{\alpha_{345}}b_{\alpha_{126}}\left[-\smashoperator{\sum_{i \in \{3,4,5\}}}\Phi\left(b_{\alpha_i}^{-1}a_{\alpha_i}\right)\right]{\scalebox{0.75}{$\begin{pmatrix} 1 & 2 & 3 \\ 1 & 5 & 6 \end{pmatrix}$}} \, \text{ and } \\
    P_{17} &= a_{\alpha_{45}}b_{\alpha_{16}}c_{\alpha_{23}}\left[-\smashoperator{\sum_{i \in \{2,3\}}}\Phi\left(b_{\alpha_i}^{-1}c_{\alpha_i}\right)-\smashoperator{\sum_{i \in \{4,5\}}} \Phi\left(b_{\alpha_i}^{-1}a_{\alpha_i}\right)\right]{\scalebox{0.75}{$\begin{pmatrix} 1 & 2 & 3 \\ 1 & 5 & 6 \end{pmatrix}$}} \, . 
    \end{align*}
The fact that the two paths $P_{4} \to P_{5} \to P_{17}$ and $P_{4} \to P_{16} \to P_{17}$ are coterminal corresponds to the relation 
    \begin{align*}
    &\left[ \, \, \, \, \smashoperator{\sum_{i \in \{3,4,5\}}}\Phi\left(b_{\alpha_i}^{-1}a_{\alpha_i}\right)\right]{\scalebox{0.75}{$\begin{pmatrix} 1 \\ 4 \end{pmatrix}$}} - \left[ \, \, \, \, \smashoperator{\sum_{i \in \{3,4,5\}}}\Phi\left(b_{\alpha_i}^{-1}a_{\alpha_i}\right)\right]{\scalebox{0.75}{$\begin{pmatrix} 1 \\ 1 \end{pmatrix}$}} \\
    &\cdot (a_{\alpha_{3}}^{-1}c_{\alpha_{3}}) \left[-\Phi\left(a_{\alpha_3}^{-1}c_{\alpha_3}\right)\right]{\scalebox{0.75}{$\begin{pmatrix} 1 & 2 & 3 \\ 1 & 5 & 6 \end{pmatrix}$}} \cdot (b_{\alpha_{2}}^{-1}c_{\alpha_{2}}) \left[-\Phi\left(b_{\alpha_2}^{-1}c_{\alpha_2}\right)\right]{\scalebox{0.75}{$\begin{pmatrix} 1 & 2 & 3 \\ 1 & 5 & 6 \end{pmatrix}$}} \\
    &=(b_{\alpha_{2}}^{-1}c_{\alpha_{2}}) \left[-\Phi\left(b_{\alpha_2}^{-1}c_{\alpha_2}\right)\right]{\scalebox{0.75}{$\begin{pmatrix} 1 & 2 & 3 \\ 4 & 5 & 6 \end{pmatrix}$}} \cdot (a_{\alpha_{3}}^{-1}c_{\alpha_{3}}) \left[-\Phi\left(a_{\alpha_3}^{-1}c_{\alpha_3}\right)\right]{\scalebox{0.75}{$\begin{pmatrix} 1 & 2 & 3 \\ 4 & 5 & 6 \end{pmatrix}$}} \\
    &\cdot \left( \left[ \, \, \, \smashoperator{\sum_{i \in \{2,3\}}}\Phi\left(b_{\alpha_i}^{-1}c_{\alpha_i}\right) + \smashoperator{\sum_{i \in \{4,5\}}} \Phi\left(b_{\alpha_i}^{-1}a_{\alpha_i}\right)\right]{\scalebox{0.75}{$\begin{pmatrix} 1 \\ 4 \end{pmatrix}$}} - \left[ \, \, \, \smashoperator{\sum_{i \in \{2,3\}}}\Phi\left(b_{\alpha_i}^{-1}c_{\alpha_i}\right) + \smashoperator{\sum_{i \in \{4,5\}}} \Phi\left(b_{\alpha_i}^{-1}a_{\alpha_i}\right)\right]{\scalebox{0.75}{$\begin{pmatrix} 1 \\ 1 \end{pmatrix}$}} \right) \, ,
    \end{align*}
which holds in the group $A^1G^{23} \cong G_{\alpha_2} \times G_{\alpha_3} \times A_{\alpha_1}$.

({\em Three segment bound in $A^1G^{23}$ square $2$-cell.}) In the prior computations, we have proved the upper bounds for the segments $P_{4,5}$ and $P_{4,16}$. The $P_{5,17}$ segment converts the $\alpha_2$ and $\alpha_3$ coordinates from $b$ to $c$ and from $a$ to $c$ respectively using generators from $G^{23}$. Thus, $\ell_{5,17} \leq d_{\alpha_2}(b,c) + d_{\alpha_3}(a,c)$. The upper bound for the $P_{16,17}$ segment is established after the discussion of $\mathcal{L}$ hexagonal $2$-cell. 

({\em $\mathcal{L}$ hexagonal $2$-cell.}) Finally, we discuss the relation corresponding to the central hexagonal $2$-cell. The six vertices of the central hexagonal region area all equal to some $\Ker(\Phi)$ version of the element $b_{\alpha_1}c_{\alpha_2}c_{\alpha_3}a_{\alpha_4}a_{\alpha_5}b_{\alpha_6}$, with six different choices of standard linear algebra from \scalebox{0.75}{$\begin{pmatrix} 1 & 2 & 3 \\ 1/4 & 2/5 & 3/6 \end{pmatrix}$}.

Starting at the top left corner, one has
    \[
    P_{15} \;=\; a_{\alpha_{45}}b_{\alpha_{16}}c_{\alpha_{23}}
    \left[-\Phi(a_{\alpha_1}^{-1}b_{\alpha_1}) -\Phi(a_{\alpha_2}^{-1}c_{\alpha_2}) -\Phi(a_{\alpha_3}^{-1}c_{\alpha_3}) 
    -\Phi(a_{\alpha_6}^{-1}b_{\alpha_6})\right] \scalebox{0.75}{$\begin{pmatrix} 1 & 2 & 3 \\ 4 & 5 & 3 \end{pmatrix}$} \, ,
    \]
with linear algebra encoding \scalebox{0.75}{$\begin{pmatrix} 1 & 2 & 3 \\ 4 & 5 & 3 \end{pmatrix}$}. For efficiency, we write this as 
    \[
    P_{15} \;=\; a_{\alpha_{45}}b_{\alpha_{16}}c_{\alpha_{23}}[L]\scalebox{0.75}{$\begin{pmatrix}1 & 2 & 3\\4 & 5& 3\end{pmatrix}$} \, .
    \]
Now, starting at the vertex $P_{15}$, we work clockwise around the central hexagon, changing the linear algebra encoding as follows: 
    \begin{enumerate}
    \item Change the $A_{\alpha_3}$ linear algebra from coordinates $\alpha_3$ to $\alpha_6$. This is achieved by right-multiplying by the element 
    $[L]\scalebox{0.75}{$\begin{pmatrix} 3 \\ 6 \end{pmatrix}$} \cdot \left( [L]\scalebox{0.75}{$\begin{pmatrix} 3 \\ 3 \end{pmatrix}$} \right)^{-1}$, resulting in the vertex $P_{16}$.
    
    \item Change the $A_{\alpha_1}$ linear algebra from coordinates $\alpha_4$ to $\alpha_1$; namely, multiply by $[L]\scalebox{0.75}{$\begin{pmatrix} 1 \\ 1 \end{pmatrix}$} \cdot \left( [L]\scalebox{0.75}{$\begin{pmatrix} 1 \\ 4 \end{pmatrix}$} \right)^{-1}$. The result is the vertex $P_{17}$. 
    
    \item Change the $A_{\alpha_2}$ linear algebra from coordinates $\alpha_5$ to $\alpha_2$ by multiplying by $[L]\scalebox{0.75}{$\begin{pmatrix} 2 \\ 2 \end{pmatrix}$} \cdot \left( [L]\scalebox{0.75}{$\begin{pmatrix} 2 \\ 5 \end{pmatrix}$} \right)^{-1}$ to give the vertex $P_{18}$.
    
    \item Change the $A_{\alpha_3}$ linear algebra from coordinates $\alpha_6$ back to $\alpha_3$; this requires multiplying by the inverse of the element of item (1); that is, we multiply by $[L]\scalebox{0.75}{$\begin{pmatrix} 3 \\ 3 \end{pmatrix}$} \cdot \left( [L]\scalebox{0.75}{$\begin{pmatrix} 3 \\ 6 \end{pmatrix}$} \right)^{-1}$. The resulting vertex is $P_{19}$.
    
    \item Change the $A_{\alpha_1}$ linear algebra from coordinates $\alpha_1$ back to $\alpha_4$; namely, multiply by $[L]\scalebox{0.75}{$\begin{pmatrix} 1 \\ 4 \end{pmatrix}$} \cdot \left( [L]\scalebox{0.75}{$\begin{pmatrix} 1 \\ 1 \end{pmatrix}$} \right)^{-1}$ to get the vertex $P_{20}$.
    
    \item Finally, change the $A_{\alpha_2}$ linear algebra from coordinates $\alpha_2$ back to $\alpha_5$ via multiplication by $[L]\scalebox{0.75}{$\begin{pmatrix} 2 \\ 5 \end{pmatrix}$} \cdot \left( [L]\scalebox{0.75}{$\begin{pmatrix} 2 \\ 2 \end{pmatrix}$} \right)^{-1}$. The result is the vertex $P_{15}$.
    \end{enumerate}
The sum of these six terms is $0$ in the abelian group $\mathcal{L}$. This is the relation corresponding to the boundary of the central hexagon.  

We need to discuss the role of the 3-fold rotational symmetry in the description of the 6 central vertices above. The top two vertices are obtained from the vertex $a$ by either traveling along the right edge or along the left edge of the algebraic triangle before turning towards the central hexagon. This results in the vertices: 
$$
P_{15} \; =\; a_{\alpha_{45}}c_{\alpha_{23}} b_{\alpha_{61}}\left[ -\smashoperator{\sum_{i \in \{1,6\}}}\Phi(a_{\alpha_{i}}^{-1}b_{\alpha_{i}}) -\smashoperator{\sum_{i \in \{2,3\}}}\Phi(a_{\alpha_{i}}^{-1}c_{\alpha_{i}}) \right] \begin{pmatrix}
    1&2&3\\4&5&3
\end{pmatrix}, 
$$
and (by step (1) above) 
$$
P_{16} \; =\; a_{\alpha_{45}}b_{\alpha_{61}}c_{\alpha_{23}} \left[ - \smashoperator{\sum_{i \in \{1,6\}}}\Phi(a_{\alpha_{i}}^{-1}b_{\alpha_{i}}) - \smashoperator{\sum_{i \in \{2,3\}}}\Phi(a_{\alpha_{i}}^{-1}c_{\alpha_{i}}) \right] \begin{pmatrix}
    1&2&3\\4&5&6
\end{pmatrix}. 
$$
Changing the standard linear algebra $\begin{pmatrix}
    1\\1/4
\end{pmatrix}$ gives (via step (2) above) the vertex 
$$
P_{17} \; =\; a_{\alpha_{45}}b_{\alpha_{61}}c_{\alpha_{23}}\left[ - \smashoperator{\sum_{i \in \{1,6\}}}\Phi(a_{\alpha_{i}}^{-1}b_{\alpha_{i}}) - \smashoperator{\sum_{i \in \{2,3\}}}\Phi(a_{\alpha_{i}}^{-1}c_{\alpha_{i}}) \right] \begin{pmatrix}
    1&2&3\\1&5&6
\end{pmatrix}. 
$$
However, using the 3-fold symmetry, we can also get to $P_{17}$ by starting from vertex $b$, moving up along the right side ($ba$-side) of the triangle and then in towards $P_{17}$. This gives the element (which we denote with prime notation) 
$$
P'_{17} \; =\; b_{\alpha_{61}}a_{\alpha_{45}}c_{\alpha_{23}} \left[ - \smashoperator{\sum_{i \in \{4,5\}}}\Phi(b_{\alpha_{i}}^{-1}a_{\alpha_{i}}) - \smashoperator{\sum_{i \in \{2,3\}}}\Phi(b_{\alpha_{i}}^{-1}c_{\alpha_{i}}) \right] \begin{pmatrix}
    1&2&3\\1&5&6
\end{pmatrix} . 
$$
To see that $P'_{17}= P_{17}$, we argue as we did to establish $P'_4 = P_4$ in the proof of Lemma~\ref{lem:triside-labdist}. First, the padded elements $a_{\alpha_{45}}$, $b_{\alpha_{61}}$, and $c_{\alpha_{23}}$ all commute since their non-trivial components have non-overlapping index sets. So we only need to establish equality of the linear algebra portions. 
The linear algebra term from $P'_{17}$ involves
$$
- \smashoperator{\sum_{i \in \{4,5\}}} \Phi(b_{\alpha_{i}}^{-1}a_{\alpha_{i}}) - \smashoperator{\sum_{i \in \{2,3\}}}\Phi(b_{\alpha_{i}}^{-1}c_{\alpha_{i}}) \; =\; -\smashoperator{\sum_{i \in \{4,5\}}}\Phi(b_{\alpha_{i}}^{-1}a_{\alpha_{i}}) - \smashoperator{\sum_{i \in \{2,3\}}}\Phi(b_{\alpha_{i}}^{-1}a_{\alpha_{i}}) - 
\smashoperator{\sum_{i \in \{2,3\}}}\Phi(a_{\alpha_{i}}^{-1}c_{\alpha_{i}}). 
$$
The last term $- \smashoperator{\sum_{i \in \{2,3\}}}\Phi(a_{\alpha_{i}}^{-1}c_{\alpha_{i}})$ appears in the linear algebra portion of $P_{17}$. So we need to show that 
$$
- \smashoperator{\sum_{i \in \{4,5\}}}\Phi(b_{\alpha_{i}}^{-1}a_{\alpha_{i}}) - \smashoperator{\sum_{i \in \{2,3\}}}\Phi(b_{\alpha_{i}}^{-1}a_{\alpha_{i}}) \; =\; -\smashoperator{\sum_{i \in \{1,6\}}}\Phi(a_{\alpha_{i}}^{-1}b_{\alpha_{i}}). 
$$
As in the proof of Lemma~\ref{lem:triside-labdist}, the right side is equal to $\smashoperator{\sum_{i \in \{1,6\}}}\Phi(b_{\alpha_{i}}^{-1}a_{\alpha_{i}})$, and subtracting this term from both sides of the equation above yields the expression 
$-\Phi(b^{-1}a) = 0$, which holds since $a$, $b$,  and so $b^{-1}a$,  all lie in $\Ker(\Phi)$. 

In a similar manner, it can be shown that one obtains the same elements on the central hexagon by either starting from $P_{15}$ and using the 6 steps above, or  by traveling from $b$ along the bottom side and then up to the central hexagon, or starting from $c$ and traveling along either the bottom side or the left side and then in to the central hexagon. 

({\em Segment bounds for $P_{15,16}$ and $P_{16,17}$.}) The purpose of the $P_{15,16}$ segment is to change the standard linear algebra \scalebox{0.75}{$\begin{pmatrix} 3 \\  3 \end{pmatrix}$} of the element $P_{15}$ into standard linear algebra \scalebox{0.75}{$\begin{pmatrix} 3 \\  6 \end{pmatrix}$} of $P_{16}$ using generators from the group $A^3(3,6)$. One needs at most $d_{\alpha_{23}}(a,c) + d_{\alpha_{16}} (a,b)$ for this, thus providing an upper bound for $\ell_{15,16}$. Finally, the $P_{16,17}$ segment is responsible for changing the standard linear algebra \scalebox{0.75}{$\begin{pmatrix} 1 \\  4 \end{pmatrix}$} of the element $P_{16}$ into standard linear algebra \scalebox{0.75}{$\begin{pmatrix} 1 \\  1 \end{pmatrix}$} of $P_{17}$ using generators from the group $A^1(1,4)$. One needs at most $d_{\alpha_{23}}(a,c) + d_{\alpha_{16}} (a,b)$ for this, thus providing an upper bound for $\ell_{16,17}$.

Recorded below in Remark~\ref{rem:record-vertlength} is the exhaustive list of upper bounds for all edges in $\Delta(a,b,c)$. By adding these upper bounds for each side of a given $2$-cell in $\Delta$, one finds that the desired upper bound for the boundary length of the related region in $\Delta(a,b,c)$ is no more than $4D$, where $D \coloneqq d_G(a,b) + d_G(b,c) + d_G(c,a)$ and $d_G$ is the $\ell^1$-metric on $\prod_{i=1}^{2n} G_i$.
\end{proof}

\begin{rem}[Record of  segment length upper bounds for Lemma~\ref{lem:triinner-labdist}]\label{rem:record-vertlength}
Similar to Remark~\ref{record1}, we record the bounds on the various  segment lengths.

    \begin{center}
    \begin{longtblr}{
        colspec = {|c|c||c|c|},
        rowhead = 1, rows={abovesep=2pt,belowsep=2pt},
    }
    \hline
    {Segment Length} & {Upper Bound} & {Segment Length} & {Upper Bound} \\
    \hline
    \hline
    $\ell_{1,2}$ & $d_{\alpha_6} (a,b)$ & $\ell_{10,12}$ & $d_{\alpha_3} (a,c) + d_{\alpha_6} (a,b)$ \\
    \hline
    $\ell_{1,9}$  & $d_{\alpha_{3}} (a,c)$ & $\ell_{10,13}$ & $d_{\alpha_3} (a,c) + d_{\alpha_6} (a,b)$ \\
    \hline
    $\ell_{2,3}$ & $d_{\alpha_6} (a,b)$ & $\ell_{11,12}$ & $d_{\alpha_{12}} (a,c)$ \\
    \hline
    $\ell_{2,10}$  & $d_{\alpha_{3}} (a,c)$ & $\ell_{11,14}$ & $d_{\alpha_{12}} (a,c)$ \\
    \hline
    $\ell_{3,4}$ & $d_{\alpha_{12}} (a,b)$ & $\ell_{13,16}$ & $d_{\alpha_1} (a,b) + d_{\alpha_2} (a,c)$ \\
    \hline
    $\ell_{3,13}$ & $d_{\alpha_3} (a,c)$ & $\ell_{12,15}$ & $d_{\alpha_1} (a,b) + d_{\alpha_2} (a,c)$ \\
    \hline
    $\ell_{4,5}$ & $d_{\alpha_{126}} (a,b)$ & $\ell_{12,13}$ & $d_{\alpha_3} (a,c) + d_{\alpha_6} (a,b)$ \\
    \hline
    $\ell_{4,16}$ & $d_{\alpha_3} (a,c) + d_{\alpha_2} (b,c)$ & $\ell_{14,15}$ & $d_{\alpha_1} (b,c) + d_{\alpha_6} (a,b)$ \\
    \hline
    $\ell_{5,17}$ & $d_{\alpha_2} (b,c) + d_{\alpha_3} (a,c)$ & $\ell_{15,16}$ & $d_{\alpha_{23}} (a,c) + d_{\alpha_{16}} (a,b)$ \\
    \hline
    $\ell_{9,10}$ & $d_{\alpha_6} (a,b)$ & $\ell_{16,17}$ & $d_{\alpha_{23}} (a,c) + d_{\alpha_{16}} (a,b)$ \\
    \hline
    $\ell_{9,11}$ & $d_{\alpha_3} (a,c)$ & \mbox{} & \mbox{} \\
    \hline
    \end{longtblr}
    \end{center}

\end{rem}

We conclude this section with the following lemma that provides an upper bound for the area of $\Delta(a,b,c)$ which will be used in the proof of Theorem~\ref{thm:main}.

\begin{lem}[Area of actualization of the algebraic triangle]\label{lem:area}
Given an ordered triple $(a,b,c)$ of elements of $ \Ker(\Phi)$, the area of actualization of the algebraic triangle $\Delta(a,b,c)$  is bounded above by $25\delta_G(24D)$.
\end{lem}

\begin{proof}
In Lemma~\ref{lem:triinner-labdist} we concluded that the lengths of all words labeling the edges of the triangle are bounded above by $4D$ and hence each boundary loop for a bounded region has perimeter at most $24D$. We thus obtain a filling whose area is bounded by $\delta_G(24D)$ for each of these twenty-five bounded regions. Thus, we conclude that the area of the triangle is bounded by $25\delta_G(24D)$. 
Then $\mathrm{Area}(\Delta(a,b,c)) \leq 25\delta_G(24D) \,$.

\end{proof}

\section{Calculation of the Dehn function}\label{sec:dehnfunction}

In this section, we  show that $\Ker(\Phi)$ is finitely presented and establish bounds for the Dehn function of $\Ker(\Phi)$ by utilizing the machinery that we have developed over the course of the article and taking inspiration from the logical structure of the proof of \cite[Theorem 3.2]{KLI}. 

However, we need one final technical lemma to ensure that we can fill the entirety of our initial loop; as seen in Figure~\ref{fig:farey-tess}, the subdivision of our initial loop into triangles potentially leaves a frill of bigons which must be accounted for; by showing that these bigons all have a uniformly bounded perimeter, we can uniformly bound their area and simultaneously use this characterization to create a finite presentation for $\Ker(\Phi)$. This lemma is parallel to Remark 3.6 of \cite{KLI}. Adopting the terminology from \cite{MR3651586} and \cite{KLI}, we call the path along a side of our algebraic triangle a \emph{spanning path}, the explicit construction of which is given in Lemma~\ref{lem:triside-labdist}.

\begin{lem}\label{lem:bound6}
If two vertices $a, b \in \Ker(\Phi)$ have $d(a,b) = 1$, then the spanning path between them has length at most $6$.
\end{lem}

\begin{proof}
Consider the two vertices $a$ and $b$. From the segment bound computations in Lemma~\ref{lem:triside-labdist}, we get that the spanning path between them has length bounded by
    \[
    3d_{\alpha_{12}}(a,b) + d_{\alpha_{34}}(a,b) + 2d_{\alpha_5}(a,b) + 3d_{\alpha_6}(a,b) \, .
    \]
Now, since $a$ and $b$ differ by a single generator of $\Ker(\Phi)$, at most two terms in the distance bound are non-zero. Hence, the bound is $6$.
\end{proof}


This previous statement together with Lemma~\ref{lem:triside-labdist} allows us to explicitly state a finite presentation for $\Ker(\Phi)$.


\begin{defn}[Finite presentation of $\Ker(\Phi)$]\label{def:kerfinpres}
Suppose that for $n > 2$, the $2n$ groups $G_i \cong \langle X_i \,|\, R_i \rangle$ and homomorphisms $\varphi_i : G_i \to \Z^m$ satisfy the conditions of the basic setup of Definition~\ref{def:setup}, where $X_i$ is a set of kernel-section generators for $G_i$. The proof of Theorem~\ref{thm:main} establishes and uses the following explicit finite presentation for $\Ker(\Phi)$,
    \[
    \Ker(\Phi)\; \cong \; \langle \, \overline{X} \;|\; \overline{R} \; \rangle \, ,
    \]
    where 
\begin{itemize}
    \item (Generators.) $\overline{X}$ is the union of sets of all vector generators for $\mathcal{L}$ (Definition~\ref{def:standard1}), of the other twelve standard subgroups of $\Ker(\Phi)$ (Definition~\ref{defn:stdsub}), and of the twelve non-standard subgroups of $\Ker(\Phi)$ (Definition~\ref{defn:nonstdsub}). 
    \item (Relations.) $\overline{R}$ is composed of two types of relations.  
    \begin{itemize}
        \item (Commutation relations and vector versions of the $R_i$ relations.) Vector versions of the $R_i$ relations, and vector versions of the commutation relations between elements of various $X_i$ and $X_j$ generating sets. 
        
        For example, the subgroup $G^{36}_N$ is a non-standard subgroup isomorphic to $G_{\alpha_3} \times G_{\alpha_6}$, itself isomorphic to a product 
            \[
            \prod_{j \in \alpha_{36}} G_j \, ,
            \]
        and so the relations include specific non-standard lifts of the finite presentations of $G_j$ for $j \in \alpha_{36}$ in addition to all the commutation relations between the lifts of the generating sets $X_j$ of these $G_j$. Similar statements hold for the other twenty-four face subgroups. 

        Note that a given subgroup $G_j$ might have several different lifts to $\Ker(\Phi)$. For example, the groups $G_j$ for $j \in \alpha_6$ lift in a non-standard way to $G^{36}_N$ as above, but also in a standard way as a subgroup of $G^{126} \cong \prod_{j \in \alpha_{126}} G_j$. 
        \item (Bigon relations.) These consist of all words in $F(\overline{X})$ of length at most 7 which represent the identity in $\Ker(\Phi)$. These account for the bigons at the boundary of the Farey tesselation of the loop in the Cayley graph of $\Ker(\Phi)$ in the proof of Theorem~\ref{thm:main} below. The key point is that this is a finite set of relations.
    \end{itemize}
  
\end{itemize}
\end{defn}

Finally, we arrive at our main theorem. 

\begin{thm}\label{thm:main}

Let $n > 2$ be an integer. For each $1 \leq i \leq 2n$, let $G_i$ be a finitely presented group such that
    \[\begin{tikzcd}
    1 \arrow[r] & N_i \arrow[r] & G \arrow[r,"\varphi_i"] & \Z^m \arrow[r]  & 1
    \end{tikzcd}\]
is an $n$-split short exact sequence. Let $G = \prod_{i=1}^{2n} G_i$, and let $\Phi : G \to \Z^m$ be defined by $\Phi(g) = \sum_{i=1}^{2n} \varphi_i(\grpproj_i(g))$.

Then the kernel $\Ker(\Phi)$ in the short exact sequence
    \[\begin{tikzcd}
    1 \arrow[r] & \Ker(\Phi) \arrow[r] & G \arrow[r,"\Phi"] & \Z^m \arrow[r]  & 1
    \end{tikzcd}\]
is finitely presented and its Dehn function satisfies $\delta_G(n) \dehnleq \delta_{\Ker(\Phi)}(n) \dehnleq \overline{\delta_G}(n) \cdot \log(n)$. If, moreover, $\frac{\delta_G(n)}{n}$ is superadditive, then $\delta_G(n) \simeq \delta_{\Ker(\Phi)}(n)$.

\end{thm}

\begin{proof}

Let $\hat{\gamma} \in \Ker(\Phi)$. Then $\hat{\gamma}$ represents a loop in $\mathrm{Cay}(\Ker(\Phi),\bar{X})$ of length $n \geq 3$. We choose $k \in \mathbb{Z}$ such that $3 \cdot 2^{k-1} \leq n \leq 3 \cdot 2^k$. Now define a loop $\gamma : [0,3 \cdot 2^k] \rightarrow \mathrm{Cay}(\Ker(\Phi),\bar{X})$ by adding a constant path at the end to the loop $\hat{\gamma}$. Algebraically, this is just using the fact that Dehn functions are monotone increasing, so if $\delta_{\Ker(\Phi)}(|\gamma|) \dehnleq \delta_G(n) \cdot \log(n)$ (or $\dehnleq \delta_G(n)$), then so is $\delta_{\Ker(\Phi)}(|\hat{\gamma}|)$. Geometrically, this is simply noticing that (except for the central triangular region) any of the triangles or bigons of Figure~\ref{fig:farey-tess} may be degenerate: thus the argument below may over-count the number of each type of region, but does not under-count them, and so the area estimate for $\gamma$ below is, at worst, an over-estimate of the area of $\hat{\gamma}$. Hence, we will show that the loop $\gamma$ has area $\dehnleq f(n)$ for the appropriate function $f$, which will complete the proof.


We now explain the process for subdividing $\hat{\gamma}$ into triangles and bigons, which, together with previous work, will be used subsequently to calculate an upper bound on $\Area(\hat{\gamma})$.

Denote the disk in Figure~\ref{fig:farey-tess} by $D$. By construction, $|\partial{D}| = 3 \cdot 2^k$, so we may subdivide $\partial{D}$ exactly into thirds and subdivide each resulting segment precisely in half, up to $k$ times. We label the first three points of subdivision $a_0$, $b_0$, and $c_0$, and join these points pairwise by paths within $D$ to create a loop, bounding a triangular region, which we call $\Delta_0$.

We can then further subdivide each resulting segment of $\partial{D}$ into halves, and each point $v$ of this subdivision can be labeled as $a_1$, $b_2$, or $c_3$ in such a way that the points on $\partial{D}$ adjacent to $v$ do not share letter labels with $v$; that is, if $v$ is adjacent to $c$ and $a$, then we label $v$ as $b_2$.

We may then repeat this process for subdividing $\partial{D}$ inductively as follows.

If $\partial{D}$ is a $1$-complex with vertices $\partial{D}^{(0)}$, then we define the {\em depth} of a vertex $v$, written $d(v)$, to be the number of subdivisions of $\partial{D}$ that occurred before $v$ appeared as a vertex. Thus, for example, $a_0$, $b_0$, and $c_0$ appear in the first subdivision of $\partial{D}$, and thus have no subdivisions beforehand; thus, if $v \in \{a_0, b_0, c_0\}$, then $d(v) = 0$. Similarly, if $v \in \{a_1, b_2, c_3\}$, then $d(v) = 1$.

We continue to subdivide $\partial{D}$ and label the vertices at which these subdivisions occur as follows. Each point $v$ of subdivision of $\partial{D}$ is given a label of the form $x_\beta$, where $x \in \{a, b, c\}$ and $\beta$ is a finite sequence taking values in $\{1, 2, 3\}$. We assign such a label to $v$ as follows. By construction, $v$ is adjacent to two vertices $v'$ with label $x'_{\beta'}$ and $v''$ with label $x''_{\beta''}$, where $x', x'' \in \{a, b, c\}$ and $x' \neq x''$. Set $x = \{a, b, c\} \cap \{x', x''\}^c$. Then $\beta$ is determined as such: one can easily see that $d(v') \neq d(v'')$ due to the Farey nature of the diagram and the fact that $d(v) > 1$; furthermore, exactly one of the equations $d(v) = d(v') + 1$ or $d(v) = d(v'') + 1$ is true; assume without loss of generality that the former equation is true. Then define
    \[
    t = \begin{cases}
    1 & \text{if } x = a \\
    2 & \text{if } x = b \\
    3 & \text{if } x = c
    \end{cases} \, ,
    \]
and set $\beta$ to be $\beta'$ with $\{t\}$ subsequently concatenated. Then we can give $v$ the label $x_\beta$.

Note that with this labeling scheme, if the vertex $v$ is given the labeling $x_\beta$, then $d(v) = |\beta|$. Additionally, this labeling scheme tells us the sequence of reflections one must make from the original triangle $\Delta_0$ to arrive at the vertex $v$; for example, the vertex $v = a_{3,1}$ can be attained by beginning with $\Delta_0$, reflecting the $c$ vertex across its opposite edge, and then reflecting the $a$ vertex similarly, corresponding to the sequence $\{3, 1\}$; then the vertex $a_{3,1}$ is the vertex of the resulting triangle corresponding to the vertex $a_0$ of $\Delta_0$.

A sample labeling of $\partial{D}$ via this process including vertices up to depth $2$ can be seen in Figure~\ref{fig:farey-tess} below.

We now connect these vertices to create triangular regions $\Delta_\beta$ in $D$ by reflecting the central triangle $\Delta_0$ across the edges opposite its vertices as specified by the sequence $\beta$; for example, given $a_1$, $b_2$, or $c_3$, respectively, we would reflect $\Delta_0$ across the edge opposite the $a_0$, $b_0$, and $c_0$ vertex, respectively, and the resulting triangular regions $\Delta_1$, $\Delta_2$, and $\Delta_3$ in $D$ would have vertices at $(a_1, b_0, c_0)$, $(a_0, b_2, c_0)$, and $(a_0, b_0, c_3)$, respectively. Figure~\ref{fig:farey-tess} shows an example of the triangles including vertices of depth at most $2$ using this labeling scheme.

    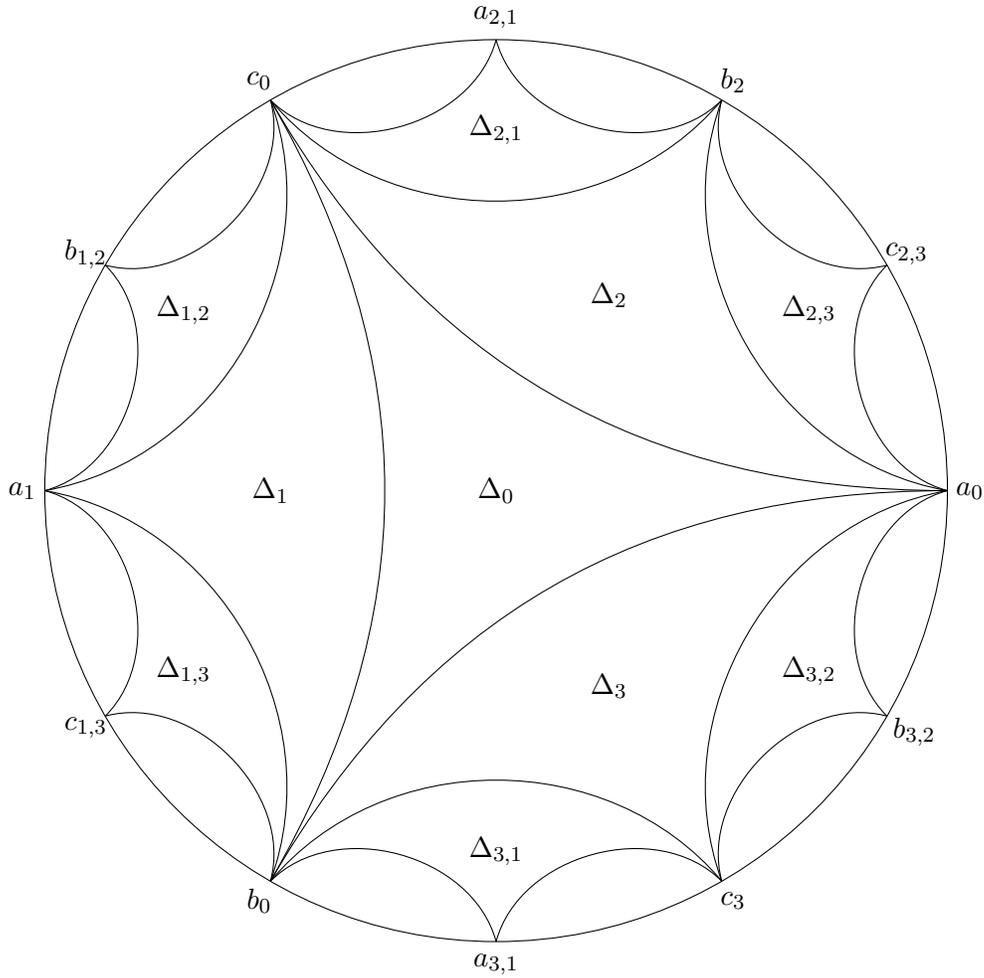
\begin{figure}[h]
    \centering
        \begin{tikzpicture}[scale = 6]
            \foreach \i in {0,1,2,3}{
                \foreach \j in {0,1,2}{
                    \foreach \k in {0,1,2}{
                        \foreach \l in {0,1,2}{
                        \node (pt-\i-\j-\k-\l) at ($({\i*120 + \j*60 + \k*30 + \l*15}:1)$){};
                        }
                    }
                }
            }
        \draw (0,0) circle (1);
        \foreach \i [evaluate=\i as \w using int(\i+1)] in {0,1,2}{
        \path ($(pt-\i-0-0-0)$) edge[bend left=30] ($(pt-\w-0-0-0)$);
            \foreach \j [evaluate=\j as \x using int(\j+1)] in {0,1}{
            \path ($(pt-\i-\j-0-0)$) edge[bend left=50] ($(pt-\i-\x-0-0)$);
                \foreach \k [evaluate=\k as \y using int(\k+1)] in {0,1}{
                \path ($(pt-\i-\j-\k-0)$) edge[bend left=60] ($(pt-\i-\j-\y-0)$);
                }
            }
        }
        \node (a) at ($(0:1.05)$){$a_0$};
        \node (b) at ($(-120:1.05)$){$b_0$};
        \node (c) at ($(-240:1.05)$){$c_0$};
        \node (D0) at ($(0:0)$){$\Delta_0$};

        \node (c3) at ($(-60:1.05)$){$c_3$};
        \node (a1) at ($(-180:1.05)$){$a_1$};
        \node (b2) at ($(-300:1.05)$){$b_2$};
        \node (D13) at ($(-60:0.5)$){$\Delta_3$};
        \node (D11) at ($(-180:0.5)$){$\Delta_1$};
        \node (D12) at ($(-300:0.5)$){$\Delta_2$};

        \node (b21) at ($(-30:1.07)$){$b_{3,2}$};
        \node (a21) at ($(-90:1.05)$){$a_{3,1}$};
        \node (c21) at ($(-150:1.05)$){$c_{1,3}$};
        \node (b22) at ($(-210:1.05)$){$b_{1,2}$};
        \node (a22) at ($(-270:1.05)$){$a_{2,1}$};
        \node (c22) at ($(-330:1.05)$){$c_{2,3}$};
        \node (D21) at ($(-30:0.8)$){$\Delta_{3,2}$};
        \node (D22) at ($(-90:0.8)$){$\Delta_{3,1}$};
        \node (D23) at ($(-150:0.8)$){$\Delta_{1,3}$};
        \node (D24) at ($(-210:0.8)$){$\Delta_{1,2}$};
        \node (D25) at ($(-270:0.8)$){$\Delta_{2,1}$};
        \node (D26) at ($(-330:0.8)$){$\Delta_{2,3}$};

        \end{tikzpicture}
    \caption{A $3$-fold Farey tesselation of a loop into triangles and bigons, along with a labeling of the vertices and triangular subregions.}
    \label{fig:farey-tess}
    \end{figure}

Similarly to the vertices, we may define the {\em depth} of a triangular region $\Delta_\beta$, again written $d(\Delta_\beta)$, to be the maximum among the depths of its vertices; this is equivalent to $d(\Delta_0) = 0$ and $d(\Delta_\beta) = |\beta|$ otherwise.

Now since $n = |\gamma| = 3 \cdot 2^k$, where $n$ and $k$ are integers, $k = \log_2(n/3)$; then since this process requires $k+1$ subdivisions before $\partial{D}$ is subdivided into segments of length $1$, this process only requires $\log_2(n/3) + 1$ subdivisions. 

Elementary counting arguments show that the number of triangles that $D$ is subdivided into is
    \[
    1 + \sum_{i=0}^{k+1} 3 \cdot 2^i = 3 \cdot 2^{k+2} - 2 \leq 3 \cdot 2^{k+2} = 4n \, .
    \]

Now, for each individual triangle with three given boundary points $x, y, z$ in this subdivision of $D$, we can tile the said triangle by the actualization of the algebraic triangle $\Delta(x,y,z)$ of Definition~\ref{def:actual-alg-tri}. At the same time, since the distance between any two adjacent vertices of $\gamma$ is $1$, Lemma~\ref{lem:bound6} tells us that each peripheral bigon in this subdivision has perimeter of at most $7$. Thus, there exists a constant $M > 0$ such that the area of each bigon in our subdivision of $D$ is at most $M$. Thus, from the above tiling of the Farey tesselation by algebraic triangles and Corollary~\ref{cor:kerfingen}, one concludes that $\Ker(\Phi)$ admits a finite presentation as defined in Definiton~\ref{def:kerfinpres}.

We are now ready to give the desired upper bound for the area enclosed by $\gamma$.

We first do the case when $\frac{\delta_G(n)}{n}$ is superadditive. Recall from previous work that the perimeter of a depth-$i$ triangle is at most $2^{k-i+1}$; additionally, there are (except for $\Delta_0$, of which there is only one) exactly $3 \cdot 2^{i-1}$ of these depth-$i$ triangles. Thus, since Remark~\ref{rem:record-vertlength} tells us that
    \[
    \mathrm{Area}(\Delta(a,b,c)) \leq 25\delta_G(24|\partial{\Delta}(a,b,c)|) \, ,
    \]
we conclude that
    \begin{align*}
    \mathrm{Area}(\gamma) &\leq \underbrace{25 \delta_G(24 \cdot 3 \cdot 2^k)}_{\text{area of } \Delta_0} + \sum_{i=1}^k \underbrace{3 \cdot 2^{i-1}}_{
    \substack{\text{number of} \\ \text{depth-$i$ triangles}}} \cdot \underbrace{25 \delta_G(24 \cdot 2^{k-i+2})}_{\substack{\text{area of a} \\ \text{depth-$i$ triangle}}} + \underbrace{3 \cdot 2^k}_{\substack{\text{number} \\ \text{of bigons}}} \cdot \underbrace{M}_{\substack{\text{area of} \\ \text{a bigon}}} \\
    &\leq 25 \cdot \delta_G(24 \cdot 3 \cdot 2^k) + 3 \cdot 25 \cdot \delta_G(24 \cdot 2^{k+2}) + M \cdot 3 \cdot 2^k \\
    &\leq 25 \cdot \delta_G(24 \cdot 4 \cdot 2^k) + 3 \cdot 25 \cdot \delta_G(24 \cdot 4 \cdot 2^k) + M \cdot 3 \cdot 2^k \\
    &= 100 \cdot \delta_G(8 \cdot 4 \cdot 3 \cdot 2^k) + M \cdot 3 \cdot 2^k \\
    &= 100 \cdot \delta_G(32n) + Mn \\
    &\dehnleq \delta_G(n) \, ,
    \end{align*}
as desired.

Now, suppose $\frac{\delta_G(n)}{n}$ is not superadditive. Once more using that $|\partial{D}| = 3 \cdot 2^k$, the same bounds on the perimeter of depth-$i$ triangles, and Remark~\ref{rem:record-vertlength}, we get:
    \begin{align*}
    \mathrm{Area}(\gamma) &\leq 25 \cdot \delta_G(24 \cdot 3 \cdot 2^k) + \sum_{i=1}^k 3 \cdot 2^{i-1} \cdot 25 \cdot \delta_G(24 \cdot 2^{k-i+2}) + M \cdot 3 \cdot 2^k \\
    &\leq 25 \cdot \overline{\delta_G}(24 \cdot 3 \cdot 2^k) + \sum_{i=1}^k 3 \cdot 2^{i-1} \cdot 25 \cdot \overline{\delta_G}(24 \cdot 2^{k-i+2}) + M \cdot 3 \cdot 2^k \\
    &\leq 25 \cdot \overline{\delta_G}(24 \cdot 3 \cdot 2^k) + \sum_{i=1}^k 3 \cdot 25 \cdot \overline{\delta_G}(24 \cdot 2^{k-1}) + M \cdot 3 \cdot 2^k \\
    &\leq 25 \cdot \overline{\delta_G}(24 \cdot 3 \cdot 2^k) + k \cdot 3 \cdot 25 \cdot \overline{\delta_G}(24 \cdot 3 \cdot 2^k) + M \cdot 3 \cdot 2^k \\
    &= (75k + 25) \overline{\delta_G}(24 \cdot 3 \cdot 2^k) + M \cdot 3 \cdot 2^k \\
    &= (75 \log_2(n/3) + 25) \overline{\delta_G}(24n) + Mn \\
    &\dehnleq \overline{\delta_G}(n) \log(n) \, .
    \end{align*}
This establishes the upper bounds.

Finally, for the lower bound, one first observes that there are retractions $\Ker(\Phi) \rightarrow G_i \times G_j$ for all $i \neq j$. So, by Lemma~\ref{lemret}, $\delta_{G_i \times G_j} (n) \preccurlyeq \delta_{\Ker(\Phi)} (n)$, for all $i \neq j$. Now, by Lemma~\ref{maxdehn}, the maximum of the Dehn functions of $G_i \times G_j$ is equivalent to that of $G$. Hence, $\delta_G (n) \preccurlyeq \delta_{\Ker(\Phi)}(n)$.
\end{proof}

\bigskip

\bibliographystyle{alpha} 
\bibliography{bibfile}

\end{document}